\documentclass{amsart}

\usepackage[nobysame]{amsrefs}
\usepackage{amsthm}
\usepackage{amssymb}

\usepackage{tikz-cd}
\usepackage{colonequals}
\usepackage{enumerate}
\usepackage{eucal}

\usepackage{hyperref}
\hypersetup{%
  bookmarksnumbered=true,%
  colorlinks=true,%
  linkcolor=blue,%
  citecolor=blue,%
  filecolor=blue,%
  menucolor=blue,%
  urlcolor=blue,%
  bookmarksopen=true,%
  bookmarksdepth=2,%
  pageanchor=true}

\usepackage{longtable}

\makeatletter
\@namedef{subjclassname@2020}{\textup{2020} Mathematics Subject Classification}
\makeatother

\numberwithin{equation}{section}

\theoremstyle{plain}
\newcounter{intro}
\newtheorem{introthm}[intro]{Theorem}

\newtheorem{conj}{Conjecture}

\newtheorem{theorem}[equation]{Theorem}
\newtheorem{proposition}[equation]{Proposition}
\newtheorem{lemma}[equation]{Lemma} 
\newtheorem{corollary}[equation]{Corollary}

\theoremstyle{definition}

\newtheorem{example}[equation]{Example}

\newtheorem{chunk}[equation]{}

\newtheorem{case}[equation]{Case}

\newtheorem*{claim}{Claim}

\makeatletter
\newcommand{\setcasetag}[1]{
	\let\oldtheequation\theequation
	\renewcommand{\theequation}{#1}}

\g@addto@macro\endcase{
	\global\let\theequation\oldtheequation
	\addtocounter{equation}{-1}
}

\makeatother

\theoremstyle{remark}

\newtheorem{remark}[equation]{Remark}

\newcommand{\acat}{\operatorname{C}_{\mathcal{O}}}

\newcommand{\ann}{\operatorname{ann}}
\newcommand{\aqh}[4]{\operatorname{D}_{#1}({#2}/{#3};#4)}

\newcommand{\cat}[1]{\mathsf{#1}}
\newcommand{\con}[1]{\Phi_{#1}}
\newcommand{\Coker}{\operatorname{Coker}}
\newcommand{\cmod}[1]{\Psi_{#1}}
\newcommand{\rmod}{\operatorname{mod}}
\newcommand{\dcat}[1]{\cat D(#1)}
\newcommand{\dbcat}[1]{\cat{D}^{\mathrm{b}}(\operatorname{mod}\, #1)}
\newcommand{\rdcat}[2]{{\cat D}_{#2}(#1)}
\newcommand{\ddual}[1]{\omega_{#1}}
\newcommand{\depth}{\operatorname{depth}}

\newcommand{\edim}{\operatorname{edim}}
\newcommand{\End}{\operatorname{End}}

\newcommand{\Ext}{\operatorname{Ext}}
\newcommand{\tExt}{\operatorname{\rlap{$\smash{\underline{\mathrm{Ext}}}$}\phantom{\mathrm{Ext}}}}
\newcommand{\ecoh}{\operatorname{F}}

\newcommand{\grade}{\operatorname{grade}}
\newcommand{\height}[2][{}]{\operatorname{height}_{#1}{#2}}
\newcommand{\hh}{\operatorname{H}}
\newcommand{\HH}[2]{\operatorname{H}_{#1}(#2)}
\newcommand{\Hom}{\operatorname{Hom}}
\newcommand{\RHom}{\operatorname{RHom}}

\newcommand{\Ker}{\operatorname{Ker}}

\newcommand{\la}{\langle}
\newcommand{\ra}{\rangle}
\newcommand{\length}{\operatorname{length}}
\newcommand{\lotimes}{\otimes^{\operatorname{L}}}
\newcommand{\pos}[1]{[\![{#1}]\!]}
\newcommand{\rank}{\operatorname{rank}}

\newcommand{\Tor}{\operatorname{Tor}}
\newcommand{\tors}{\operatorname{tors}}
\newcommand{\tfree}[1]{{#1}^{\operatorname{tf}}}

\newcommand{\vf}{\varphi}
\newcommand{\ve}{\varepsilon}
\newcommand{\fa}{\mathfrak{a}} 
\newcommand{\fb}{\mathfrak{b}} 
\newcommand{\fm}{\mathfrak{m}} 
 
\newcommand{\fn}{\mathfrak{n}} 
\newcommand{\fp}{\mathfrak{p}}

\newcommand{\uf}{\mathfrak{F}}
\newcommand{\fD}{\mathfrak{D}}

\newcommand{\lra}{\longrightarrow}
\newcommand{\xra}{\xrightarrow}
\newcommand{\bs}{\boldsymbol}
\newcommand{\mco}{\mathcal{O}}

\newcommand{\iso}{\xrightarrow{\raisebox{-.4ex}[0ex][0ex]{$\scriptstyle{\sim}$}}}
\newcommand{\longiso}{\xrightarrow{\ \raisebox{-.4ex}[0ex][0ex]{$\scriptstyle{\sim}$}\ }}

\newcommand{\PatchSys}{\operatorname{PaSy}}
\newcommand{\perfs}{\mathfrak}
\newcommand{\patch}{\mathcal{P}}
\newcommand{\patcht}{\widehat{\patch}}
\newcommand{\base}[1]{#1^{{\scriptscriptstyle\square}}}
\newcommand{\mbb}{\mathbb}

\newcommand{\rhobar}{\overline{\rho}}

\newcommand{\GL}{\mathrm{GL}}
\newcommand{\PGL}{\mathrm{PGL}}
\newcommand{\SL}{\mathrm{SL}}

\DeclareMathOperator{\fl}{fl}
\DeclareMathOperator{\ur}{ur}
\DeclareMathOperator{\loc}{loc}

\DeclareMathOperator{\Nm}{Nm}
\DeclareMathOperator{\tr}{tr}
\DeclareMathOperator{\Frob}{Frob}

\newcommand{\sm}{\smallsetminus}
\newcommand{\es}{\varnothing}

\newcommand{\into}{\hookrightarrow}
\newcommand{\onto}{\twoheadrightarrow}
\newcommand{\isomto}{\xrightarrow{\sim}}

\newcommand{\CNLO}{{\mathrm{CNL}_\mco}}
\newcommand{\CNL}{{\mathrm{CNL}}}
\newcommand{\Sets}{{\mathrm{Sets}}}

\newcommand{\mca}{\mathcal{A}}

\newcommand{\mcd}{\mathcal{D}}

\newcommand{\mcn}{\mathcal{N}}
\newcommand{\X}{\mathcal{X}}

\newcommand{\hecke}[3]{\mathcal{H}\left(#2\backslash#1/#3\right)}

\newcommand{\mcohat}{\widehat{\mco}}
\newcommand{\Zhat}{\widehat{\mathbb Z}}

\newcommand{\Art}{{\mca_\mco}}
\newcommand{\TT}{\mathbb{T}}
\newcommand{\QQ}{\mathbb{Q}}
\newcommand{\ZZ}{\mathbb{Z}}
\newcommand{\CC}{\mathbb{C}}
\newcommand{\RR}{\mathbb{R}}

\newcommand{\A}{\mathbb{A}}
\newcommand{\fullT}{\widetilde{\mathbb T}}
\newcommand{\half}{\mathcal{H}}
\newcommand{\St}{\mathrm{St}}

\newcommand{\disc}{\fD}

\begin{document}

\title[Congruence modules in higher codimensions]{Congruence modules and the Wiles-Lenstra-Diamond numerical criterion  \\ in higher codimensions}

\author[S.~B.~Iyengar]{Srikanth B.~Iyengar}
\address{Department of Mathematics,
University of Utah, Salt Lake City, UT 84112, U.S.A.}
\email{iyengar@math.utah.edu}

\author[C.~B.~Khare]{Chandrashekhar  B. Khare}
\address{Department of Mathematics,
University of California, Los Angeles, CA 90095, U.S.A.}
\email{shekhar@math.ucla.edu}

\author[J.~Manning]{Jeffrey Manning}
\address{Max Planck Institute for Mathematics,
Vivatsgasse 7, 53111 Bonn, Germany}
\email{manning@mpim-bonn.mpg.de }

\date{\today}


\dedicatory{to Andrew Wiles}

\begin{abstract} 

We define a congruence module  $\Psi_A(M)$ associated to  a  surjective $\mco$-algebra morphism $\lambda\colon A \to \mco$, with $\mco$ a discrete valuation ring,  $A$ a complete noetherian local $\mco$-algebra regular at $\mathfrak{p}$, the kernel of $\lambda$, and  $M$ a finitely generated $A$-module.  We establish a numerical criterion for $M$ to have a free direct  summand  over $A$ of positive rank. It is in terms of the lengths of $\Psi_A(M)$ and the torsion part of $\mathfrak{p}/\mathfrak{p}^2$. It generalizes results of Wiles, Lenstra, and Diamond, that deal with the case when the codimension  of $\mathfrak{p}$ is zero. 

Number theoretic applications include integral  (non-minimal) $R=\mathbb T$ theorems in situations of positive defect  conditional on certain standard conjectures. Here $R$ is a deformation ring parametrizing certain Galois representations and $\mathbb T$ is a Hecke algebra.  An example of a positive defect situation is that of proving modularity lifting for 2-dimensional $\ell$-adic  Galois representations over an imaginary quadratic field. The proofs combine our  commutative algebra results with   a generalization  due to Calegari and Geraghty  of the patching method of Wiles and Taylor--Wiles   and level raising arguments that go back to Ribet. The results  provide new evidence in favor of the  intriguing, and  as yet fledgling,   torsion  analog of the classical Langlands correspondence.   

We also prove unconditional integral $R=\mathbb T$  results for   Hecke algebras  $\mathbb T$ acting on weight one cohomology of Shimura curves over $\mathbb Q$.  This leads to a torsion Jacquet--Langlands correspondence comparing integral Hecke algebras acting on weight one cohomology of Shimura curves and modular curves. In this case the cohomology has abundant torsion and so our correspondence cannot be deduced by means of the classical Jacquet--Langlands correspondence.

\end{abstract}

\maketitle

\setcounter{tocdepth}{1}
\tableofcontents

\section{Introduction}
\label{se:intro}

The \emph{raison d'etre} of this paper is to develop new commutative algebra inspired by number theoretic applications. The main contribution of our work is a generalization of the theory of congruence modules and the numerical criteria of Wiles, Lenstra, and Diamond to higher codimension. As a demonstration of the utility of our work, we prove new (integral) $R={\TT}$ theorems at non-minimal level in  certain  positive defect situations that  arise, for instance, when proving modularity of elliptic curves over imaginary quadratic fields.    We also prove unconditional  (integral) $R=\TT$ theorems at non-minimal level   when $\TT$ is a Hecke algebra acting on the weight one cohomology of Shimura curves. This also gives instances of an  analog of the Jacquet-Langlands correspondence that relates torsion in cohomology arising from weight one sheaves on Shimura curves and modular curves. 

Another motivation for our work  is to define congruence modules attached to  augmentations of (big) ordinary deformation rings arising from cohomological cuspidal automorphic representations. These number-theoretic  applications provided  the intuition  that  guided us to our  commutative algebra results. 

To set the stage to present our work, we fix a prime number $\ell$, and $\mco$ the ring of integers of a finite extension of $\mathbb Q_\ell$ with residue field $k$ uniformizer $\varpi$.  We consider the category $\CNL_\mco$ of complete noetherian local $\mco$-algebras with residue field $k$, with morphisms the  maps of local $\mco$-algebras inducing the identity on $k$. 

The patching method introduced  in the work of Wiles \cite{Wiles:1995}  and Taylor--Wiles \cite{Taylor/Wiles:1995}  on modularity lifting theorems, and instrumental in Wiles' proof of Fermat's Theorem, has undergone intense development. In the original approach of Wiles, patching was combined with level raising results, and a numerical criterion \cite[Appendix, Proposition]{Wiles:1995},   extended by Lenstra \cite{Lenstra},  for a surjective map between rings  $R \to {\TT}$, with ${\TT}$ finite flat over $\mco$, to be an isomorphism of complete intersections. In applications in \cite{Wiles:1995},  the ring $R$ is a deformation ring for Galois representations and ${\TT}$ is a Hecke algebra.  Diamond~\cite[Theorem 2.4]{Diamond:1997} developed the results of Wiles and Lenstra by proving a numerical criterion for freeness of $R$-modules $M$ finite flat over $\mco$,  in terms of an augmentation $\lambda\colon R \to \mco$ supported on $M$, that is to say, $M_{\fp} \neq 0$ for  $\fp$ the kernel of $\lambda$. The work of Wiles and Lenstra corresponds to the case when $M$ is a cyclic $A$-module.  Diamond’s criterion is in terms of the cotangent module $\fp/\fp^2$ of $\lambda$, assumed to be a finite abelian group,   and the congruence module  for $M$ at $\lambda$, defined to be: 
\[
\frac{M}{M[\fp] \oplus M[\ann \fp]}.
\]
In \cite[\S 3]{Diamond:1997}, Diamond applies this in the case where $R$ is a  suitable deformation  ring, acting on $M\colonequals \hh_1(X_0(N),\mco)_{\fm}$ with $\fm$ a maximal ideal corresponding to the Galois representation $\rhobar_{\fm}$ whose deformations are parametrized by $R$. The  action of $R$  on $M$ factors through the corresponding Hecke algebra $\TT$. Thus if $M$ is free over $R$, the map $R\to \TT$ is injective, and hence an isomorphism. In this way Diamond recovers the results of \cite{Wiles:1995}.

The patching method has been  extended   by Calegari and Geraghty \cite{Calegari/Geraghty:2018} to situations of {\it positive defect} in which one patches complexes rather than modules: this is the situation one finds oneself   in  when  proving   modularity of elliptic curves over number fields which are not totally real, such as imaginary quadratic fields. 

The positive defect case presents a number of additional complications compared to the situation of Wiles, which have prevented the numerical criterion from being generalized in the same way as the patching method. First, the rings $R$ and ${\TT}$ are no longer expected to be complete intersections in general.  Secondly, they will typically have torsion, in contrast to the defect zero case where $R$ and ${\TT}$ are expected to be finite flat over $\mco$. in the positive defect case. Moreover, the torsion plays an important role in the theory (and a careful understanding of the torsion is necessary for generalizing the patching method). This is an issue as the numerical criterion cannot easily account for torsion in ${\TT}$ or $M$. To make things even more complicated,  $R$ and ${\TT}$ can be entirely torsion which would mean that the map $\lambda\colon R\to \mco$ used in the numerical criterion may not even exist. Lastly, replacing the module $M$ with a complex makes it somewhat unclear how to correctly generalize Diamond's criterion.

The main contribution of our work is to give a generalization of the numerical criterion to the positive defect case. The key idea is to combine the numerical criterion with patching by applying the numerical criterion directly to the patched objects rather than to the rings $R$ and ${\TT}$. This sidesteps the issues raised above, as the patched version, $R_\infty$, of the ring $R$ will automatically be a complete intersection and flat over $\mco$, and the patched version, $M_\infty$, of the complex will be a (maximal Cohen--Macaulay) $R_\infty$-module rather than a complex.

The complication that arises in this approach is that $\Ker(R_\infty\to \mco)$ is no longer be a minimal prime of $R_\infty$, and $M_\infty$ is no longer be finite over $\mco$.  We thus consider a larger families of rings $R$ in $\CNL_\mco$ than those discussed above, placing no restrictions on the height (also known as \emph{codimension}) of $\Ker(R\to \mco)$, and requiring only that $M$ is finitely generated and has sufficient depth. For such $R$ and $M$, we establish a criterion for $M$ to have a non-zero finite free summand and $R$ to be a complete intersection. Applying this to the patched objects $R_\infty$ and $M_\infty$, we show that $M_\infty$ has a free direct summand, which in turn implies that the ring $R$ acts faithfully on the original complex, giving $R={\TT}$. We elaborate on this  later.

The proof of the new numerical criterion takes up the first part of this paper, and involves only commutative algebra.  This is applied in the third part to prove modularity lifting theorems. The second part of the paper establishes a result concerning the compatibility of patching and duality. 

In the remainder of the introduction we present the results in commutative algebra and modularity lifting in greater detail, and explain how  they  fit in the line of development  of modularity lifting theorems pioneered  by Wiles, and continued in the work  of Taylor--Wiles,  Diamond, Kisin, and Calegari--Geraghty.

\subsection*{Commutative algebra results}
Consider  pairs $(A, \lambda)$, with $A$ a complete local noetherian $\mco$-algebra and $\lambda\colon A \to \mco$  a surjective map of $\mco$-algebras   with kernel $\fp_A$; we treat a somewhat larger class of rings in the text. The category consisting of pairs $(A,\lambda)$ for which the local ring  $A_{\fp_A}$ is regular, of dimension $c$, is denoted  $\acat(c)$. The regularity hypothesis is tantamount to the condition that the conormal module $\fp_A/\fp_A^2$ of $\lambda$, viewed as an $\mco$-module, has rank equal to $\dim A_{\fp_A}$, the Krull dimension of $A_{\fp_A}$. In what follows its torsion-submodule 
\[
\con A\colonequals \tors{(\fp_A/\fp_A^2)}\,,
\] 
also plays a key role. Here $\tors(U)$ denotes the torsion-submodule of any finitely generated $\mco$-module $U$.  We also need to consider the torsion-free quotient of $U$, namely $\tfree U \colonequals U/\tors(U)$. 

For any finitely generated $A$-module $M$ the natural map $M \lra  M/\fp_AM$ induces a map 
\[
\tfree{\Ext^c_A(\mco, M)} \lra \tfree{\Ext^c_A(\mco, M/\fp_AM)}
\]
of finitely generated  torsion-free  $\mco$-modules.  We define the \emph{congruence} module of $M$, denoted $\cmod A(M)$, to be the cokernel of this map. When $c=0$ this coincides with Diamond's definition recalled above; see Proposition~\ref{pr:old-and-new}. The definition is reminiscent of the evaluation map in rational homotopy theory and local algebra defined in exactly the same way for the augmentation $A\to k$ to the residue field of $A$, and with $M\colonequals A$. In particular, the localization of the map above (again for $M=A$) at $\fp_A$ is the evaluation map of the local ring $A_{\fp_A}$, and the fact this map is non-zero precisely when $A_{\fp_A}$ is regular---this is a result of Lescot~\cite{Lescot:1983}---is critical to all that follows. It implies  that the $\mco$-module $\cmod A(M)$ is torsion, and hence of finite length.  The \emph{Wiles defect} of $M$ is the integer
\[
\delta_A(M)\colonequals (\rank_{A_{\fp_A}}\! M_{\fp_A}) \length_{\mco}{\con A} - \length_{\mco}{\cmod A(M)}\,.
\]

We view this definition of the congruence module and Wiles defect, which extends  to all $c\ge 0$ the usual definition for $c=0$, as one of the key contributions of this work. We expect these  will play a role in number theoretic applications, like  analyzing  the structure of  deformation rings or eigenvarieties at classical points. One piece of evidence in support of their utility is Theorem~\ref{th:intro-2} below that extends the results of Wiles, Lenstra, and Diamond. Building on the results in this paper in \cite[Theorem~2.6]{Iyengar/Khare/Manning/Urban:2024} we prove any $A$ in $\acat$ is regular if and only if  $\cmod A(A)=0$.

Guided by the intended number theoretic applications, for $A$ in $\acat(c)$, we consider finitely generated $A$-modules $M$ with
\[
  \depth_A M\ge c+1 \quad\text{and}\quad  M_{\fp_A}\ne 0\,.
\]
Some of our results require weaker hypothesis on the depth of $M$; see Section~\ref{se:depth}, and in particular \ref{ch:summary}, for a discussion of these various conditions. By the Auslander-Buchsbaum criterion, the properties above imply that the  $A_{\fp_A}$-module $M_{\fp_A}$ is free.

\begin{introthm}
\label{th:intro-2}
With $A$ and $M$ as above, set $\mu \colonequals \rank_{A_{\fp_A}}(M_{\fp_A})$.  The following statements hold.
\begin{enumerate}[\quad\rm(1)]
\item 
One has $\delta_A(M) \geq 0$, with equality iff $A$ is complete intersection and there is an isomorphism of $A$-modules  $M \cong A^\mu\oplus W$, where  $W_{\fp_A}=0$.
\item If  $\delta_A(M)=0$  and $e_A(M)\le \mu \cdot e(A)$, then $M\cong A^\mu$.
\end{enumerate}
\end{introthm}

The invariant $e_A(M)$ is the multiplicity of $M$; see \cite[\S4.6]{Bruns/Herzog:1998}. The non-negativity of the defect, which follows by a simple computation involving Fitting ideals when $c=0$, is proved in Corollary~\ref{co:tate}; the remaining assertions are contained in Theorems \ref{th:wiles} and \ref{th:diamond}. One way to think about our result is that the vanishing of the Wiles defect allows one to propagate the local property that  $M$ is free at $\fp_A$ of rank $\mu$  to  the  global property that $M$ has a free direct summand $A^\mu$.

As in the proof of \cite[Theorem~2.4]{Diamond:1997} the first step in the proof of Theorem~\ref{th:intro-2} is to deduce that $A$ is complete intersection. This involves a ``defect formula" that relates the defect of $M$ to that of $A$; see Lemma~\ref{le:defect-formula}.
At this stage, one has the following more general criterion for detecting free summands in $M$, without the complete intersection property as a consequence.
 
\begin{introthm}
\label{th:intro-1}
Let $A$ be a Gorenstein local ring in $\acat$ and $M$ a maximal Cohen--Macaulay $A$-module with $\mu\colonequals \rank_{\fp_A}(M_{\fp_A})\ne 0$. One has $\delta_A(M)=  \mu \cdot \delta_A(A)$, if and only if
\[
M\cong A^\mu \oplus W \qquad\text{where $W_{\fp_A}=0$.}
\]
If moreover $e_A(M)\le \mu \cdot e(A)$, then $M$ is free.
\end{introthm}

This is the content of Theorem~\ref{th:gorenstein}. For $A$ and $M$ as above, the condition that $\delta_A(M)=  \mu \cdot \delta_A(A)$ is equivalent to the condition that the Tate cohomology module, in the sense of Buchweitz~\cite{Buchweitz:2021}, of the pair $(\mco,M)$ in degree $c+1$ vanishes; see Proposition~\ref{pr:eta-Gorenstein}. This gives yet another perspective on the Wiles defect of $M$. 

We prove Theorems~\ref{th:intro-2} and \ref{th:intro-1} by a reduction to the case $c=0$. For this a key result, recorded below and contained in Theorem~\ref{th:reduction-delta}, is that the Wiles defect remains invariant on going modulo certain $M$-regular sequences. Here $\fp_A^{(2)}$ denotes the second symbolic power of $\fp_A$.

\begin{introthm}
\label{th:intro-3}
Let $A$ be  in $\acat(c)$ and $M$ a finitely generated $A$-module such that $\depth_AM\ge c+1$. Let $f$ in $\fp_A\setminus \fp_A^{(2)}$ be such that it is not a zero-divisor on $M$. Then the ring $B\colonequals A/fA$ is in  $\acat(c-1)$, and 
\[
\delta_A(M) = \delta_{B}(M/fM)\,.
\]
\end{introthm}

In the passage from $A$ to $B$,  the lengths of both the cotangent module and the congruence module change; the crux is that they change by the same amount. Tracking the change in the cotangent module  is relatively straightforward. The other, more subtle problem, is to control the change in the congruence module of $M$. A key result that permits this is the following structure theorem for the torsion-free quotient of Yoneda Ext-algebra $\Ext^*_{A}(\mco,\mco)$.

\begin{introthm}
\label{th:intro-serre}
For $A$ in $\acat(c)$ there is a natural isomorphism of $\mco$-algebras
\[
\bigwedge_{\mco}\Hom_{\mco}(\fp_A/\fp_A^2,\mco) \cong \tfree{\Ext^*_A(\mco,\mco)} \,.
\]
\end{introthm}

This is the content of Theorem~\ref{th:xi-algebra}. It may be seen as an integral analogue of Serre's result that for a regular local ring $R$, with maximal ideal $\fm$ and residue field $k$, the $k$-algebra $\Ext^*_R(k,k)$ is the exterior algebra on $\Hom_k(\fm/\fm^2,k)$. Implicit in  Theorem~\ref{th:intro-serre} is the fact that $\Ext_A^1(\mco,\mco)$ is the dual of the cotangent module of $\lambda$; this is not difficult to verify. The theorem above implies that the torsion-free quotient of $\Ext^*_A(\mco,\mco)$ is finite and graded-commutative, which is striking because the Ext-algebra $\Ext^*_{A}(\mco,\mco)$ itself is typically infinite and non-commutative. The proof of the theorem above uses techniques that grew out of Tate's work~\cite{Tate:1957}, which introduced methods of differential graded algebras into commutative algebra.

Another noteworthy feature of  congruence modules, is an ``invariance of domain" property, Theorem~\ref{th:invariance-of-domain}, extending \cite[Lemma~3.4]{Brochard/Iyengar/Khare:2021b} which is the case $c=0$. It too is an important ingredient in the proof of Theorems~\ref{th:intro-1} and \ref{th:intro-3}.

\begin{introthm}
\label{th:intro-4}
Let $\vf\colon A\to B$ be a surjective map of local rings in $\acat(c)$. For any finitely generated $B$-module $M$ with $\depth_BM\ge c$, one has 
\[
\cmod A(M)\cong \cmod B(M)\,.
\]
Thus $\delta_A(M)\ge \delta_B(M)$ with equality if and only if $\length_{\mco}\con A =\length_{\mco}\con B$.
\end{introthm}

These results above, especially the invariance of domain property, and the fact that the element $f$ in Theorem~\ref{th:intro-3} need only be $M$-regular and not $A$-regular, convinces us that the definition of Wiles defect we have is the correct one.

\subsection*{Torsion Langlands correspondence and integral $R=\TT$ theorems for Betti cohomology in positive defect}

We restrict ourselves to considering 2-dimensional representations of $G_F$ with $F$ a number field, and assume that the prime $\ell$ does not ramify in $F$. The relevant cohomology groups arise from complexes which are conjectured to be of length $\ell_0$ where $\ell_0=r_2$ is the defect  (see Conjecture \ref{vanishing conj}). Here $[F:\mbb Q]=r_1+2r_2$ with $r_1$ and $r_2$ the number of real and complex places of $F$, respectively. The strategy to use our commutative algebra results  for proving  (integral, non-minimal)  modularity  lifting theorems  in positive defect  is to  apply  them for suitable augmented rings and modules that arise  after carrying out patching.  Thus we prove that in a   patched non-minimal situation a module $M_\infty$  over a ring $R_\infty$ has a free direct summand $R_\infty^\mu$ for a positive integer $\mu$,  and $R_\infty$ is a complete intersection.  

These  applications to modularity lifting  assume Conjectures \ref{rhobar conj}, \ref{R->T conj}, \ref{vanishing conj} and \ref{Ihara conj} listed in Section~\ref{sec:Hecke}, and are similar to those in \cite{Calegari/Geraghty:2018}.  Conjectures \ref{rhobar conj}, \ref{R->T conj}  are about existence of Galois representations attached to Hecke eigenclasses arising from the cohomology of symmetric manifolds  and resulting maps $R \to \TT$.  Conjecture \ref{vanishing conj}  is about  concentration of integral cohomology of these manifolds in a certain range $\ell_0$, and Conjecture \ref{Ihara conj}  is a version of Ihara’s lemma.

Patching methods of \cite{ACC+:2018}   can be  used   in principle to  prove  that a certain patched module $M_\infty$ is maximal Cohen--Macaulay as a module over a patched deformation ring $R_\infty$. Using Taylor's  Ihara avoidance \cite[\S 6.3]{ACC+:2018} one   may in principle  deduce (even when  considering non-minimal deformations, so $R_\infty$ need not be a domain) that $M_\infty$ is faithful as a $R_\infty$-module and $M_\infty[1/\ell]$ is free over $R_\infty[1/\ell]$. After passing to a quotient ring by a sequence that may not be regular when $l_0>0$,  one gets a deformation ring $R$ acting on  the homology module $M$ in lowest degree $q_0=r_1+r_2$ of an arithmetic manifold,  and one may  deduce that $M[1/\ell]$ is  a free $R[1/\ell]$-module. However, by these methods (that rely primarily on patching) it does not seem possible to get that $M$ itself is a faithful $R$-module. 

Our work  is an integral refinement of such results (conditional on the conjectures referred to above)  and when applied  in the situation of 2-dimensional Galois representations over imaginary quadratic fields $K$, and with $X$ a Bianchi 3-manifold, yields that   $M\colonequals \HH 1{X,\mco}_{\fm}$ has a free  direct summand  $R^\mu$ for some $\mu \geq 1$ for a suitable  (possibly non-minimal) deformation ring $R$ that acts on $M$.  Here $\fm$ is the maximal ideal of a Hecke algebra $\TT$ that acts faithfully on $M$ with associated residual Galois representation $\rhobar_{\fm}\colon G_K \to \GL_2(\TT/{\fm})$. (In the statement of Theorem \ref{th:intro-5}  below  the non-minimality of the deformation ring $R_\Sigma$ which acts on  the homology of an arithmetic manifold $Y_0(\mcn_\Sigma)$  via Conjectures   \ref{rhobar conj}, \ref{R->T conj}  is reflected in the fact that the level $\mcn_\Sigma$ might be strictly divisible by the conductor of the residual representation $\rhobar_{\fm}$.) The  resulting integral  $R=\TT$ theorems that follow  produce new evidence towards an  emerging  \emph{torsion} Langlands correspondence. This addresses torsion  Galois representations $\rho\colon G_K \to \GL_2(\mco/\omega^n)$   that may not lift to characteristic 0. Our results show that such a $\rho$ is automorphic in the sense that it  arises from specialising the universal modular Galois deformation $\rho_\TT\colon G_K \to \GL_2(\TT)$ lifting $\rhobar_{\fm}$, arising  thus from torsion Hecke eigenclasses in  $\HH{1}{X,\mco}_{\fm}$.

With $\rhobar_\fm$, $\Sigma$ and $\mcn_\Sigma$ as in Section \ref{sec:Hecke}, our main number theory application, Theorem \ref{R=T theorem},  is the following result. In the main text we define $Y_0(\mcn_\Sigma)$ to have additional level structure at an auxiliary prime $t\nmid \mcn_\Sigma$, in order to make the subgroups neat.

\begin{introthm}
\label{th:intro-5}
Assume Conjectures \ref{rhobar conj}, \ref{R->T conj}, \ref{vanishing conj} and \ref{Ihara conj} hold for $F$ and assume $\ell>3$. Let $\fm$ be a non-Eisenstein maximal ideal of ${\TT}(K_0(\mcn_\es))$ such that $N(\rhobar_\fm) = \mcn_\es$ and $\rhobar_\fm|_{G_{F(\zeta_\ell)}}$ is absolutely irreducible. There exists an integer $\mu\ge 1$ such that for all $\Sigma$  there is an isomorphism of $R_\Sigma$-modules
\[
\hh_{r_1+r_2}(Y_0(\mcn_\Sigma),\mco)_{\fm_\Sigma}\cong R_\Sigma^\mu \oplus W_\Sigma
\]
for some $R_\Sigma$-module $W_\Sigma$. In particular, $R_\Sigma$ acts faithfully on $\hh_{r_1+r_2}(Y_0(\mcn_\Sigma),\mco)_{\fm_\Sigma}$ and so the map $R_\Sigma\onto {\TT}_\Sigma$ from Conjecture \ref{R->T conj} is an isomorphism for all $\Sigma$.
\end{introthm}

In analogy with \cite[Theorem 3.4]{Diamond:1997}, one might expect our method to produce the stronger result that $\hh_{r_1+r_2}(Y_0(\mcn_\Sigma),\mco)_{\fm_\Sigma}\cong R_\Sigma^\mu$ (that is,  one would expect that $W_\Sigma =0$; if further $F$ is totally complex one would expect $\mu=1$). The reason for our slightly weaker result is that one cannot control the generic rank of the patched module $M_\infty$ unless the ring $R$ has enough characteristic $0$ points (in order to appeal to classical multiplicity one results for automorphic forms), which is not guaranteed in the positive defect case. Fortunately the weaker statement above is still enough to deduce $R={\TT}$. Building on these techniques, in subsequent work~\cite{Iyengar/Khare/Manning:2022b},  we prove that  $\hh_{r_1+r_2}(Y_0(\mcn_\Sigma),\mco)_{\fm_\Sigma}\cong R_\Sigma$  when $F$ has no real places and there are geometric characteristic 0 lifts of $\rhobar_\fm$.

When $F$ is a CM field and $\ell$ is sufficiently large\footnote{The precise bound is given in \cite[(1.2)]{NewtonThorne}, in  our situation, $\ell>20$ is sufficient if $[F:\QQ] = 2$.}, Conjecture \ref{rhobar conj} is known by \cite[Theorem 1.3]{NewtonThorne}. Moreover if $F$ is an imaginary quadratic field, then Conjectures \ref{vanishing conj} and \ref{Ihara conj}   are known as well (see the discussions after their  statements), and so Theorem \ref{th:intro-5} gives:

\begin{introthm}
	\label{th:intro-5 QI}
	Let $F$ be a  imaginary quadratic field. Assume Conjecture \ref{R->T conj} holds for $F$.  Let $\fm$ be a non-Eisenstein maximal ideal of ${\TT}(K_0(\mcn_\es))$ such that $N(\rhobar_\fm) = \mcn_\es$ and $\rhobar_\fm|_{G_{F(\zeta_\ell)}}$ is absolutely irreducible.
For $\ell>20$,  there exists an integer $\mu\ge 1$ such that for all $\Sigma$  there is an isomorphism of $R_\Sigma$-modules
	\[
	\hh_{1}(Y_0(\mcn_\Sigma),\mco)_{\fm_\Sigma}\cong R_\Sigma^\mu \oplus W_\Sigma
	\]
	for some $R_\Sigma$-module $W_\Sigma$. In particular, $R_\Sigma$ acts faithfully on $\hh_{1}(Y_0(\mcn_\Sigma),\mco)_{\fm_\Sigma}$ and so the map $R_\Sigma\onto {\TT}_\Sigma$ from Conjecture \ref{R->T conj} is an isomorphism for all $\Sigma$.
\end{introthm}

When $F$ is CM weaker versions of Conjecture \ref{R->T conj} were proven in \cite{ACC+:2018} (which primarily replaced $\TT_\Sigma$ with $\TT_\Sigma/I$ for a nilpotent ideal $I$, and also placed some restrictions on $F$). It is likely that a refinement of these methods could prove Conjecture \ref{R->T conj} for all CM fields, which would make Theorem \ref{th:intro-5 QI} hold unconditionally.

\subsection*{Torsion Jacquet--Langlands correspondence for weight one coherent cohomology of modular and Shimura curves}
We also have the following unconditional applications to integral modularity lifting of our  new results in commutative algebra. These applications arise when we consider coherent cohomology of the weight one line bundle $\omega_\mco$ on modular and Shimura curves.

We refer to \S \ref{sec:Hecke wt1} for the  unexplained notation used in the statement below, which is Theorem \ref{weightone:thm} of the main text.  In the case of modular curves (when $\disc=\es$) such a result was proven in \cite{Calegari} using $q$-expansion principle. One crucial place where $q$-expansions are used in   \cite{Calegari}  is to show  that, in the  relevant non-minimal patched situation, not only is $M_\infty$ faithful as a $R_\infty$-module, but that it is free over the (patched) Hecke algebra, and hence also free  as a $R_\infty$-module. This  argument does not work for Shimura curves for the compelling reason that these have no cusps. Our methods on the other hand do carry over to Shimura curves.
 
\begin{introthm}\label{th:intro-5b}
Let $X^\disc_\Sigma$ be a Shimura curve arising from an indefinite quaternion algebra $D_\disc$ over $\QQ$ that ramifies only at a finite set of primes $\disc$ and assume $\ell>3$. Let $\TT^\disc_\Sigma$ be the Hecke algebra acting on $\hh^1(X^\disc_\Sigma,\omega_\mco)_{\fm}$ and $R^\disc_\Sigma$ the corresponding deformation ring which parametrizes deformations of $\rhobar_\fm$ that are  in particular unramified at $\ell$.  The surjective map   $R^\disc_\Sigma \onto \TT^\disc_\Sigma$ is an isomorphism.
\end{introthm}

From this we deduce a Jacquet--Langlands correspondence  comparing the  degree 1  integral cohomology with coefficients in  the ``weight one’’ sheaf $\omega_\mco$   of modular and Shimura curves. As these cohomology groups typically have abundant torsion this  cannot be deduced from the classical Jacquet--Langlands correspondence and, as far as we know, this is  the first such correspondence in the literature relating torsion in  the coherent cohomology of   Shimura varieties that arise from inner forms. The possibility of such a result had been considered earlier  by  Boxer, Calegari and Gee. The following result is Theorem \ref{JLcorr} of the main text.

\begin{introthm}
\label{th:intro-5c}
Let $\disc\ne\es$ and take $\ell>3$.  Consider  a residual representation $\rhobar: G_\QQ 
	\to \GL_2(k)$ that  arises from a maximal ideal $\fm_\disc$ of the Hecke algebra acting on $\hh^1(X^\disc_\Sigma,\omega_\mco)$.  We assume that $\rhobar|_{\QQ(\zeta_\ell)}$ is irreducible. Then:
\begin{enumerate}[\quad\rm(1)]	
\item $\rhobar$ arises from a maximal ideal $\fm_\es$ of the Hecke algebra acting on $\hh^1(X^\es_{\Sigma\cup\disc},\omega_\mco)$.
\item If $\TT^\es_{ \Sigma\cup\disc}$ and   $\TT^\disc_\Sigma$ are the Hecke algebras acting on $\hh^1(X^\es_{\Sigma\cup\disc},\omega_\mco)$ and $\hh^1(X^\disc_{\Sigma},\omega_\mco)$, localized at $\fm_\es$ and $\fm_\disc$ respectively, then there is a natural surjective map $\TT_{\Sigma\cup\disc}^{\es}  \onto \TT_\Sigma^{\disc}$ with kernel generated by $q_vU_v^2 -\psi(\Frob_v)$ for $v \in \disc$.
\end{enumerate}
\end{introthm}

 \subsection*{Wiles defects of  cohomology groups of modular curves  and Shimura curves}
The deformation invariance proved in Theorem \ref{th:intro-3}  has applications to reproving and generalizing, by a different method, the results of \cite{Bockle/Khare/Manning:2021a}. Specifically an easy corollary of Theorem \ref{th:intro-3} is the following:

\begin{introthm}
\label{th:intro-6}
Let $R_\infty$ be a complete local noetherian $\mco$-algebra and  $\lambda\colon R_\infty\to\mco$ a surjective map with kernel $\fp$, such that $(R_\infty)_{\fp}$ is regular of dimension $c$. Let $M_\infty$ be an $R_\infty$-module with $\depth_{R_\infty}M_\infty \ge c+1$ and $(M_\infty)_{\fp}\ne 0$.

Let $(\bs y,\varpi)$ be an $M_\infty$-regular sequence, where ${\bs y} = \{y_1,\ldots,y_c\}\subseteq \fp$ and let $R_{\bs y} = R_\infty/(\bs y)$, $M_{\bs y} = M_\infty/(\bs y)$ and $\fp_{\bs y} = \fp/(\bs y)\subseteq R_{\bs y}$, so  $M_{\bs y}$ is flat over $\mco$.

Provided that $\fp_{\bs y}/\fp^2_{\bs y}$ has finite length, the Wiles defect $\delta_{R_{\bs y}}(M_{\bs y})$ depends only on $R_\infty$, $M_\infty$ and $\lambda\colon R_\infty\to\mco$, and not on the choice of $\bs y$.
\end{introthm}

When $M_\infty = R_\infty$ (and hence $R_\infty$ is Cohen--Macaulay) this recovers \cite[Theorem 3.25]{Bockle/Khare/Manning:2021a}. Theorem \ref{th:intro-6} represents a significant generalization of this result, as it does not require that $R_\infty$ is Cohen--Macaulay.

Removing the requirement that $R_\infty$ is Cohen--Macaulay is particularly significant for number theoretic applications. In the context considered in \cite{Bockle/Khare/Manning:2021a}, the ring $R_\infty$ would be a power series ring over a completed tensor product of certain local Galois deformation rings, $M_\infty$ would be a `patched' module arising from the cohomology of certain Shimura curves and the augmentation $\lambda\colon R_\infty\to\mco$ would arise from a certain cuspidal Hilbert modular form (this is still a special case, the same commutative algebra setup would occur in any patching situation with defect $\ell_0=0$). Crucially, while the ring $R_\infty$ is only known to be Cohen--Macaulay in some rather restrictive cases, the module $M_\infty$ would automatically satisfy $\depth_{R_\infty}M_\infty = c+1$ in all cases, as a consequence of the patching argument.

The primary objects of number theoretic interest would be a Galois deformation ring $R$ and a cohomology module $M$, which would have the form $R = R_{\bs y}$ and $M = M_{\bs y}$ for some sequence ${\bs y}$ as above. Theorem \ref{th:intro-6} would thus imply that $\delta_R(M)$ (which was referred to as the \emph{cohomological Wiles defect} in \cite{Bockle/Khare/Manning:2021b} and \cite{Bockle/Khare/Manning:2021a}) depends only $R_\infty$, $M_\infty$ and $\lambda\colon R_\infty\to\mco$, and not on $R$ and $M$ directly.

As in \cite{Bockle/Khare/Manning:2021a}, both $R_\infty$ and $\lambda\colon R_\infty\to\mco$ are determined entirely by local Galois theoretic information. The module $M_\infty$ is not always known to be determined by local information, but it is reasonable to expect that it is. In fact it is plausible that one can prove that it is in some circumstances where $R_\infty$ cannot be proven to be Cohen--Macaulay. This would then imply that the cohomological Wiles defect, $\delta_R(M)$, is determined entirely by local information without requiring the assumption that $R_\infty$ is Cohen--Macaulay, allowing one to generalize \cite[Theorem 6.5]{Bockle/Khare/Manning:2021a}.

\subsection*{Links between the commutative algebra and modularity lifting}
 We outline where our results  fit in  the evolution of modularity lifting results.
 
The original papers of Wiles and Taylor--Wiles, \cite{Wiles:1995} and \cite{Taylor/Wiles:1995}, deduced modularity lifting from an $R={\TT}$ result which they first proved in the minimal level case via a patching argument, and then proved in non-minimal levels via an induction argument using the numerical criterion. The innovation of Diamond in \cite{Diamond:1997} reformulated the patching and numerical criterion for modules rather than rings, which allowed the argument to prove freeness results for certain cohomology groups (rather than requiring these results as inputs to the argument).

Kisin's work in \cite{Kisin} gave the patching method far greater flexibility and applicability by patching global deformation rings and Hecke  modules relative to local deformation rings. The work of Kisin almost erases the distinction between minimal and non-minimal cases, and removes the need for the numerical criterion. This comes at the expense of only proving $R[1/\ell] = {\TT}[1/\ell]$, rather than $R={\TT}$, which does not matter for proving modularity of $\ell$-adic Galois representations.
 
These works were in the case when the (global) defect $\ell_0$ is 0. The  insight of Calegari and Geraghty in  \cite{Calegari/Geraghty:2018} generalized the patching method to positive defect (subject to certain conjectures, which we must also assume for our work). Following the approach of Kisin, this gives modularity lifting results for $\ell$-adic representations in both minimal and non-minimal levels. However in non-minimal levels it again can only prove that $R[1/\ell] = {\TT}[1/\ell]$, or even only that $R^{\mathrm{red}} = {\TT}^{\mathrm{red}}$. In the $\ell_0>0$ case one typically expects $R$ and ${\TT}$ to have torsion, and so not having an approach to proving $R={\TT}$ is a noticeable shortcoming to the patching method. 

Our work generalizes the numerical criterion to positive defect, and thus gives us $R={\TT}$ results in non-minimal level by an inductive argument similar to the one in \cite{Diamond:1997}. Unlike in the $\ell_0=0$ case, when the patching and numerical criterion arguments were treated fairly independently of each other, when $\ell_0>0$ we are compelled to work at a patched level to apply our numerical criterion. This is one of the key insights of this paper: patching is a necessary prelude to using our numerical criterion for applications to modularity lifting in positive defect.

In the positive defect case patching is much more intrinsic to deducing  non-minimal $R={\TT}$ theorems from their minimal analogs than in the defect 0 case. In both the defect $0$ case and the positive defect case the augmentation $\lambda\colon R_\infty \to \mco$ is crucial for formulating the numerical criterion. In the positive defect case the augmentation which we consider does not need to factor though the global deformation ring $R_\infty\onto R$, and indeed may not have any relation to Galois representations or modular forms. This is necessary, as $R$  may not have any characteristic 0 points if the residual representation $\rhobar_{\fm}$ has no appropriate characteristic 0 lifts, and so the augmentation $\lambda$ may only exist after patching.

As an important step in our proof of Theorem \ref{th:intro-5} we have to  generalize the  classical change of congruence modules arguments of Ribet \cite{RibCong}  (in codimension $c=0$ setting) to the present higher ($c>0$) codimension  setting to deduce that  the Wiles defect  $\delta_{R_\infty’}(M_\infty’)$ vanishes, for patched rings $R_\infty’$  and modules $M_\infty’$ in non-minimal situations,  from  $\delta_{R_\infty}(M_\infty)=0$, for patched rings $R_\infty$  and modules $M_\infty$ in minimal situations. For this one uses the following results  in concert:
\begin{enumerate}
   \item patching in Theorem \ref{th:ultrapatching};
     \item Theorem \ref{th:duality}  which gives a self-duality statement for a patched module that arises from patching complexes with self-duality properties;
   \item a version of Ihara’s lemma for surjectivity of maps between patched modules of different levels  in \S \ref{sec:Ihara};
   \item   change of congruence module arguments for our congruence modules in higher codimensions which  is inspired by \cite{RibCong} and uses in addition the commutative algebra  results Lemma \ref{le:change-modules}  and Proposition \ref{pr:change-modules}, as well as Theorem \ref{th:invariance-of-domain}.
\end{enumerate}

\subsection*{Earlier work}
As noted before, our work is inspired by work of Wiles, Lenstra, and Diamond.  The result of Wiles and Lenstra has been extended by Huneke and Ulrich~\cite{Huneke/Ulrich:1996}, and Zarzuela~\cite{Zarzuela:1996}, to cover more general surjective maps of rings $A\to B$, where $B$ need not be a discrete valuation ring.  The direct  precursors  to  the  commutative algebra developed in this paper are  \cite{FKR},   \cite{Bockle/Khare/Manning:2021b}, \cite{Bockle/Khare/Manning:2021a} (and its appendix),   \cite{Brochard/Iyengar/Khare:2021b}, as well as the work of Tilouine and Urban in \cite{TiUr}.    We describe briefly the previous work, to provide historical  context. 

The results of \cite[Appendix]{FKR} and \cite{Brochard/Iyengar/Khare:2021b}  generalized results of Wiles, Lenstra and Diamond   in \cites{Wiles:1995, Lenstra, Diamond:1997} so that rather than asking that the  ring $A$ or the module $M$ be finite and  flat over $\mco$, it is only required that  $M$ has positive depth. Theorem \ref{th:intro-2} in this paper  also requires that $M$ have sufficient depth.

The papers   \cites{Bockle/Khare/Manning:2021b, Bockle/Khare/Manning:2021a, TiUr}, introduced the  notion of Wiles defect of rings, and showed that it has arithmetic significance. The appendix  to  \cite{Bockle/Khare/Manning:2021a}  (written by N. Fakhruddin and the second  author) established a formula for the Wiles defect proposed  by Venkatesh for $\mco$-algebras $A\in \acat(0)$ with $\dim A=1$. This involves the Andr\'e--Quillen cohomology module $\mathrm{D}^1(A/\mco,E/\mco)$, where $E$ is the fraction  field of $\mco$, and another invariant that  he  defined using suitable complete intersections that map onto $A$.  The main article   \cite{Bockle/Khare/Manning:2021a}  extended  the definition of Venkatesh  invariants, and Wiles defect, to higher dimensional rings $A$ that are Cohen--Macaulay and flat over $\mco$ and with an augmentation $\lambda\colon A \to \mco$ such that $A$ is smooth at $\lambda$, basically by proving that the invariants  remain unchanged on going modulo regular sequences, so that one can reduce to dimension one.

The paper \cite{Brochard/Iyengar/Khare:2021b} considered Wiles defect for $A \in \acat(0)$, and $M$ a finitely generated $A$-module. One of the main observations of that paper, \cite[Theorem 1.2]{Brochard/Iyengar/Khare:2021b},  is a formula  relating the Wiles defects of the module $M$ and ring $A$. This is generalized in our paper Lemma \ref{le:defect-formula}  and plays an important role in the proof of Theorem \ref{th:intro-1}. 

The present paper defines the Wiles defect for all rings $A$ in $\acat (c)$ for arbitrary $c$. This uses  invariants $\con A$ and $\cmod A$ that are more direct generalizations of the invariants of  Wiles; their definition does not use Andr\'e--Quillen homology of rings. 
They do not remain invariant individually on going modulo regular sequences, but the Wiles defect $\delta_A$ does; see  Theorem~\ref{th:intro-3} and the discussion around it. Our invariants are defined in a way that is more intrinsic to $A$ than the ones introduced by Venkatesh which use as a crutch complete intersections mapping to $A$. 
 
\subsection*{Further applications}

In a sequel \cite{Iyengar/Khare/Manning/Urban:2024}  to this work,    we have used  the  techniques  of this paper  to relate our work on congruence modules in positive codimension to ``zeta-elements’’   (see  \cite{Urban:2021}) for the adjoint motive of a newform $f$ of weight at least 2  which is ordinary at $p$. (For this paragraph we denote the  prime $\ell$ of the rest of the paper by $p$.)   This is contained in  \cite[Proposition 2.5, Theorem 3.7]{Iyengar/Khare/Manning/Urban:2024}  that we briefly describe.  The  $p$-stablization  of  $f$ induces  an augmentation  $\lambda_f : \TT^{\rm ord} \to \mco$   of Hida’s ordinary  $p$-adic Hecke algebra.  By Hida theory  ${\TT}^{\rm ord}$  is finite  flat over the Iwasawa algebra $\Lambda=\mco\pos{T}$ (and hence $\Ker \lambda_f$ has height  $1$),  and it is smooth at $\lambda_f$.  Let ${\rm Ad}_f$ be the adjoint motive of $f$, $L(s,{\rm Ad}_f)$  the corresponding degree 3 $L$-function and $\rho_f:G_\QQ \to \GL_2(\mco)$ the Galois representaiton associated to a $p$-adic place of the Hecke field $K_f$ of $f$. Our construction of the congruence module $\Psi_{\lambda_f}(\TT^{\rm ord})$ leads to ``zeta elements’’ in Galois cohomology $H^1_{\rm ord}(\mathbb Q,{\rm Ad} \rho_f)$ whose image in  quotients of the corresponding local cohomology groups at $p$ are related   to the congruence module $\Psi_{\lambda_f}(\TT)$. Here ${\TT}$ is a classical  Hecke algebra of fixed level that is finite flat over $\mco$ and $\Psi_{\lambda_f}(\TT)$ is its classical (codimension 0) congruence module.  Hida's work in \cite{Hida}  relates the length of the congruence module $\cmod {\TT}$ to an $L$-value  $L(1,{\rm Ad}_f)$.  

In   \cite{Bockle/Khare/Manning:2021a} we study  situations when we have $R=\TT$ without the rings being complete intersections and determine the Wiles defect $\delta_{\lambda_f}(R)$  at $\lambda_f: R \to \mco$;  we show it  is  given by  $\sum_{q} \delta_{\lambda_f}(R_q)$  of  Wiles defects of the corresponding local deformation rings $R_q$. The work of  \cite{Bockle/Khare/Manning:2021a} and the present paper lead to the intriguing possibility of proving the Bloch-Kato conjecture  for  $L(1,{\rm Ad}_f)$ in cases which will go beyond  the known results in \cite{DFG}  that are proved  when the  map $R \to \TT$ is an isomorphism of complete intersections. It also raises the intriguing possibility that $\length_\mco \Psi_{\lambda_f}(R_q)$ is related to ${\rm ord}_\varpi c_q$ where $c_q$ is the local Tamagawa number at $q$  which Bloch-Kato associate to the motive ${\rm Ad}_f$. This would be a local analog of the results in   \cite{Iyengar/Khare/Manning/Urban:2024}; we will pursue this in future work.

 The residual representations $\rhobar_{\fm}$ and non-minimal levels $\mcn_\Sigma$   that we allow in Theorem~\ref{th:intro-5} are restrictive: for instance  we assume  $q_v\not\equiv 1\pmod{\ell}$ if $v$ divides $\mcn_\Sigma$ and $\rhobar_\fm$ is unramified at $v$.  In forthcoming work  with Fred Diamond we will extend the  methods of this paper to prove a version  of Theorem \ref{th:intro-5} without such restrictions: the level raising arguments in Part \ref{part:deformation rings} become more complicated in the general case.   We will also prove a more general version of Theorem \ref{th:intro-5c} which will   include  proving under the necessary assumptions   the existence of  maximal ideals of  the Hecke algebra  acting on  $\hh^1(X^\disc_{\Sigma},\omega_\mco)$    that correspond   to maximal ideals $\fm_\es$ of the Hecke algebra acting on $\hh^1(X^\es_{\Sigma\cup\disc},\omega_\mco)$(see Remark \ref{JLremark} below); namely the converse to Theorem \ref{th:intro-5c} (1).

\subsection*{Comparison to the literature}

The question of what might be the right generalization of Wiles numerical criterion to positive defect has been considered before us. For example in \cite{Calegari} Calegari  proves some non-minimal modularity lifting theorems in the case of weight one forms. There he combines known multiplicity one theorems for the relevant Hecke modules with the patching argument from \cite{Calegari/Geraghty:2018}. He remarks in the introduction to the paper that in the weight one case he considers,  as the relevant Hecke rings  ${\TT}^1$ are not complete intersections in general, and  may  not be flat over  $\mco$, the arguments  of Wiles in \cite[Chapter 2]{Wiles:1995}  seem not to be applicable. The work of this paper gives  a work around such an obstruction. Our methods  apply  unconditionally, to  prove the non-minimal $R={\TT}$ theorems of \cite{Calegari}   without a priori using any multiplicity one theorems. As explained above we use our methods to prove similar results in the setting of Shimura curves (Theorems \ref{th:intro-5b} and \ref{th:intro-5c} above): the methods of \cite{Calegari} do not extend to this case.

The work of Calegari and Venkatesh in \cite{CV} considers level raising in the positive defect case---for example in the Bianchi case---and remarks that a form of Ihara’s lemma is true for homology in degree one of Bianchi manifolds. In our work we apply Ihara’s lemma at a patched level: as the patched complexes are then quasi-isomorphic to a self-dual module, Ihara’s lemma  gives that the growth of the congruence module we define is given by an Euler factor.

\part{Commutative algebra of congruence modules}\label{part:ca}

In this part we introduce congruence modules for a suitable class of local rings, and establish some of their fundamental properties. Although the focus is on the pure commutative algebraic aspects of congruence modules, the results we prove, and even their formulations, are dictated by applications to number theory, and indeed our intuition about congruence modules comes from that realm. Nevertheless, many of the statements are interesting in their own right, so with an eye towards possible applications, also to commutative algebra, the setting is more general and the development is more in-depth than is strictly needed for the number theory applications presented in this work.
 
\section{Congruence modules and cotangent modules}\label{se:congruence module}

Throughout  $\mco$ is a discrete valuation ring, with uniformizer $\varpi$. 

\begin{chunk}
\label{ch:tors-remark}
We write $\tors(U)$ for the torsion submodule of an $\mco$-module $U$. The \emph{torsion-free quotient} of $U$, denoted $\tfree U$,  is the cokernel of the inclusion $\tors(U)\subseteq U$. There is thus an exact sequence of $\mco$-modules
\[
0\lra \tors(U) \lra U \lra \tfree{U} \lra 0\,,
\]
functorial in $U$. When $U$ is finitely generated, so is $\tfree U$ and hence it is a free $\mco$-module. In particular, the sequence above splits. Though there is no  functorial splitting, the existence of a splitting does mean that for any additive functor, say $E(-)$, on finitely generated $\mco$-modules, the induced map 
\[
E(U) \lra E(\tfree{U})
\]
remains surjective. This remark will be used often, as will be the following ones: Consider  a map of $\mco$-modules
\[
\alpha \colon U\lra V
\]
and  the  map $\tfree \alpha\colon \tfree{U}\to \tfree{V}$. The induced map $\Coker\alpha \to \Coker (\tfree{\alpha})$ is surjective;  in particular, when $\alpha$ is onto, so is $\tfree{\alpha}$.  If $\Ker \alpha$ is torsion, then the induced map $\tfree{\alpha}$ is injective.  
\end{chunk}

\begin{chunk}
\label{ch:congruence-module}
Throughout this manuscript  $A$ is a noetherian local ring and $\lambda_A\colon A\to \mco$ a surjective local homomorphism. We set
\[
\fp_A\colonequals \Ker(\lambda_A)\qquad \text{and}\qquad c \colonequals \height \fp_A\,.
\]
Thus $\fp_A$ is a prime ideal in $A$ and $c$ is its height, also known as the codimension, whence the choice of notation. It is also the Krull dimension of the local ring $A_{\fp_A}$. Our focus will be on rings for which  $c$ is the rank of the cotangent module of $\lambda_A$; see Lemma~\ref{le:con-rank}. One has that
\[
\dim A \ge \height \fp_A + \dim(A/\fp_A) = c+1\,;
\]
we place no restrictions on $\dim A$.
\end{chunk}

\subsection*{Torsion-free quotients of Ext}
\label{pg:ecoh}
We write $\dcat A$ for the derived category of all $A$-modules; the suspension (also known as shift) functor is denoted $[-]$. As usual, a module is viewed as complex concentrated in degree $0$.   Although the focus of our work is on finitely generated modules, some of the basic definitions extend, and it is helpful in some arguments to do so, to complexes with finite homology. We write $\rmod A$ for the category of finitely generated $A$-modules.  Let $\dbcat A$ denote the full subcategory of $\dcat A$ consisting of $A$-complexes $M$ with $\HH iM$ finitely generated for each $i$ and equal to zero for $|i|\gg 0$. For such  $M$ the $\mco$-modules $\Tor^A_i(\mco,M)$ and $\Ext^i_A(\mco,M)$ are finitely generated  for each integer $i$, and equal to $0$ for $i\ll 0$.
For $M$ in $\dbcat A$ consider the $\mco$-module 
\[
\ecoh^i_A(M) \colonequals \tfree{\Ext^i_A(\mco,M)}\,,
\]
namely, the torsion-free quotient of $\Ext^i_A(\mco, M)$. Evidently $\ecoh^n_A(-)$ is a functor from $\dbcat A$, the derived category of $A$-modules, to $\mco$-modules.

\subsection*{Congruence modules}
\label{pg:cmod}
Let $M$ be a finitely generated $A$-module. The map
\[
\lambda_A(M)\colon M \lra \mco \otimes_A M = M/\fp_AM
\]
induces a map of $\mco$-modules
\[
\ecoh^c_A(\lambda_A(M)) \colon   \ecoh^c_A(M) \lra \ecoh^c_A(M/\fp_AM)\,;
\]
recall $c=\height\fp_A$.  The cokernel of this map is the \emph{congruence} module of $M$:
\[
\cmod A(M) \colonequals \Coker \ecoh^c_A(\lambda_A(M))\,.
\]
For a given $A$ there may be many choices of an augmentation $\lambda\colon A\to\mco$ and for some applications it is important to keep track of this dependency, and then it would be better to write $\cmod{\lambda}(M)$ for the congruence module. However, in this work the map $\lambda$ is fixed, so we stick to the notation above.

We write $\cmod A$, instead of $\cmod A(A)$, for the congruence module of $A$. See the discussion around Proposition~\ref{pr:old-and-new} and  \ref{ch:BKM-defn} for connections to other notions of congruence modules  in the literature. It is easy to check that the assignment $M\mapsto \cmod A(M)$ is a functor from $\rmod A$ to $\rmod \mco$; see \ref{ch:change-modules}.

Here is a simple but useful observation.

\begin{chunk}
For any finitely generated $\mco$-module $U$, viewed as an $A$-module via $\lambda_A$, the surjection $U\to \tfree U$ is split with torsion kernel as an $\mco$-map, and hence also as an $A$-map---see \ref{ch:tors-remark}---and so it induces the first isomorphism below: 
\[
\ecoh^c_A(U)\xra{\ \cong\ }\ecoh^c_A(\tfree U) \cong \ecoh^c_A(\mco)\otimes_\mco \tfree U\,.
\]
The second one holds because the $\mco$-module $\tfree U$ is finite free. In particular, for any finitely generated $A$-module $M$, there is a natural isomorphism:
\begin{equation}
\label{eq:ecoh-MO}
\ecoh^c_A(M/\fp_AM)\cong \ecoh^c_A(\mco)\otimes_\mco \tfree {(M/\fp_AM)}
\end{equation}
\end{chunk}

\subsection*{Regularity}\label{sse:regularity}
Given a prime ideal $\fp$ in  $A$, we say $A$ is \emph{regular at $\fp$} to mean that the local ring $A_\fp$ is regular.  The result below is the starting point of our work. 

\begin{theorem}
\label{th:catC2}
In the context of \ref{ch:congruence-module},  the ring $A$ is regular at $\fp_A$ if and only if the $\mco$-module $\cmod A$ is torsion. When this holds, so do the following statements:
\begin{enumerate}[\quad\rm(1)]
\item
The $\mco$-modules $\ecoh^c_A(A)$ and $\ecoh^c_A(\mco)$ are of rank one;
\item
The map $\ecoh^c_A(\lambda_A) \colon \ecoh^c_A(A) \to \ecoh^c_A(\mco)$ is injective;
\item
The $\mco$-module $\cmod A$ is torsion and cyclic;
\end{enumerate}
\end{theorem}

\begin{proof}
Write $\fp$ instead of $\fp_A$ and let $k(\fp)$ be the field of fractions of $\mco$; it is also the residue field of the local ring $A_\fp$. For any integer $i$ there is an isomorphism 
\[
\Ext_A^i(\mco,\mco)_\fp\cong \Ext_{A_\fp}^i(k(\fp),k(\fp)) 
\]
of $k(\fp)$-vector spaces. It is well-known that the vector space on the right is nonzero for each $0\le i\le c$, hence the $\mco$-module $\Ext_A^i(\mco,\mco)$ has nonzero rank for the same range of $i$; see, for instance, \cite[Corollary~1.3.2, Theorem~1.3.3]{Bruns/Herzog:1998}. This observation will be used in what follows. 

It helps to introduce one more property equivalent to the regularity of $A$ at $\fp$, and prove that the following conditions are equivalent:
\begin{enumerate}[\quad\rm(i)]
\item
$A$ is regular at $\fp$;
\item
The rank of the map $\Ext^c_A(\mco,\lambda_A)$ is nonzero;
\item
$\cmod A$ is torsion as an $\mco$-module.
\end{enumerate}

(i)$\Leftrightarrow$(ii): Condition (ii) holds if and only if the localization of the map $\Ext^c_A(\mco,\lambda_A)$ at $\fp$ is nonzero; namely the map
\[
\Ext^c_{A_{\fp}}(k(\fp),A_{\fp}) \lra \Ext^c_{A_{\fp}}(k(\fp),k(\fp))
\]
induced by the surjection $A_\fp\to k(\fp)$  is nonzero. It remains to recall Lescot's result~\cite[1.4]{Lescot:1983} that the map above is nonzero if and only if the local ring $A_\fp$ is regular; see also \cite[Theorem~2.4]{Avramov/Iyengar:2013}. 

(i)$\Rightarrow$(iii): Since $A_\fp$ is regular of Krull dimension $c$ one has 
\begin{align*}
&\Ext^c_{A}(\mco,A)_\fp \cong \Ext^c_{A_{\fp}}(k(\fp),A_{\fp}) \cong k(\fp)  \\
&\Ext^c_{A}(\mco,\mco)_\fp \cong \Ext^c_{A_{\fp}}(k(\fp),k(\fp)) \cong k(\fp)\,.
\end{align*}
Thus the source and target of the map $\Ext^c_A(\mco,\lambda_A)$ have rank one, and the rank of the map is nonzero, since we already know that (i)$\Rightarrow$(ii). It follows that its cokernel is torsion, justifying (iii). 

(iii)$\Rightarrow$(ii):  The rank of the $\mco$-module $\Ext^c_A(\mco,\mco)$ is nonzero, and $\cmod A$ is the cokernel of  $\ecoh^c_A(\lambda_A)$, so (ii) follows.

This completes the proof of the equivalences of conditions (i), (ii), and (iii). As to remaining claims, (1) and (2) were implicitly verified in that proof that (i)$\Rightarrow$(iii), whilst (3) is a consequence of (1).
\end{proof}

\begin{lemma}
\label{le:Etors}
Let $\lambda_A\colon A\to \mco$ be as in \ref{ch:congruence-module} with $A$ is regular at $\fp_A$, and $M$ a finitely generated $A$-module. Then for each integer $n\ge \height\fp_A+1$ one has $\ecoh^n_A(M)=0$, and  the $\mco$-module $\cmod A(M)$ is torsion.
\end{lemma}

\begin{proof}
As before we write $\fp$ instead of $\fp_A$ and $k(\fp)$ for the fraction field of $\mco$. For each $i$ there is an isomorphism of $k(\fp)$-vector spaces
\[
\Ext^i_{A}(\mco,M)_{\fp} \cong \Ext^i_{A_\fp}(k(\fp),M_\fp)\,.
\]
As $A_\fp$ is regular of dimension $c\colonequals \height \fp$,  for $i\ge c+1$ one has   $\Ext^i_{A_\fp}(k(\fp),M_\fp)=0$. Thus  the $\mco$-module $\Ext^i_A(\mco,M)$ is torsion, and hence $\ecoh^i_A(M)=0$ for $i \geq c+1$.

It remains to verify that $\cmod A(M)$ is torsion; equivalently, with $\pi\colon M\to M/\fp M$ the natural surjection, the cokernel of $\Ext^c_A(\mco, \pi)$ is torsion; equivalently that
\[
\Ext^c_{A_\fp}(k(\fp), \pi_\fp)\colon  \Ext^c_{A_\fp}(k(\fp),M_\fp) \lra \Ext^c_{A_\fp}(k(\fp),  M_\fp/\fp M_\fp)
\]
is onto. Replacing $A_\fp$ by $R$ and $M_\fp$ by $N$, we find ourselves in the situation where $R$ is a regular local ring, with residue field $k(\fp)$ and Krull dimension $c$, and $N$ is an $R$-module. The desired result is that the map $\Ext^c_R(k(\fp),\lambda)$ is onto, where $\lambda\colon N\to N/\fm N$ is the natural map with $\fm$ the maximal ideal of $R$.

Since $R$ is regular, the $R$-module $k(\fp)$ has a finite free resolution, so for any $R$-module $N$ the natural map
\[
\RHom_R(k(\fp), R)\lotimes_R N \lra \RHom_R(k(\fp), N)
\]
 is a quasi-isomorphism. Again because $R$ is regular $\RHom_R(k(\fp),R)\simeq k(\fp)[-c]$; this fact yields a natural isomorphism
 \[
 k(\fp) \otimes_R N \xra{\ \cong\ } \Ext^c_R(k(\fp),N)\,.
 \]
We apply this also for the $R$-module $N/\fm N$.  Then the map $N\to N/\fm N$ induces a commutative diagram
\[
\begin{tikzcd}
 k(\fp)\otimes_R N  \arrow{d}[swap]{\cong}  \arrow{r} &  k(\fp)\otimes_R (N/\fm N)  \arrow{d}{\cong} \\
 \Ext^c_R(k(\fp) ,N) \arrow{r} & \Ext^c_R(k(\fp),N/\fm N) 
\end{tikzcd}
\]
It remains to note that the map in the top row, and hence in the bottom row, is even an isomorphism.
\end{proof}

\subsection*{Cotangent modules}
\label{sse:conormal}
We are also interested in the cotangent module $\fp_A/\fp_A^2$ of the map $\lambda_A$. This module too detects when $A$ is in the category $\acat$.

\begin{lemma}
\label{le:con-rank}
For $\lambda_A\colon A\to \mco$ as in \ref{ch:congruence-module} there is an equality
\[
\rank_{\mco}(\fp_A/\fp_A^2) = \edim A_{\fp_A}\,.
\]
Thus $\height \fp_A \le \rank_{\mco}(\fp_A/\fp_A^2)$; equality holds if and only if $A$ is regular at $\fp_A$.
\end{lemma}

\begin{proof}
We write $\fp$ instead of $\fp_A$, and let $e\colonequals \rank_{\mco}(\fp_A/\fp_A^2)$.   With $\fm$ the maximal ideal $\fp A_{\fp}$ of the local ring $A_{\fp}$, one has an isomorphism 
\[
\fm/\fm^2 = (\fp/\fp^2)_{\fp} \cong k(\fp)^e
\]
of $k(\fp)$-vector spaces,  where $k(\fp)$ is field of fractions of $\mco$, which is also the residue field of $A_{\fp}$. This justifies the first part of the statement. This also yields inequality $\height {\fp} \le e$. Since $\height {\fp}=e$ if and only if $A$  is regular at $\fp$, by definition~\cite[\S2.2]{Bruns/Herzog:1998}, the second part of the result follows as well.
\end{proof}

The preceding result has to do with the torsion-free quotient of the cotangent module of $\lambda_A$. In the sequel, its torsion submodule plays the major role. Set
\begin{equation}
\label{eq:conormal}
\con A\colonequals \tors{(\fp_A/\fp_A^2)}\,.
\end{equation}
When $c=0$ this is all of the cotangent module for the latter is torsion when $A$ is regular at $\fp_A$, by Lemma~\ref{le:con-rank}.

\subsection*{Wiles defect}
\label{pg:acat}
Let $\acat$ be the category of pairs $(A,\lambda_A)$ where $A$ is a complete noetherian local ring and $\lambda_A\colon A\to \mco$ is a surjective local homomorphism with $A$ regular at $\fp_A\colonequals \Ker(\lambda_A)$.  The morphisms in $\acat$ are the obvious ones. We impose the hypothesis that $A$ is complete only for ease of exposition, especially in Section~\ref{se:ci} and beyond. Often we speak of the ring $A$ being in the category, rather than the pair $(A,\lambda_A)$.  Let $\acat(c)$ be the subcategory of $\acat$  of rings $A$ with $\height \fp_A=c$; equivalently $\rank_{\mco}(\fp_A/\fp_A^2)=c$; see Lemma~\ref{le:con-rank}.

Fix $A$ in $\acat(c)$. The \emph{Wiles defect} of finitely generated $A$-module $M$ is the integer
\begin{equation}
\label{eq:defect}
\delta_A(M)\colonequals (\rank_{A_{\fp_A}}\! M_{\fp_A}) \length_{\mco}{\con A} - \length_{\mco}{\cmod A(M)}\,.
\end{equation}
For $c=0$ this invariant has been studied in \cites{Bockle/Khare/Manning:2021b, Brochard/Iyengar/Khare:2021b}; see the result below. For general $c$ but only for $M=A$ and Cohen--Macaulay,  an analogous invariant was introduced in \cite{Bockle/Khare/Manning:2021a}; see \ref{ch:BKM-defn} and \ref{ch:BKM2}. we  have a lot more to say about this invariant from Section~\ref{se:invariance-of-domain} onwards. 

The result below reconciles the definition of congruence modules introduced in this work with the one when  $c=0$ that has been discussed in the literature.

\begin{proposition}
\label{pr:old-and-new}
Suppose $A$ is in $\acat(0)$ and set $I_A\colonequals A[\fp_A]$. For any finitely generated $A$-module $M$ there is a natural surjection of $\mco$-modules
\[
\frac M{M[\fp_A] + M[I_A]}\twoheadrightarrow \cmod A(M)\,.
\]
with kernel a torsion $\mco$-module. This map is an isomorphism if $\depth_AM\ge 1$.
\end{proposition}

\begin{proof}
Given the equality
\[
\Hom_A(\mco, M/\fp_A M) = M/\fp_A M
\]
the congruence module of the $M$ is the cokernel of the composition
\[
M[\fp_A] \lra M/\fp_A M \to \tfree{(M/{\fp_A M})}\,.
\]
Consider the exact sequence of $\mco$-modules
\[
0\lra \frac{M[I_A]}{\fp_A M} \lra \frac{M}{\fp_A M} \lra  \frac{M}{M[I_A]} \lra 0
\]
The term on the left is $\mco$-torsion for it is annihilated by $\fp_A$ and $I_A$, and so a module over the ring $\cmod A$.  Thus there are natural surjections
\[
M/\fp_A M \twoheadrightarrow M/M[I_A] \twoheadrightarrow \tfree{(M/{\fp_A M})}
\] 
and the map on the right is bijective when $\depth_AM\ge 1$, for then the  term in the middle is torsion-free, being a submodule of $\Hom_A(I_A,M)$. This gives the desired result. Along the way we get that when $\depth_AM\ge 1$, one has
\[
M/M[I_A]  = \tfree{\left(M/\fp_A M\right)}
\]
 so one can recover the module on the left without recourse to $I_A$.
\end{proof}

Here is an example that shows that the map in the preceding proposition is not always an isomorphism.

\begin{example}
Let $A\colonequals \mco\pos{t}/(\varpi^2 t)$ and $M\colonequals A/(\varpi^3)$. Then $M$ is $\mco$-torsion so $\cmod A(M)=0$, whereas, in the notation of Proposition~\ref{pr:old-and-new} one has $I_A=(\varpi^2)$, so 
\[
M[\fp_A] = (\varpi^2)M \quad\text{and}\quad M[I_A]= (\varpi, t)M\,,
\]
and hence $M/(M[\fp_A]+M[I_A]) = \mco/(\varpi)$. 
\end{example}

\section{Depth conditions}
\label{se:depth}
Two invariants of a finitely generated $A$-module $M$ play a key role in what follows: its depth as an $A$-module, denoted $\depth_AM$, and the length of the longest $M$-regular sequence in $\fp_A$, denoted $\grade(\fp_A,M)$. Their value is closely connected to the property that $M$ is free at $\fp_A$. These interdependencies are clarified in the next two results; see also \ref{ch:summary}.

\begin{lemma}
\label{le:mcat-old}
Let $A$ be in $\acat(c)$ and $M\in \rmod A$.
\begin{enumerate}[\quad\rm(1)]
\item
If $M_{\fp_A}\ne 0$, then $\grade(\fp_A,M)\le c$ and $\depth_AM\le c+1$.
\item
If $\grade(\fp_A,M)\le c$ and $\depth_AM\ge c+1$, then in fact $\grade(\fp_A,M)=c$ and $\depth_AM= c+1$.
\item
If $\grade(\fp_A,M)\ge c$, then $M$ is free at $\fp_A$.
\end{enumerate}
In particular, $\depth A \le c + 1 \le \dim A$.
\end{lemma}

\begin{proof}
(1) When $M_{\fp_A}$ is nonzero, its depth as an $A_\fp$-module is at most $\dim A_\fp$; this justifies the second inequality below:
\[
\grade(\fp_A,M) \le \depth_{A_{\fp_A}}M_{\fp_A} \le c\,.
\]
The first one is clear. This justifies the first inequality in (1). The second one follows, because for any finitely generated $A$-module $M$ there in an inequality
\[
 \depth_AM \le \grade(\fp_A,M) + 1\,.
\]
This holds by  \cite[Exercise~1.2.23]{Bruns/Herzog:1998}, since $\dim (A/\fp_A)=\dim\mco=1$.
 
(2) The hypotheses gives inequalities
\[
c+1\le \depth_AM \le \grade(\fp_A,M) + 1 \le c+1\,.
\]
Thus one has equalities everywhere.

(3) The hypothesis implies that the depth of $M_{\fp_A}$ as a module over $A_{\fp_A}$  is at least $c$. Since  $A_{\fp_A}$ is a regular local ring of dimension $c$, the equality of Auslander and Buchsbaum~\cite[Theorem~1.3.3]{Bruns/Herzog:1998} implies that $M$ is free at $\fp_A$. 

Finally, the inequality $\depth A\le c+1$ follows from (1), as  $A_{\fp_A}\ne 0$, whereas the upper bound on $\dim A$ is the special case $\fp\colonequals \fp_A$ of the inequality 
\[
\dim (A/\fp) + \height\fp \le \dim A
\]
that holds for any prime ideal $\fp$ in $A$.
\end{proof}

In what follows the focus will be on the modules whose depth is at least $c+1$.
One reason for this is the following result.

\begin{lemma}
\label{le:mcat}
Let $A$ be in $\acat(c)$ and $M\in\rmod A$ with $\depth_AM\ge c+1$.
\begin{enumerate}[\quad\rm(1)]
\item
If $M$ is supported at $\fp_A$, then $\grade(\fp_A,M)=c$ and $\depth_AM=c+1$.
\item
If $M$ is not supported at $\fp_A$, then $\grade(\fp_A,M)=\depth_AM$.
\item
$M$ is free at $\fp_A$.
\item
$\Ext^c_A(\mco, M)$ is torsion-free as an $\mco$-module.
\end{enumerate}
\end{lemma}

\begin{proof}
Part (1) is immediate from parts (1) and (2) of Lemma~\ref{le:mcat-old}.

(2) Since the only prime ideals containing $\fp_A$ are itself and the maximal ideal of $A$, by \cite[Proposition~1.2.10]{Bruns/Herzog:1998} there is an equality
\[
\grade(\fp_A,M) = \min\{\depth_{A_\fp}M_\fp,\depth_AM\}\,.
\]
Thus if $M$ is not supported at $\fp_A$, then the equality above gives
\[
\grade(\fp_A,M)=\depth_AM\,,
\]
since the depth of $0$ is infinity. 

(3) Since $\grade(\fp_A,M)\ge c$, by (1) and (2), one can invoke Lemma~\ref{le:mcat-old}(3).

(4) Let $k$ denote the residue field of $\mco$, and consider the exact sequence 
\[
0\lra \mco \xra{\ \varpi\ } \mco \lra k\lra 0\,.
\]
It induces the exact sequence of $\mco$-modules
\[
\cdots \lra \Ext^{c}_A(k,M) \lra \Ext^c_A(\mco,M) \xra{\ \varpi\ } \Ext^c_A(\mco,M) \lra \cdots
\]
Since $\depth_AM = c+1$ one has $\Ext^{c}_A(k,M)=0$; see \cite[Theorem~1.2.5]{Bruns/Herzog:1998}. Thus the sequence above implies that $\varpi$ is not a zero-divisor on $\Ext^c_A(\mco,M)$.
\end{proof}

\begin{chunk}
\label{ch:summary}
Perhaps it is helpful to summarize the results above as follows:
\[
\begin{tikzcd}
\depth_AM\ge c+1 \arrow[Rightarrow]{r} & \grade(\fp_A,M)\ge c \arrow[Rightarrow]{d} \arrow[Rightarrow]{r} &  \depth_AM \ge c \\
&\text{$M$ is free at $\fp_A$}
\end{tikzcd}
\]
The two conditions on the right are vacuous when $c=0$, but not so the left-most one.  This condition comes up in the criteria for freeness and complete intersection presented in Section~\ref{se:diamond-wiles}.  The condition on grade comes up in arguments involving reduction to the case $c=0$; see Section~\ref{se:deformations}. Only the  weaker lower bound on $\depth_AM$ is required for the invariance of domain property discussed in Section~\ref{se:invariance-of-domain}.

We round up this discussion on depth with the following observation that is needed for reduction to $c=0$. In what follows, given a prime ideal $\fp$ in a ring $A$, we write $\fp^{(2)}$ for the ideal $\fp^2A_\fp\cap A$; this is the second symbolic power of $\fp$.
\label{pg:symbolic}

\begin{lemma}
\label{le:symbolic}
Fix $A$  in $\acat(c)$ for  $c\ge 1$, and an $f\in \fp_A$ nonzero in $A_{\fp_A}$. The local ring $A/fA$ is in $\acat$ if and only if $f\not\in\fp_A^{(2)}$; in that case $A/fA$ is in $\acat(c-1)$.
\end{lemma}

\begin{proof}
Set $B\colonequals A/fA$; since $f$ is in $\fp_A$ the map $\lambda_A\colon A\to \mco$ factors through $B$, yielding an augmentation $\lambda_B\colon B\to \mco$. Its kernel, $\fp_B$, equals $\fp_A/f$. Since the local ring $A_{\fp_A}$ is regular the quotient ring $B_{\fp_B}\cong A_{\fp_A}/f A_{\fp_A}$ is regular if and only if the image of $f$  is not in $\fp_A^2A_{\fp_A}$, that is to say, $f$ is not in $\fp^{(2)}$;  see \cite[Proposition~2.2.4]{Bruns/Herzog:1998}.  When this property holds $\dim B_{\fp_B} = \dim A_{\fp_A} - 1$, so $B$ is in $\acat(c-1)$. 
\end{proof}

\begin{lemma}
\label{le:prime-avoidance}
Let $A$ be a local ring in $\acat(c)$ and $M$ a finitely generated $A$-module with $\grade(\fp_A,M)\ge c$. Given any integer $1\le n\le c$, there exists a sequence $\bs g\colonequals g_1,\dots,g_n$ in $\fp_A$ such that:
\begin{enumerate}[\quad\rm(1)]
\item
The ring $A/{\bs g}A$ is in $\acat (c-n)$;
\item
The sequence $\bs g$ is $M$-regular and $\depth_{A/{\bs g}A}(M/{\bs g}M)=\depth_AM-n$.
\end{enumerate}
When $A$ is Cohen--Macaulay, $\bs g$ can be chosen to be regular on $A$ as well. 
\end{lemma}

\begin{proof}
It suffices to verify the statement for $n=1$. Given Lemma~\ref{le:symbolic}, one has to find an element $g$ in $\fp_A$ that avoids the associated primes of $M$ and also $\fp_A^{(2)}$. This can be done by a standard prime avoidance argument; see \cite[Theorem~81]{Kaplansky:1974}.
\end{proof}

\end{chunk}

We illustrate the use of these conditions by establishing a link between the Wiles defect of a module to that of the underlying ring, and by tracking the change in the congruence modules under certain maps of $A$-modules.

\subsection*{A defect formula}\label{sse:defect}
The natural morphism of complexes
\[
\RHom_A(\mco, A) \lotimes_A M \lra \RHom_A(\mco, M)
\]
and the K\"unneth map induce the map of $\mco$-modules
\[
\Ext^c_A(\mco,A)\otimes_A M \lra \Ext^c_A(\mco, M)\,,
\]
and hence the map 
\[
\eta_M\colon \ecoh^c_A(A)\otimes_{\mco}\tfree{(M/\fp_AM)} \lra \ecoh^c_A(M)\,.
\]
This map fits into a commutative diagram of $\mco$-modules
\begin{equation}
\label{eq:ecoh-cd}
\begin{tikzcd}[column sep=huge]
\ecoh^c_A(A)\otimes_{\mco}\tfree{(M/\fp_AM)} \arrow{d}[swap]{\eta_M} \arrow{r}{\ecoh^c(\lambda_A)\otimes \textrm{id}} 
	& \ecoh^c_A(\mco)\otimes_{\mco}\tfree {(M/\fp_AM)}   \arrow[leftarrow]{d}{\cong}\\
\ecoh^{c}_A(M)   \arrow{r}{\ecoh^c{(\lambda_A(M))}}
	& \ecoh^{c}_A(M/\fp_AM) 
\end{tikzcd}
\end{equation}
The isomorphism on the right is from \eqref{eq:ecoh-MO}. In fact, one has such a diagram for each  cohomological index $i$. This leads to the following ``defect formula" for congruence modules. It generalizes \cite[Theorem~1.2]{Brochard/Iyengar/Khare:2021b}, which is the case $c=0$, for then $\Coker(\eta_M)=M[\fp_A]/I_AM$, where $I_A=A[\fp_A]$.  See Proposition~\ref{pr:eta-Gorenstein} for another point of view on $\Coker(\eta_M)$.

\begin{lemma}
\label{le:defect-formula}
When  $M$ is  free at $\fp_A$, the map $\ecoh^c{(\lambda_A(M))}$ is injective,  and there is an exact sequence of $\mco$-modules
\[
0\lra \Coker(\eta_M) \lra \cmod A\otimes_\mco{\tfree{(M/\fp_AM)}} \lra \cmod A(M)\lra 0
\]
and hence there is an equality
\[
\delta_A(M) = \rank_{A_{\fp_A}}(M_{\fp_A}) \cdot \delta_A(A) + \length_{\mco}\Coker(\eta_M)\,.
\]
\end{lemma}

\begin{proof}
The map $\ecoh^c_A(\lambda_A)\colon \ecoh^c_A(A)\to \ecoh^c_A(\mco)$ is injective, by Theorem~\ref{th:catC2}.  Since $M$ is free at $\fp_A$,  the map $\ecoh^c_A(\lambda_A(M))$ is injective when localized at $\fp_A$, and hence injective. Given this observation the stated exact sequence is immediate from \eqref{eq:ecoh-cd}.  Since the $A_\fp$-module $M_{\fp_A}$ is free, its rank is the rank of the $\mco$-module $\tfree{{(M/\fp_AM)}}$. This remark, and the given exact sequence,  yield the desired equality.
\end{proof}

\begin{chunk}
\label{ch:change-modules}
Let $A$ be in $\acat(c)$ and $\pi\colon M\to N$ a map of finitely generated $A$-modules. This data induces the commutative diagram
\begin{equation}
\label{eq:change-diagram}
\begin{tikzcd}[column sep=huge]
\ecoh^c_A(M) \arrow{d} \arrow{r}{\ecoh^c_A(\pi)} 
	& \ecoh^{c}_A(N) \arrow{d} \\
\ecoh^c_A(M/\fp_A M)  \arrow{r}{\ecoh^c_{A}(\overline{\pi})} 
	& \ecoh^{c}_A(N/\fp_AN) 
\end{tikzcd}
\end{equation}
where $\overline{\pi}\colonequals \pi\otimes_A\mco$. This induces a map  $\cmod A(M) \to \cmod A(N)$ of $\mco$-modules. Thus $\cmod A(-)$ is a functor from the category of finitely generated $A$-modules to $\mco$-modules.
\end{chunk}

\begin{lemma}
\label{le:change-modules}
Let $A$ be in $\acat(c)$ and $\pi\colon M\to N$ a  map in $\rmod A$ such that $M$ and $N$ are free at $\fp_A$ and of the same rank. 
\begin{enumerate}[\quad\rm(1)]
\item
The free-$\mco$-modules $\ecoh^c_A(M)$ and $\ecoh^c_A(N)$ have the same rank.
\item
When $\pi$ is surjective, $\ecoh^c_A(\pi)$ is injective and $\ecoh^c_A(\overline{\pi})$ is bijective, and hence
\[
\length_{\mco}\cmod A(M) = \length_{\mco}\cmod A(N)  +  \length_{\mco}\Coker \ecoh^c_A(\pi)\,.
\]
\item
When $\pi$ is injective with cokernel of depth $c+1$, the map $\ecoh^c_A(\pi)$ is bijective.
\end{enumerate}
\end{lemma}

\begin{proof}
(1) This follows from the fact that the rank of an $\mco$-module can be computed after localization at $\fp_A$, and the isomorphisms
\[
 \ecoh^c_A(M)_{\fp_A} \cong \Ext^c_A(\mco,M)_{\fp_A} \cong \Ext^c_{A_{\fp_A}}(k(\fp),M_{\fp_A})
\]
where $k(\fp)$ is the residue field of $\mco$.

The following observation will be used below: Since $M$ and $N$ are free of the same rank at $\fp_A$, when $\pi$ is  surjective,  it is an isomorphism at $\fp_A$. The same conclusion holds when $\pi$ injective and $\depth_A\Coker(\pi)\ge c+1$.   In particular, $\Ext^c_A(\mco,\pi)$ is bijective when localized at $\fp_A$ and hence $\ecoh^c_A(\pi)$ is injective. 

(2) Since $\pi$ is surjective, so is $\overline{\pi}$, and hence also the induced map 
\[
\tfree{(\overline{\pi})}\colon \tfree{(M/\fp_AM)} \xra{\ \cong\ } \tfree{(N/\fp_AN)}\,.
\]
That the map is also injective can be verified by localizing at $\fp_A$, and using the fact $\pi$ is an isomorphism at $\fp_A$. Thus the natural isomorphism~\ref{eq:ecoh-MO} implies that $\ecoh^c_A(\overline{\pi})$ is bijective as claimed. Given that $\ecoh^c_A(\pi)$ is one-to-one and $\ecoh^c_A(\overline{\pi})$ is bijective, the stated formula for the lengths of congruence modules is immediate from the commutative diagram \eqref{eq:change-diagram}. Here we also use the fact that the vertical maps in the diagram are one-to-one, since $M$ and $N$ are free at $\fp_A$; see Lemma~\ref{le:defect-formula}.

(3) We already know that $\ecoh^c_A(\pi)$ is injective, so it remains to verify that it is also surjective.  Complete $\pi$ to an exact sequence
\[
0 \lra M\lra N\lra C\lra  0\,.
\]
Since $M,N$ are free at $\fp_A$ and of the same rank, $C_{\fp_A}=0$, and so the hypothesis that $\depth_AC = c+1$ and Lemma~\ref{le:mcat} yield $\grade(\fp_A,C)=c+1$. Thus $\Ext^c_A(\mco,C)=0$, by \cite[Theorem~1.2.5]{Bruns/Herzog:1998}, and  so the exact sequence above induces the exact sequence
\[
 \lra \Ext^c_A(\mco,M) \xra{\ \Ext^c_A(\mco,\pi)\ } \Ext^c_A(\mco,N) \lra \Ext^c_A(\mco,C) =0\,.
\]
Therefore the induced map on torsion-free quotients,   $\ecoh^c_A(\pi)$,  is also surjective.
\end{proof}

\section{Cohen--Macaulay modules}\label{se:CM-modules}
In this section we record results on congruence modules special to Cohen--Macaulay modules, intended for later use. There is also an alternative description of congruence modules which brings up a connection to duality. See also \cite[2.10]{Iyengar/Khare/Manning/Urban:2024}, which contains a more streamlined and transparent description of this pairing.

\subsection*{Dualizing complexes}\label{sse:dualizing}
Since each $A$ in $\acat$ is a quotient of a regular ring, namely, a power series ring over $\mco$, it has a dualizing complex, $\ddual A$, unique up to isomorphism in $\dcat A$, once it is normalized so that $\Ext^i_A(k,\ddual A)\ne 0$ precisely when $i=0$; see \cite[\href{https://stacks.math.columbia.edu/tag/0A7M}{Tag 0A7M}]{stacks-project}. Given any $A$-complex $M$ in $\dcat A$, set
\begin{equation}
\label{eq:dagger}
\ddual A(M)\colonequals \RHom_A(M,\ddual A)\,.
\end{equation}
The dualizing complex captures the depth and the dimension of any nonzero finitely generated $A$-module $M$, in that
\begin{equation}
\label{eq:depth-dim}
\begin{aligned}
\depth_AM &= \min\{i\mid \hh_i(\ddual A(M))\ne 0\} \\
\dim M &=  \max\{i\mid \hh_i(\ddual A(M))\ne 0\}\,.
\end{aligned}
\end{equation}
This follows from Grothendieck's  local duality theorem and the fact that depth and dimension are detected by the local cohomology of $M$; see~\cite[\href{https://stacks.math.columbia.edu/tag/0DWZ}{Tag 0DWZ}]{stacks-project} and also \cite[Theorem~3.5.7]{Bruns/Herzog:1998}. The assignment $M\mapsto \ddual A(M)$ induces an auto-equivalence 
\[
\ddual A(-) \colon {\dbcat A}^{\mathrm{op}} \xra{\ \equiv\ } \dbcat A\,.
\]
Hence the  biduality map $M\iso {\ddual A^2}(M)$ is an isomorphism for $M$ in $\dbcat A$.

\subsection*{Cohen--Macaulay modules}\label{sse:CM}
Fix $A$ in $\acat(c)$.  An $A$-module $M$ is  said to be \emph{Cohen--Macaulay} if it is finitely generated and satisfies $\depth_AM\ge \dim M$; equality holds when $M\ne 0$.  Thus if $M$ is a Cohen--Macaulay $A$-module $\hh_{i}(\ddual A(M))=0$ for all $i\ne \dim_AM$. When $M$ is a non-zero Cohen-Macaulay module, the $A$-module 
\begin{equation}
\label{eq:dual-cm}
M^{\vee}\colonequals \hh_{d}(\ddual A(M)) \quad\text{for $d=\dim_AM$}
\end{equation}
is also Cohen--Macaulay of dimension $d$; this is easy to verify, given $M\iso {\ddual A^2}(M)$. We speak of this as the \emph{dual} of $M$. 
Since 
\[
\ddual A(M) \simeq  M^{\vee} [d] \qquad \text{in $\dcat A$.}
\]
the assignment $M\mapsto M^{\vee}$ is an auto-equivalence on the category of Cohen--Macaulay $A$-modules. In what follows, we say a Cohen--Macaulay module $M$ is \emph{self-dual} to mean that there is an isomorphism of $A$-modules
\[
 M\cong M^{\vee}\,.
 \]

The result below concerns Cohen--Macaulay modules $M$ supported at $\fp_A$ and having maximal possible depth; namely, $\dim M = c+1 = \depth M$; see Lemma~\ref{le:mcat-old}. Let $\nu_{\mco}$ be the valuation associated to $\mco$.  The determinant of an endomorphism $\alpha\colon U\to U$ of free $\mco$-modules of finite rank is denoted $\det\alpha$.

\begin{proposition}
\label{pr:change-modules}
Let $A$ be a local ring in $\acat(c)$ and $M,N$ Cohen--Macaulay $A$-modules of dimension $c+1$, that are self-dual, and have the same rank at $\fp_A$. If $\pi\colon M\to N$ is a surjective map,  then
\[
\length_{\mco}\cmod A(M) - \length_{\mco}\cmod A(N) = \nu_{\mco}(\det\alpha) = \nu_{\mco}(\det\beta)\,,	
\]
where $\alpha \colon\ecoh^c_A(M)\to \ecoh^c_A(M)$ and $\beta \colon \ecoh^c_A(N)\to \ecoh^c_A(N)$ are the compositions
\begin{gather*}
\ecoh^c_A(M)\xra{\ \ecoh^c_A(\pi)\ } \ecoh^c_A(N) \cong \ecoh^c_A(N^{\vee}) \xra {\ \ecoh^c_A({\pi}^\vee)\ } 
	\ecoh^c_A(M^\vee)\cong \ecoh^c_A(M)\\
\ecoh^c_A(N)\cong \ecoh^c_A(N^{\vee}) \xra {\ \ecoh^c_A({\pi}^\vee)\ } 
	\ecoh^c_A(M^\vee)\cong \ecoh^c_A(M)\xra {\ \ecoh^c_A({\pi})\ }  \ecoh^c_A(N)\,.
	\end{gather*}
\end{proposition}

\begin{proof}
Since $M$ and $N$ have depth $c+1$, they are free at $\fp_A$, by Lemma~\ref{le:mcat}. We can assume their ranks are nonzero, else  each term in the desired equality is zero, and also that $\pi$ is not an isomorphism.  Consider the exact sequence 
\[
0 \lra L\lra M\lra N\lra  0
\]
induced by $\pi$. There are (in)equalities
\[
c+1\le \depth_A L \le \dim L \le \dim M = c+1\,,
\]
where the first and last inequality on the left can be verified easily using \eqref{eq:depth-dim}. Thus $L$ too is Cohen--Macaulay of dimension $c+1$. Thus the exact sequence above induces an exact sequence
\[
0\lra N^{\vee} \xra{\ \pi^{\vee}\ } M^{\vee}  \lra L^{\vee} \lra 0\,.
\]
Since  $L^{\vee}$ is also Cohen--Macaulay of dimension $c+1$,  Lemma~\ref{le:change-modules}(3) yields that $\ecoh^c_A(\pi^{\vee})$ is bijective. Given this, the desired equalities are immediate from the definition of $\alpha$ and $\beta$, and Lemma~\ref{le:change-modules}(2).
\end{proof}

\begin{chunk}
Let $A$ be in $\acat(c)$ and $M$ a Cohen--Macaulay $A$-module of dimension $c+1$. Then $M^{\vee}$ is also Cohen--Macaulay of dimension $c+1$ and there is a pairing
\[
 M \lotimes_A M^{\vee} \simeq M\lotimes_A \ddual A(M)[-c-1] \lra \ddual A[-c-1]\,.
\]
Up to a shift, this is the adjoint of the natural biduality map 
\[
M\to \RHom_A(\RHom_A(M,\ddual A),\ddual A)\,;
\]
see, for instance,  \cite[\href{https://stacks.math.columbia.edu/tag/0A5W}{Tag 0A5W}]{stacks-project}. This yields a natural pairing 
\begin{equation}
\label{eq:pairing}
\langle -,-\rangle \colon \Ext^c_A(\mco,M)\otimes_{\mco} \Ext^c_A(\mco, M^{\vee}) \lra \ecoh^c_A(\mco) \cong \mco\,.
\end{equation}
of $\mco$-modules defined as follows: Given morphisms 
\[
\alpha\colon \mco \to  M[c]\quad\text{and}\quad  \beta\colon \mco \to M^{\vee}[c]
\]
in $\dcat A$ consider the morphism
\[
\alpha\otimes \beta \colon \mco \lotimes_A \mco \lra  M[c] \lotimes_A  M^\vee[c] \simeq  (M \lotimes_A M^\vee)[2c] \lra  \ddual A[c-1]
 \]
This induces by adjunction a morphism 
\[
\gamma\colon \mco\lra \RHom_A(\mco,\ddual A[c-1])\simeq  \mco[c]\,,
\]
where the isomorphism holds as the ring $\mco$ is Gorenstein and $\dim\mco=1$.  Thus $\gamma$ represents an element in $\Ext^c_A(\mco,\mco)$. Set
\[
\langle \alpha,\beta\rangle\colonequals \text{the class of the morphism $\gamma$ in $\ecoh^c_A(\mco)$.}
\]

For the Cohen--Macaulay modules of dimension $c+1$, the pairing above gives another interpretation of the congruence module. 

\begin{proposition}\label{pr:CM-duality}
Let $M$ be a Cohen--Macaulay $A$-module of dimension $c+1$. The congruence module $\cmod A(M)$ is the cokernel of the map
\[
\Ext^c_A(\mco,M) \lra \Hom_{\mco}(\Ext^c_A(\mco,M^\vee), \ecoh^c_A(\mco))\,.
\]
adjoint to the pairing \eqref{eq:pairing}.
\end{proposition}

\begin{proof}
One has natural isomorphisms
\begin{align*}
\Ext^c_A(\mco,M^{\vee})
	&\cong \Hom_{\dcat A}(\mco, M^\vee[c]) \\
	&\cong \Hom_{\dcat A}(\mco, \ddual A(M)[-1]) \\
	& \cong \Hom_{\dcat A}( M[1], \ddual A(\mco)) \\
	&\cong \Hom_{\dcat A}( M[1], \mco[1]) \\
	&\cong \Hom_A(M,\mco) \\
	&\cong \Hom_{\mco}(M/\fp_AM, \mco) \\
	&\cong \Hom_{\mco}(\tfree{(M/\fp_AM)},\mco)
\end{align*}
In this chain, the third isomorphism is a composition of two adjunctions. The fourth uses the fact that $\ddual A(\mco)$ is the dualizing complex of $\mco$, and hence equal to $\mco[1]$ since the ring $\mco$ is Gorenstein; see \cite[\href{https://stacks.math.columbia.edu/tag/0AX1}{Tag 0AX1}, \href{https://stacks.math.columbia.edu/tag/0DW7}{Tag 0DW7}]{stacks-project}. The rest of the isomorphisms are standard. This yields isomorphisms
\begin{align*}
\Hom_{\mco}(\Ext^c_A(\mco,M^\vee),\ecoh^c_A(\mco)) 
	&\cong \Hom_{\mco}(\Hom_{\mco}(\tfree{(M/\fp_AM)},\mco),\ecoh^c_A(\mco)) \\
	& \cong \ecoh^c_A(\mco) \otimes_\mco \tfree{(M/\fp_AM)}\,.
\end{align*}
With this on hand, and a diagram chase, one gets that the adjoint map in question is naturally isomorphic to the map
\[
\Ext^c_A(\mco,M) \lra  \ecoh^c_A(\mco) \otimes_\mco \tfree{(M/\fp_AM)}\,,
\]
induced by $M\to M/\fp_AM$. This justifies the claim; see \eqref{eq:ecoh-MO}.
\end{proof}

\end{chunk}

\begin{chunk}
\label{ch:BKM-defn}
Let $M$ be a Cohen--Macaulay $A$-module of dimension $c+1$. Suppose there is map $R\to A$ such that $R$ is Gorenstein with $\dim R=c+1$, and $M$ is finitely generated over $R$. Thus $M$ is maximal Cohen--Macaulay as an $R$-module and hence $M^\vee \cong \Hom_R(M,R)$. The pairing \eqref{eq:pairing} is induced by the natural pairing
\[
M\otimes_R \Hom_R(M,R) \lra R
\]
The paring \eqref{eq:pairing} becomes even more concrete when $c=0$, in which case one can take $R=\mco$, and then the pairing is the composite
\[
M[\fp_A] \otimes_\mco \Hom_{\mco}(M,\mco)[\fp_A] \subseteq M\otimes_{\mco} \Hom_{\mco}(M,\mco) \lra \mco\,.
\]
This observation reconciles the definition of congruence modules introduced in this work with that in \cite[\S3]{Bockle/Khare/Manning:2021b}.
\end{chunk}

The following result that expresses, under certain conditions, the invariant $\Coker(\eta_M)$ appearing in the defect formula~\ref{le:defect-formula} in terms of the Tate cohomology modules, $\tExt^{i}_A(\mco,M)$, of the pair $(\mco,M)$; see \cite[Definition~6.1.1]{Buchweitz:2021}.

\begin{proposition}
\label{pr:eta-Gorenstein}
If $A$ is Gorenstein and $M$ is maximal Cohen-Macaulay, then there is an isomorphism of $\mco$-modules
\[
\Coker(\eta_M)\cong \tExt^{c}_A(\mco,M)\,.
\]
\end{proposition}

\begin{proof}
By construction, $\eta_M$ is the composition of maps
\[
\Ext^c_A(\mco,A) \otimes_A M \lra \Ext^c_A(\mco,M) \lra \tfree{\Ext^c_A(\mco,M)}\,.
\]
Since $M$ is maximal Cohen-Macaulay the $\mco$-module $\Ext^c_A(\mco,M)$ is torsion-free, by Lemma~\ref{le:mcat}, so the map on the right is an equality. We identify $\eta_M$ with the map on the left. Moreover, as $A$ is Gorenstein one gets that
\[
\omega_A(\mco)\cong \Ext_A^c(\mco,A)[-c] \quad\text{and}\quad \Ext_A^c(\mco,A)\cong\mco\,.
\]
From the first isomorphism  we deduce that
\[
\Tor^A_{-i}(M,\omega_A(\mco)) \cong
\begin{cases}
0  & \text{for $i\le c+1$}\\
\Ext^c_A(\mco,A)\otimes_AM & \text{for $i=c$}
\end{cases}
\]
We leave it to the reader to verify that $\eta_M$ identifies with the map
\[
\hh^{c}\mathbf{N}(\mco,M)\colon \Tor^A_{-c}(M,\omega_A(\mco)) \lra  \Ext^c_A(\mco,M)\,.
\]
from \cite[Theorem~6.2.5]{Buchweitz:2021}. It remains to note that from part (3) of \emph{op.~cit.} and the computation above one gets an exact sequence
\[
\Tor^A_{-c}(M,\omega_A(\mco)) \xra{\ \eta_M\ } \Ext^c_A(\mco,M) \lra \tExt^c_A(\mco,M) \lra 0
\]
of $\mco$-modules.
\end{proof}

\section{Complete intersections}\label{se:ci}

\begin{chunk}
\label{ch:cat}
From now on we focus on  rings $A$ of the form $P/I$ where  $P\colonequals \mco\pos{t_1,\dots,t_n}$, a ring of formal power series over $\mco$, and $I\subseteq (\varpi)(\bs t)+(\bs t)^2$. Set $\lambda_P\colon P\to \mco$ to be the quotient modulo $(\bs t)$, and $\lambda_A\colon A\to \mco$ the induced map. 

Let $\bs f \colonequals f_1,\dots,f_m$ be a minimal generating set for $I$. The cotangent module $\fp_A/\fp_A^2$ depends only on $n$ and the linear part of the $f_i$, in the following sense: By our assumption on $I$,  each $f_i$ has an unique expression of the form
\begin{equation}
\label{eq:A-presentation}
f_i \colonequals \sum_{j=1}^n{u_{ij}}t_j + g_i \qquad\text{with  $u_{ij}\in (\varpi)\mco $ and $g_i\in (\bs t)^2$.}
\end{equation}
Then one has a presentation
\[
\mco^m \xra{\ (u_{ij}) \ } \mco^n \lra \fp_A/\fp_A^2\lra 0\,.
\]
The following observations will be useful in the sequel.

\begin{lemma}
\label{le:con-rank2}
Fix $A$ as in \ref{ch:cat}. Then $A$ is in $\acat(c)$ if and only if $\rank (u_{ij})= n-c$. When this holds $\dim A\ge c+1$ and  $\height I\le n-c$; the latter inequality is an equality if and only if $\dim A=c+1$. 
\end{lemma}

\begin{proof}
The first part of the statement is immediate from Lemma~\ref{le:con-rank}, and the lower bound on $\dim A$ has been commented on in \ref{ch:congruence-module}.  As the ring $P$ is regular ring, \cite[Corollary~2.1.4]{Bruns/Herzog:1998} gives the first equality below:
\[
\height I = \dim P - \dim A \le  n+1 - (c+1) = n-c
\]
and the rest of the statement follows. 
\end{proof}

It also follows from the preceding discussion that when $A$ is in $\acat(c)$ one can put it in the form $P/(\bs f)$, where the relations $\bs f\colonequals f_1,\dots,f_m$ satisfy
\begin{equation}
\label{eq:nice-form}
f_i = \varpi^{d_i}t_i + g_i \qquad\text{with $g_i\in (\bs t)^2$}
\end{equation}
and $d_1\le\cdots\le d_{m}$, where $d_i<\infty$ for $i\le n-c$ and $d_i=\infty$ for $i>n-c$, and then $\varpi^{d_i}=0$. Observe that the $d_i$ are independent of the choice of the presentation for $A$, since one gets an isomorphism 
\[
\frac{\fp_A}{\fp_A^2} \cong \mco^{c} \oplus \frac{\mco}{\varpi^{d_1}\mco} \oplus \cdots \oplus \frac{\mco}{\varpi^{d_{n-c}}\mco}
\]
of $\mco$-modules.  In particular the $d_i$ are Fitting invariants of $\mco$-module $\fp_A/\fp_A^2$ Thus
\begin{equation}
\label{eq:con}
\con A \cong \frac{\mco}{\varpi^{d_1}\mco} \oplus \cdots \oplus \frac{\mco}{\varpi^{d_{n-c}}\mco}
\end{equation}
A standard prime-avoidance argument allows us to go further and obtain the result below; see \cite[Appendix, Proposition 2]{Wiles:1995} for the case $c=0$, and also \cite[Theorem~5.26]{Darmon/Diamond/Taylor:1997}. The hypothesis on $\dim A$ cannot be omitted in the last part of the statement for the complete intersection $C$ satisfies $\dim C=c+1$.

\begin{theorem}
\label{th:ci-approximation}
Fix $A$ in $\acat(c)$. In the notation above, let  $h\colonequals \height I$. There exists a minimal generating set $\bs f \colonequals f_1,\dots,f_m$ of the ideal $I$ where the $f_i$ have the form in \eqref{eq:nice-form} and the sequence $f_1,\dots,f_h$ is regular. In particular, when $\dim A=c+1$, the complete intersection $C\colonequals \mco\pos{\bs t}/(f_1,\dots,f_h)$  is in $\acat(c)$ and the surjective map $C\to A$  induces isomorphisms 
\[
\fp_C/\fp_C^2\xra{\ \cong\ } \fp_A/\fp_A^2\quad\text{and}\quad \con C\xra{\ \cong\ } \con A\,.
\]
\end{theorem}

\begin{proof}
 We start by choosing a minimal generating set $f_1',\dots,f_m'$ for the ideal $I$ where the $f'_i$ are of the form in \eqref{eq:nice-form}. Since the ring $\mco\pos{\bs t}$ is regular, and hence Cohen--Macaulay, the height of $I$ equals its grade, that is to say, the length of any maximal $\mco\pos{\bs t}$-regular sequence contained in $I$; see \cite[Corollary~2.1.4]{Bruns/Herzog:1998}. We claim that there exists a regular sequence $f_1,\dots,f_h$  such that
\[
f_i\equiv f'_i \mod I^2\,.
\] 
It is clear that setting $f_i=f'_i$ for $i=h+1,\dots,m$ gives the desired sequence.  We construct $f_1,\dots,f_h$ by an induction argument. 

To start with, set $f_1=f'_1$. Suppose that for some $1\le j<h$ elements $f_1,\dots,f_j$ have been found, with the desired properties. 
We can suppose $f'_{j+1}$ is zero-divisor in $\overline{A}\colonequals \mco\pos{\bs t}/(f_1,\dots,f_j)$; else we can take $f_{j+1}= f'_{j+1}$.  The ideal $I \overline{A}$ contains an element that is not a zerodivisor, because the grade of $I$ is $h$, and $f_1,\dots,f_j$ is a regular sequence in $I$ of length $j<h$. Since the ideal generated by $(f'_{j+1})+ (f'_{j+1},\dots,f'_h)^2$ in $\overline{A}$ agrees with $I \overline{A}$ up to radical, a prime avoidance argument yields an element $f''$ in $(f'_{j+1},\dots,f'_h)^2$ such that $f'_{j+1}+f''$ is a not a zerodivisor in $\overline{A}$; see \cite[Theorem~124]{Kaplansky:1974}, or \cite[Lemma~1.2.2]{Bruns/Herzog:1998}. Setting $f_{j+1}\colonequals f'_{j+1}+f''$ completes the induction step.

Suppose $\dim A=c+1$, so that $\height I = n-c$, by Lemma~\ref{le:con-rank2}. It is clear by the choice of the sequence $\bs f$  that the ring $C$ is a complete intersection, of dimension $c+1$, that it is  in $\acat(c)$, and that the induced map on cotangent modules has the stated properties.
\end{proof}

\begin{example}
Consider ring $A\colonequals \mco\pos{s,t}/I$ where $I=(\varpi s, \varpi^2s +st)$. This ring is in $\acat(1)$ and of dimension two. We have at least two choices for a complete intersection in $\acat(1)$ mapping onto $A$:  $\mco\pos{s,t}/(\varpi s)$ or $\mco\pos{s,t}/(\varpi^2s+st)$. The former is the one we want, so choosing any minimal generator for $I$ would not do.

In the same vein, consider $A\colonequals \mco\pos{s,t}/I$ where $I=(\varpi s, \varpi t, st)$. This ring is in $\acat(0)$ and of dimension one. The ring $\mco\pos{s,t}/(\varpi s, \varpi t)$ is in $\acat(0)$, maps onto $A$, and has the same cotangent space, but it is not a complete intersection. What we want is the ideal $(\varpi s, \varpi t + st)$, which is a complete intersection. 
\end{example}
\end{chunk}

\begin{chunk}
Next we record a criteria for detecting isomorphisms in $\acat$ in terms of cotangent modules. To that end, consider a surjective map $\vf\colon A\to B$  in $\acat(c)$. Since $\vf$ is surjective, so is the vertical map in the middle in the commutative diagram:
\begin{equation}
\label{eq:con-diagram}
\begin{tikzcd}
0 \arrow{r} & \con A \arrow[twoheadrightarrow]{d}[swap]{\con{\vf}} \arrow{r} & \fp_A/\fp_A^2 \arrow[twoheadrightarrow]{d} \arrow{r} 
	& \tfree {(\fp_A/\fp_A^2)} \arrow{r}\arrow{d}{\cong} & 0 \\
0 \arrow{r} & \con B \arrow{r} & \fp_B/\fp_B^2 \arrow{r} 
	& \tfree {(\fp_B/\fp_B^2)} \arrow{r} & 0	
\end{tikzcd}
\end{equation}
Thus the vertical map on the right is also surjective; the source and target of this map are free $\mco$-modules of the same rank, since $A$ and $B$ are in $\acat(c)$, so we deduce that the map is an isomorphism. The Snake Lemma then implies $\con{\vf}$ is surjective. 

Here is a criterion for detecting isomorphisms in $\acat$; it generalizes \cite[Theorem~5.21]{Darmon/Diamond/Taylor:1997} that covers the case $c=0$.  The argument below is essentially the one in \emph{op.\ cit.}, but couched in terms of $\aqh iBA{\mco}$,  the $i$th Andr\'e--Quillen homology module of the map $A\to B$.

\begin{lemma}
\label{le:ci-con}
Let $\vf\colon A\to B$ be a surjective a map in $\acat(c)$ with $B$ a complete intersection. If $\length_{\mco}\con A =\length_{\mco}\con B$, then $\vf$ is an isomorphism, and hence $A$ is also a complete intersection.
\end{lemma}

\begin{proof}
The crucial point is that $\aqh 2{\mco}B{\mco}=0$,  since $B$ is a complete intersection in $\acat(c)$. To see this, consider a presentation $P\to B$, as in \ref{ch:cat}. Since the augmentation $P\to \mco$ is generated by a regular sequence $\aqh i{\mco}P{\mco}=0$ for all $i\ne 1$, thus the Jacobi-Zariski sequence associated to $P\to B\to \mco$ yields an injection:
\[
0\lra \aqh 2{\mco}B{\mco} \lra \aqh 1BP{\mco} \,.
\]
Since $B_{\fp_B}$ is regular, the $\mco$-module $\aqh 2{\mco}B{\mco}$ is torsion, as can be seen by localizing at $\fp_B$. On the other hand, the $\mco$-module $\aqh 1BP{\mco}$ is free, as the map $P\to B$ is complete intersection. Thus $\aqh 2{\mco}B{\mco}=0$.

Consider the commutative diagram \eqref{eq:con-diagram}. The hypothesis is that the source and target of $\con{\vf}$ have the same length, hence it is a bijection. Then the diagram implies that  the map $\fp_A/\fp_A^2\to \fp_B/\fp_B^2$ is also a bijection. 

Set $J\colonequals \Ker(\vf)$. The Jacobi-Zariski sequence associated to $A\to B\to \mco$ reads
\[
0\lra (J/J^2)\otimes_B \mco \lra \fp_A/\fp_A^2 \lra \fp_B/\fp_B^2 \lra 0\,.
\]
The map on the left is one-to-one because  $\aqh 2{\mco}B{\mco}=0$. The map between the cotangent modules is bijective, so we conclude that $(J/J^2)\otimes_B \mco =0$. Hence  $J=0$, by Nakayama's Lemma. 
\end{proof}

\end{chunk}

\section{Tate constructions}\label{se:Tate}

Let $A$ be in $\acat(c)$, with presentation as in \ref{ch:cat}. When convenient we can assume that the relations $\bs f$ defining have $A$ the form given in Theorem~\ref{th:ci-approximation}. In this section we completely describe the structure of  the torsion-free quotient, $\ecoh^*_A(\mco)$, of the $\mco$-algebra $\Ext_A^*(\mco,\mco)$, which is the target of the map that computes the congruence module of $A$. The key tool used in the proofs are constructions of free resolutions introduced by Tate~\cite{Tate:1957}, and developed further in, for instance, Avramov's book~\cite{Avramov:1998}. We take the latter as our basic reference.

\begin{chunk}
\label{ch:acyclic-closure}
Let $A\la X\ra$ be the exterior algebra over $A$ on indeterminates $X\colonequals \{x_1,\dots,x_n\}$, with $|x_i|=1$ for each $i$.  We make this into a DG (=differential graded) algebra with differential determined by $d(x_i)=t_i$ for $1\le i\le n$. Thus $A\la X\ra$ is the Koszul complex on elements $\bs t$. 

Since the $g_i$, from \ref{eq:A-presentation}, are in $(\bs t)^2$ one can write $g_i = \sum_{j} g'_{ij}t_j$ with $g'_{ij}$ in $(\bs t)$. 
Evidently, the elements 
\[
z_i \colonequals  \sum_{j=1}^n u_{ij}x_i + \sum_{j=1}^n g'_{ij}x_j \qquad\text{for $1\le i\le m$,}
\]
are cycles in $A\la X\ra$ of degree one. Since $\bs f$ is a minimal generating set for the ideal $I$,  the homology classes $z_1,\dots,z_m$  are a minimal generating set for the $\mco$-module $\HH 1{A\la X\ra}$. Following Tate~\cite{Tate:1957}, we adjoin divided powers variables $Y\colonequals \{y_1,\dots,y_m\}$, in degree two, to kill these homology classes: 
\[
A\la X,Y\ra \colonequals A\la X,Y \mid d(y_i)=  z_i \text{ for $1\le i\le m$}\ra\,.
\]
This is sometimes called the \emph{Tate construction} of $\mco$ over $A$;  up to an isomorphism of DG algebras with divided powers, it is independent of the choice of a minimal generating set for the ideal $\fp_A$, and the $\mco$-module $\HH 1{A\la X\ra}$.  The surjection $\lambda_A\colon A\to \mco$ extends to a map  $A\la X,Y\ra\to \mco$ of DG algebras over $A$, where $X,Y$ map to $0$, for degree reasons.  By construction $\HH 1{A\la X,Y\ra}=0$. Unless the local ring $A$ is a complete intersection, $A\la X,Y\ra$ will have homology in higher degrees; see Proposition~\ref{pr:ci-tate}. Adjoining further exterior variables and divided power variables one can extend $A\la X,Y\ra$ to an acyclic closure 
\[
\ve\colon A\la U\ra\xra{\ \simeq\ }\mco\,;
\]
see \cite[Theorem~1]{Tate:1957} and \cite[Construction 6.3.1]{Avramov:1998}. By construction $U=\{U_i\}_{i\geqslant 1}$ is a graded set of indeterminates, with $U_1 = X$ and $U_2=Y$, and $A\la U\ra$ is the free $A$-algebra with divided powers, on the set $U$. 

The differential on $A\la U \ra$ satisfies
\begin{equation}
\label{eq:d-minimal}
d(A\la X\ra) \subseteq (\bs t)A\la X\ra \quad\text{and}\quad d(A\la U\ra) \subseteq (\varpi,\bs t)A\la U\ra
\end{equation}
The first inclusion holds by construction. The second one, which says that the $A\la U\ra$ is a minimal free resolution of $\mco$ over $A$, holds because the map $A\to \mco$ is large, for it has an algebra section, namely the structure map $\mco\to A$; for instance, see \cite[Corollary~2.7]{Avramov/Iyengar:2000}. The latter observation does not play a role in the sequel. 

Since the acyclic closure is a free resolution of $\mco$ over $A$, one has
\[
\Ext^*_A(\mco,\mco) = \hh^*(\End_A(A\la U\ra))\,,
\]
where $\End_A(A\la U\ra) \colonequals \Hom_A(A\la U\ra, A\la U\ra)$. The complex $\End_A(A\la U\ra)$ is a DG $A$-algebra, with composition product, and  induces the product on cohomology. This endows $\Ext^*_A(\mco,\mco)$ with a multiplication, making it into a (typically non-commutative) $\mco$-algebra. In the sequel we view $\Ext^*_A(\mco,\mco)$ as an $\mco$-algebra with this product; it coincides with the Yoneda product, up to a sign. 

The quasi-isomorphism $A\la U\ra\xra{\sim} \mco$ induces a quasi-isomorphism
\[
\Hom_A(A\la U\ra, A\la U\ra) \xra{\ \simeq\ } \Hom_A(A\la U\ra, \mco)
\]
so one can also use the complex on the right to compute $\Ext_A^*(\mco,\mco)$ as an $\mco$-module. More generally, for any $A$-module, or complex $M$, one has
\[
\Ext_A^*(\mco, M)\cong \hh^*(\Hom_A(A\la U\ra, M))\,.
\]
This identification is functorial in $M$.
\end{chunk}

\begin{chunk}
\label{ch:theta-construction}
So far we have not used the fact that $A$ is in $\acat(c)$, but now we do, to construct an element in $\Ext^c_A(\mco,\mco)$ that maps to a generator for $\ecoh_A^c(\mco)$. The discussion below is vacuous when $c=0$ so we assume $c\ge 1$. 

We can assume the relations $\bs f$ defining $A$ have the form in \eqref{eq:nice-form}. Thus the residue class of the $t_{n-c+1},\cdots,t_n$ are a basis for  the free $\mco$-module $\tfree{(\fp_A/\fp_A^2)}$. Applying \cite[Proposition~1.4]{Iyengar:2001b} yields $A$-linear maps 
\[
\theta_i\colon A\la U\ra\to A\la U\ra\,, \quad\text{for $n-c+1\le i\le n$,}
\]
with the following properties. 
\begin{enumerate}[\quad\rm(a)]
\item
$\theta_i$ is a $\varGamma$-derivation and for each $1\le j\le n$ one has
\[
\theta_i(x_j)=
\begin{cases}
1 & j=i \\
0 &\text{else}
\end{cases}
\]
where the $x_j$ are as in \ref{ch:acyclic-closure}.
\item
$d \theta_i + \theta_i d =0$, where $d$ is the differential on $A\la U\ra$.
\end{enumerate}
See  \cite[1.1]{Iyengar:2001b} for the notion of a $\varGamma$-derivation; the convention in \emph{op.~cit.} is that such derivations commute with the differential, up to the usual sign. Condition (b) above states that each $\theta_i$ represents a cycle, of (upper) degree $1$, in the complex $\End_A(A\la U\ra)$ and hence a class in $\Ext^1_A(\mco,\mco)$. Moreover the composition
\begin{equation}
\label{eq:thetaA}
\theta_A \colonequals \ve\circ \theta_n\circ\cdots \circ\theta_{n-c+1}\colon A\la U\ra \lra \mco 
\end{equation}
is a chain-map so represents a class in $\Ext^c_A(\mco,\mco)$. Evidently 
\[
\theta_A(x_{n-c+1}\cdots x_n)=1
\]
This computation is interesting because of the result below. Recall that by construction~\ref{ch:acyclic-closure} the Koszul complex, $A\la X\ra$, is a DG subalgebra of $A\la U\ra$, so any map from $A\la U\ra$ can be restricted to $A\la X\ra$. 

\begin{lemma}
\label{le:theta-class}
Let $\theta \colon A\la U\ra \to \mco$ be any $A$-linear chain-map,  of upper degree $c$, with $\theta(A\la X\ra)=\mco $. The class $[\theta]$ in $\Ext^c_A(\mco,\mco)$ generates the free $\mco$-module $\ecoh^c_A(\mco)$. This conclusion applies, in particular, to the class  $[\theta_A]$ from \eqref{eq:thetaA}.
\end{lemma}

\begin{proof}
Suppose to the contrary that $\theta$ does not generate the free $\mco$-module  $\ecoh^c_A(\mco)$. Since the latter has rank one, there must exist an $A$-linear chain-map $\alpha\colon A\la U\ra\to \mco$ of upper degree $c$ such that the class of $\theta - \varpi \alpha$ is zero in $\ecoh^c_A(\mco)$; equivalently that, for some integer $s\ge 1$, the class of $\varpi^s(\theta - \varpi\alpha)$ is zero in $\Ext^c_A(\mco,\mco)$. Thus, with $d$ denoting the differential on $A\la U\ra$, there exists a $A$-linear homotopy $\beta\colon A\la U\ra \to \mco$, of upper degree $c-1$, such that
\[
\varpi^s(\theta - \varpi\alpha) = \beta d\,.
\]
By the hypothesis, there exists an element $a$ in $A\la X\ra$ such that $\theta(a)=1$. Then from the equality above we get
\[
\varpi^s = \varpi^s\theta(a) =  \varpi^{s+1}\alpha(a) + \beta d(a) = \varpi^{s+1}\alpha(a)
\]
where the last equality holds as $d(a)$ is in $(\bs t)A\la U\ra$, hence $\beta d(a)=0$; see \eqref{eq:d-minimal}. This is a contradiction.
\end{proof}

The preceding result can be upgraded to a complete description of the torsion-free quotient of $\Ext^*_A(\mco,\mco)$. This is discussed below.

\begin{chunk}
\label{ch:algebra-structures}
Consider $\ecoh_A^*(M)\colonequals \{\ecoh_A^i(M)\}_{i\geqslant 0}$ for any $A$-complex $M$. The graded $\mco$-algebra structure on $\Ext^*_A(\mco,\mco)$ descends to $\ecoh_A^*(\mco)$. Moreover, $\ecoh^*_A(M)$ is a graded module over $\ecoh^*_A(\mco)$ for any $A$-complex $M$.

By construction, $\ecoh^*_A(\mco)$ is a subring of $\Ext^*_R(k(\fp),k(\fp))$ where $R\colonequals A_{\fp_A}$ and $k(\fp)$ is its residue field. Since $R$ is regular, $\Ext^*_R(k(\fp),k(\fp))$ is an exterior algebra on $\Ext^1_R(k(\fp),k(\fp))$, and hence strictly graded-commutative; see \cite[pp.~110]{Gulliksen/Levin:1969}. Thus $\ecoh^*_A(\mco)$ is also graded-commutative and  so there is a natural map
\begin{equation}
\label{eq:xi-definition}
\xi_A\colon \bigwedge_{\mco} \ecoh^1_A(\mco) \lra \ecoh^*_A(\mco)
\end{equation}
of graded $\mco$-algebras, with $\ecoh^1_A(\mco)$ in degree one. It is easy to check that this map is bijective when $c=0$, for then $\ecoh^*_A(\mco)\cong \mco$ and $\bigwedge_{\mco} \ecoh^1_A(\mco)\cong \mco$, since $\ecoh^1_A(\mco)=0$. The result below  generalizes this to arbitrary $c$.

\begin{theorem}
\label{th:xi-algebra}
For any $A$ in $\acat(c)$  the following statements hold.
\begin{enumerate}[\quad\rm(1)]
\item
There is an isomorphism of $\mco$-modules
\[
 \mco^{c} \cong \Hom_{\mco}(\fp_A/\fp_A^2,\mco)\cong \Ext^1_A(\mco,\mco)\,;
\]
in particular,  the map $\Ext^1_A(\mco,\mco)\to \ecoh^1_A(\mco)$ is bijective.
\item
The map $\xi_A$ from \eqref{eq:xi-definition} is bijective.
\end{enumerate}
\end{theorem}

\begin{proof}
Since the $\mco$-module $\fp_A/\fp_A^2$ has rank $c$, its $\mco$-dual is a free module of rank $c$; this explains the first isomorphism in (1). As to the other one,  consider the exact sequence of $A$-modules
\[
0\lra \fp_A\lra A \lra \mco \lra 0
\]
Since $\Ext^1_A(A,-)=0$, applying $\Hom_A(-,\mco)$ yields an exact sequence of $\mco$-modules
\[
0\lra \Hom_{A}(\mco,\mco)\xra{\ \cong\ } \Hom_A(A,\mco) \lra \Hom_A(\fp_A,\mco)\lra \Ext^1_A(\mco,\mco)\lra 0\,.
\]
It remains to note that adjunction yields an isomorphism
\[
\Hom_A(\fp_A,\mco) \cong \Hom_{\mco}(\mco \otimes_A\fp_A,\mco)\cong \Hom_{\mco}(\fp_A/\fp_A^2,\mco)\,.
\]
This justifies the isomorphisms in (1).

(2) Set $\Lambda\colonequals  \wedge_{\mco} \ecoh^1_A(\mco)$; thus $\Lambda^1= \Ext^1_A(\mco,\mco)$, by (1).   In the notation of \ref{ch:theta-construction}, the class of the maps 
\[
\theta_i\colon A\la U\ra \lra A\la U\ra \quad\text{for $n-c+1\le i\le n$}
\]
form a basis for the $\mco$-module $\Lambda^1$. These generate the $\mco$-algebra $\Lambda$ so Lemma~\ref{le:theta-class} is equivalent to: $\xi_A(\Lambda^{c}) = \ecoh^c_A(\mco)$. This is the key to the bijectivity of $\xi_A$. 

Indeed, the map $\xi_A$ is an isomorphism when localized $\fp_A$. Thus $\xi_A$ is one-to-one, for its source is torsion-free, and its cokernel is $\mco$-torsion. If $\xi_A$ is not surjective, then for some integer $i$  there exists an element $\alpha$ in $\Lambda^i\setminus \varpi \Lambda^i$ such that $\xi_A(\alpha)$ is in $\varpi \ecoh^i_A(\mco)$. Since $\Lambda$ is an exterior algebra on $\Lambda^1$, a free $\mco$-module of rank $c$, one has $\Lambda^{r-i}\alpha = \Lambda^{c}$ so that
\[
\xi_A(\Lambda^{c}) =  \xi_A(\Lambda^{c-i}\alpha) \subseteq \xi_A(\Lambda^{c-i}) \xi_A(\alpha) \subseteq \ecoh_A^{c-i}(\mco) \varpi \ecoh_A^{i}(\mco) \subseteq \varpi \ecoh_A^{c}(\mco)\,.
\]
This is a contradiction. Therefore $\xi_A$ is bijective, as claimed.
\end{proof}
\end{chunk}

It is clear from the construction that the map $\xi_A$ is functorial in $A$. Namely, given a map $\vf\colon A\to B$ in $\acat$, the natural map of $\mco$-algebra $\Ext^*_B(\mco,\mco)\to \Ext^*_A(\mco,\mco)$ induces the commutative diagram
\begin{equation}
\label{eq:xi-naturality}
\begin{tikzcd}
	& \bigwedge_{\mco} \ecoh^1_B(\mco) \arrow{d}[swap]{\cong} \arrow{rr}{\wedge_{\mco} \ecoh^1_{\vf}(\mco)}
			 && \bigwedge_{\mco} \ecoh^1_A(\mco)  \arrow{d}{\cong}   \\
	&  \ecoh^*_B(\mco) \arrow{rr}[swap]{\ecoh^*_{\vf}(\mco)}  
			 && \ecoh^*_A(\mco)
\end{tikzcd}
\end{equation}
where the isomorphisms are by Theorem~\ref{th:xi-algebra}.

\begin{proposition}
\label{pr:xi-iso}
Let $\vf\colon A\to B$ be a surjective map in $\acat(c)$. If $\vf$ is an isomorphism at $\fp_A$, then induced map 
\[
\ecoh^*_{\vf}(\mco)\colon \ecoh^*_B(\mco)\to \ecoh^*_A(\mco)
\]
of $\mco$-algebras is an isomorphism. In particular, $\ecoh^c_{\vf}(\mco)$ is bijective.
\end{proposition}

\begin{proof}
 Set $I\colonequals \Ker(\vf)$. Since $\vf$ is an isomorphism at $\fp_A$, one has $I_{\fp_A}=0$.  The diagram $A\to B\to\mco$ induces an exact sequence
\[
I/\fp_A I\lra \fp_A/\fp_A^2 \lra \fp_B/\fp_B^2 \lra 0
\]
of $\mco$-modules. Since $I_{\fp_A}=0$ the $\mco$-module $I/\fp_A I$ is torsion, so the exact sequence above induces an isomorphism
\[
\Hom_{\mco}( \fp_B/\fp_B^2,\mco)\xra{\ \cong\ } \Hom_{\mco}( \fp_A/\fp_A^2,\mco)\,.
\]
Thus the map $\ecoh^1_{\vf}(\mco)$ is bijective; see Theorem~\ref{th:xi-algebra}(1). It remains to recall the commutative diagram~\eqref{eq:xi-naturality}.
\end{proof}

\end{chunk}

\subsection*{Complete intersections}\label{sse:CI}
The remainder of the discussion in this section is on complete intersection rings, leading to the statement that the Wiles defect of such rings is zero; see Theorem~\ref{th:wiles}. The result below, which is already contained in the literature, leads us to it. 

\begin{proposition}
\label{pr:ci-tate}
Let $A$ be in $\acat$ and $A\la X,Y\ra$ the Tate construction from \ref{ch:acyclic-closure}. The following conditions are equivalent:
\begin{enumerate}[\quad\rm(a)]
\item
The local ring $A$ is complete intersection;
\item
The map $A\la X,Y\ra \to \mco$ is a quasi-isomorphism;
\item
The $\mco$-module $\HH 1{A\la X\ra}$ is free and the natural map is bijective:
\[
\bigwedge_{\mco}(\HH 1{A\la X\ra})\xra{\ \cong\ } \HH *{A\la X\ra}
\]
\end{enumerate}
\end{proposition}

\begin{proof}
Let $P\colonequals \mco\pos{t_1,\dots,t_n}$ and $A\colonequals P/(\bs f)$, as in \ref{ch:cat}. The kernel of the surjection $\lambda_P\colon P\to \mco$ is generated by the regular sequence $\bs t\colonequals t_1,\dots,t_m$.

(a)$\Rightarrow$(b): Since $A$ is complete intersection,  $\bs f$ is a regular sequence, contained in the ideal $(\bs t)$. Thus \cite[Theorem~4]{Tate:1957} applies, and yields the stated result.

(b)$\Leftrightarrow$(c): This is a special case of \cite[Theorem~1]{Blanco/Majadas/Rodicio:1998}; see also \cite[Theorem 2.3]{Iyengar:2001b}.

(c)$\Leftrightarrow$(a): Condition  (c) says precisely that the map $\lambda_A\colon A\to\mco$ is a quasi-complete intersection, in the sense of \cite[1.1]{Avramov/Henriques/Sega:2013}. Since $\mco$ is regular, hence a complete intersection, the desired equivalence is contained in \cite[Proposition~7.7]{Avramov/Henriques/Sega:2013}.
\end{proof}

\begin{chunk}
\label{ch:ci-case}
Let $A$ in $\acat(c)$ be a complete intersection. We can assume $A\colonequals P/(\bs f)$, where $\bs f \colonequals f_1,\dots,f_m$ is a regular sequence, and of the form given in \eqref{eq:nice-form}; see Theorem~\ref{th:ci-approximation}. By Proposition~\ref{pr:ci-tate}, the Tate construction $A\la X,Y\ra$ is the acyclic closure of  $\mco$ over $A$. Thus there is an isomorphism
\[
\Ext^*_A(\mco,\mco) \cong \hh^*(\Hom_A(A\la X,Y\ra,\mco))\,.
\]
Writing $\mco\la X,Y\ra$ for $\mco\otimes_A A\la X,Y\ra$, standard adjunction yields an isomorphism
\[
\Hom_A(A\la X,Y\ra,\mco)\cong \Hom_{\mco}(\mco\la X,Y\ra,\mco)
\]
Given that the $f_i$ are as in \eqref{eq:nice-form}, the differential $d$ on $\mco\la X,Y\ra$ satisfies $d(X)=0$ and $d(y_i) = \varpi^{d_i}x_i$ for $1\le i\le m=n-c$. Keeping in mind that the dual of a free algebra with divided powers is a symmetric algebra, one gets an isomorphism of complexes of $\mco$-modules
\[
\Hom_{\mco}(\mco\la X,Y\ra,\mco) \cong \mco[\eta_1,\dots,\eta_n,  \chi_1\,\dots,\chi_m\mid  d(\eta_i) = \varpi^{d_i}\chi_i\,,  d(\chi_j)=0]
\]
where each $ \eta_i$ is exterior variable of (upper) degree one, and  each $\chi_j$ is a polynomial variable of degree two. This yields an isomorphism of $\mco$-modules
\[
 \Ext^*_A(\mco,\mco) \cong \frac{\mco[\eta_{m+1},\dots,\eta_n, \chi_1\,\dots,\chi_m]}{(\varpi^{d_1}\chi_1,\dots,\varpi^{d_m}\chi_m )}
\]
Recall that $d_i=\infty$ for $c+1\le i\le n$.  Thus one gets an isomorphism of $\mco$-modules
\[
\Ext^c_A(\mco,\mco)\cong \mco \eta_{m+1}\cdots\eta_{n} \oplus (\text{torsion})
\]
where the torsion part depends on $c$ and $m$. Thus $\ecoh^c_A(\mco)$ is generated by the residue class of $\eta_{m+1}\cdots\eta_{n}$; confer the construction of the class $\theta$ in \ref{ch:theta-construction} and Lemma~\ref{le:theta-class}.

Since $A$ and $\mco$ are both complete intersections, the map $\lambda_A\colon A\to \mco$ is a quasi-complete intersection map, in the sense of \cite{Avramov/Henriques/Sega:2013}. Applying Theorem~2.5(4) from \emph{op.~cit.} yields
\[
\Ext^i_A(\mco,A) =
\begin{cases}
\mco & \text{when $i=c$} \\
0    & \text{else}.
\end{cases}
\]
In fact, one can follow the proof of \cite[Theorem~2.5(4)]{Avramov/Henriques/Sega:2013} and  deduce that
\[
\cmod A \cong \mco/(\varpi^{d_{1}+\cdots +d_m})
\]
Given this isomorphism and the description of $\con A$ in \eqref{eq:con}, it follows that when $A$ in $\acat$ is a complete intersection ring,  $\delta_A(A)=0$.  We omit details, for later on we present another proof that goes by a reduction to the case $c=0$; see Theorem~\ref{th:wiles}, which also contains a converse. 
\end{chunk}

\section{Invariance of domain}
\label{se:invariance-of-domain}
In this section we establish an ``invariance of domains" property for congruence modules; see Theorem~\ref{th:invariance-of-domain}.
The statement takes some preparation.

\begin{chunk}
Let $\vf\colon A\to B$ be a surjective map in $\acat(c)$ and $M$ a finitely generated $B$-module. For each integer $i$,   the natural transformation $\Ext^i_B(\mco, -)\to \Ext^i_A(\mco,-)$ induces the map:
\[
\ecoh^i_{\vf}(M) \colon \ecoh^i_B(M) \lra \ecoh^i_A(M)\,.
\]
Noting that $\fp_AM=\fp_BM$, the naturality of the transformation $\ecoh^i_{\vf}(-)$ gives a commutative diagram of $\mco$-modules
\begin{equation}
\label{eq:invariance-diagram}
\begin{tikzcd}[column sep=huge]
\ecoh^i_B(M) \arrow{d} \arrow{r}{\ecoh^i_{\vf}(M)} 
	& \ecoh^{i}_A(M) \arrow{d} \\
\ecoh^i_B(M/\fp_B M)  \arrow{r}{\ecoh^i_{\vf}(M/\fp_AM)} 
	& \ecoh^{i}_A(M/\fp_AM) 
\end{tikzcd}
\end{equation}
For $i=c$, the diagram above induces a map on the cokernels of the vertical maps:
\begin{equation}
\label{eq:invariance}
\cmod{\vf}(M)\colon \cmod B(M)\lra \cmod A(M)\,.
\end{equation}
\end{chunk}

This  result below may be viewed as a statement about invariance of domains for congruence modules. The case $c=0$ is covered by \cite[Lemma~3.4]{Brochard/Iyengar/Khare:2021b}. 

\begin{theorem}
\label{th:invariance-of-domain}
Let $\vf\colon A\to B$ be a surjective map of local rings in $\acat(c)$. For any finitely generated $B$-module $M$ with $\depth_BM\ge c$, the  map of $\mco$-modules in \eqref{eq:invariance} is bijective, and  hence 
\[
\length_{\mco}\cmod A(M) = \length_{\mco}\cmod B(M) \,.
\]
Thus $\delta_A(M)\ge \delta_B(M)$ with equality if and only if $\length_{\mco}\con A =\length_{\mco}\con B$.
\end{theorem}

\begin{proof}
It suffices to verify that \eqref{eq:invariance} is bijective; the claim about defects is a simple consequence. The map $\vf$ induces the surjective map $\vf_{\fp_A} \colon A_{\fp_A}\to B_{\fp_B}$ of local rings.  Its source and target are regular local rings of dimension $c$ so it is an isomorphism. This remark will be used in the argument below.

Consider the commutative diagram \eqref{eq:invariance-diagram}. We prove that the horizontal maps are isomorphisms for $i=c$; this implies the desired result. 

First we verify that the induced map  $\ecoh^c_{\vf}(M)\colon \ecoh^c_B(M)\to \ecoh^c_A(M)$ is bijective. Indeed since $\vf_{\fp_A}$ is an isomorphism, it is easy to see by localizing at $\fp_A$ that the kernel of the map 
\[
\Ext^c_B(\mco, M)\to \Ext^c_A(\mco, M)
\]
is $\mco$-torsion, and hence the map $\ecoh^c_{\vf}(M)$ is one-to-one. This requires no hypotheses on $M$.  The crucial point is that since $\depth_BM\ge c$ the map above is surjective, and hence so is $\ecoh^c_{\vf}(M)$. The argument goes as follows: Since $A$-acts on $M$ through $B$, there is an isomorphism of complexes 
\[
\RHom_A(\mco, M)\cong \RHom_B(B\lotimes_A \mco, M)\,,
\]
and the map $\RHom_B(\mco, M)\to \RHom_A(\mco,M)$ is induced by the natural morphism $ B\lotimes_A\mco\to \mco$ in $\dcat B$. Consider the exact triangle in $\dcat B$ that it generates:
\[
 J\lra B\lotimes_A\mco \lra \mco \lra J[1]
\]
For the desired statement it suffices that $\Ext^c_B(J,M)=0$. Evidently the map $\HH i{B\lotimes_A\mco}\to \HH i{\mco}$ is surjective for $i\ge 1$ and bijective for $i\le 0$; moreover, it is bijective for all $i$ once we localize at $\fp_B$. It follows that $\HH iJ$ is $0$ for $i\le 0$ and $\mco$-torsion for all $i$; equivalently, $\fm_B$-power torsion, where $\fm_B$ denotes the maximal ideal of $B$, for the $B$ action on $\HH iJ$ factors through $\lambda_B$. Since $\depth_BM \ge c$, any $\fm_B$-power torsion $B$-module $W$ satisfies
\[
\Ext_B^i(W,M)=0 \quad\text{for $i\le c-1$.}
\]
Given that $\HH iJ=0$ for $i\le 0$ and $\fm_B$-power torsion for all $i$, it follows, for example by using a standard spectral sequence argument, that
\[
\Ext_B^i(J,M)=0 \quad\text{for $i\le c$.}
\]
This is as desired.

Next we verify that the following map  is bijective:
\[
\ecoh^c_B(M/\fp_BM)\to \ecoh^c_A(M/\fp_AM)\,.
\]
 This part of the proof also does not require that $\depth_BM\ge c$. Given the  isomorphism \eqref{eq:ecoh-MO} it is enough to check that  $\ecoh^c_B(\mco)\to \ecoh^c_A(\mco)$ is bijective. This is contained in Proposition~\ref{pr:xi-iso}.
\end{proof}

The following application of Theorem~\ref{th:invariance-of-domain} will be used often in the sequel.

\begin{lemma}
\label{le:faithful}
Let $A$ be in $\acat(c)$ and $M$ a finitely generated $A$-module supported at $\fp_A$ and satisfying $\grade(\fp_A,M)\ge c$. Let $A'$ be the image of the natural map $A\to \End_A(M)$. 
The following statements hold.
\begin{enumerate}[\quad\rm(1)]
\item
$A'$ is in $\acat(c)$ and the surjective map $A\to A'$ is an isomorphism at $\fp_A$;
\item
$\mathrm{ass}_A A'\subseteq \mathrm{ass}_{A}M$;
\item
If $f\in A$ is not a zero-divisor on $M$, it is also not a zero-divisor on $A'$.
\item
There are equalities
\[
\length_{\mco}\cmod {A}(M) =\length_{\mco} \cmod {A'}(M)\quad\text{and}\quad 
\length_\mco \con {A}\ge \length_{\mco} \con {A'}\,.
\]
Hence  $\delta_{A}(M)\ge \delta_{A'}(M)$, with equality if and only if $\con{A}\cong\con{A'}$.
\end{enumerate}
\end{lemma}

\begin{proof}
The hypothesis on $M$ implies that $\depth_AM\ge c$ and also that $M$ is free at $\fp_A$; see Lemma~\ref{le:mcat-old}. These observations will be used below.

(1) Since $M$ is faithful as an $A'$-module and free at $\fp_A$, it follows that the map $A\to A'$ is an isomorphism at $\fp_A$. 

(2) Since $A'$ is an $A$-submodule of $\End_A(M)$, the inclusion below holds:
\[
\mathrm{ass}_A A'\subseteq \mathrm{ass}_A \End_A(M) = \mathrm{ass}_AM\,.
\]
The equality is a direct computation; solve \cite[Exercise~1.2.27]{Bruns/Herzog:1998}.

(3) The zero-divisors of a module over a noetherian ring are the union of its associated primes, so (3) is a consequence of (2).

(4) The action of $A$ on $M$ factors through $A'$ so  $\depth_{A'}M = \depth_AM\ge c$. Hence, given (1), it remains to apply Theorem~\ref{th:invariance-of-domain} to the map $A\to A'$.
\end{proof}

\section{Deformations}
\label{se:deformations}
The main result of this section, Theorem~\ref{th:reduction-delta},  is that the defect of a module does not change under deformations. Throughout  we fix the following notation.  

\begin{chunk}
\label{ch:notation}
Fix $A$  in $\acat(c)$, for some $c\ge 1$, and $M$ a finitely generated $A$-module supported at $\fp_A$, and with $\grade(\fp_A,M)=c$; this is the maximum value possible, by Lemma~\ref{le:mcat-old}. Let $f$ in $\fp_A\setminus \fp_A^{(2)}$ be such that it is not a zero-divisor on $M$. Set 
\[
B\colonequals A/fA \qquad\text{and}\qquad N \colonequals M/fM\,,
\]
where $N$ is viewed as a $B$-module. 
\end{chunk}

\begin{theorem}
\label{th:reduction-delta}
In the context of \ref{ch:notation}, the ring $B$ is in $\acat(c-1)$, and 
\[
\delta_A(M) = \delta_{B}(N)\,.
\]
Moreover $N$ is supported at $\fp_B$ and $\grade(\fp_B,N)=\grade(\fp_A,M)-1$.
\end{theorem}

This result often permits one to reduce questions about defects to the case when $c=0$, which has been studied earlier in the literature. Here is one application; see also Section~\ref{se:diamond-wiles}.

\begin{corollary}
\label{co:tate}
Let $A$ be a local ring in $\acat(c)$. If $M$ is a finitely generated $A$-module with $\grade(\fp_A,M)\ge c$, then $\delta_A(M)\ge 0$. 
\end{corollary}

\begin{proof}
Let $\bs g$ be an $M$-regular sequence of length $c$ provided by Lemma~\ref{le:prime-avoidance}. Set $B\colonequals A/\bs gA$ and $N\colonequals M/\bs g M$. Then a repeated application of Theorem~\ref{th:reduction-delta} yields $\delta_A(M) = \delta_B(N)$. After replacing $A$ and $M$ by $B$ and $N$, respectively, one can assume $c=0$. Then given the defect formula  Lemma \ref{le:defect-formula}, it suffices to verify that $\delta_A(A)\ge 0$, that is to say that
\[
\length_{\mco}\con A \ge \length_{\mco}\cmod A\,.
\]
Since $c=0$, one has $\con A = \fp_A/\fp_A^2$, and so the inequality above is exactly the one proved by Wiles'; see \cite[Appendix, Proposition 1]{Wiles:1995} and \cite[\#(5.2.3)]{Darmon/Diamond/Taylor:1997}.
\end{proof}

\begin{chunk}
\label{ch:BKM2}
Theorem~\ref{th:reduction-delta} also permits one to reconcile our notion of Wiles defect with the one introduced in \cite{Bockle/Khare/Manning:2021b} when $M=A$ is a Cohen--Macaulay ring. Indeed, both invariants  remain unchanged when we go modulo a suitable regular sequence; for the former, this is by Theorem~\ref{th:reduction-delta}, and for the latter, this is essentially by definition~\cite[\S3.4]{Bockle/Khare/Manning:2021b}. So it suffices to compare them when $c=0$, in which case they are defined in the same way.
\end{chunk}

The proof of Theorem~\ref{th:reduction-delta} makes repeated appeal to Theorem~\ref{th:invariance-of-domain}, and some observations of independent interest recorded below.

\begin{chunk}
\label{ch:order}
Let $A$ be a local ring in $\acat$, and consider the natural pairing
\[
\langle -, -\rangle \colon \Hom_\mco(\fp_A/\fp_A^2,\mco)\times \tfree{(\fp_A/\fp_A^2)} \lra \mco\,.
\]
Given an element $f\in \fp_A$, let $[f]$ denote its residue class in $\tfree{(\fp_A/\fp_A^2)}$, and consider the induced map
\[
\langle - ,[f]\rangle\colon \tfree{(\fp_A/\fp_A^2)} \lra \mco\,.
\]
Let $O(f)$ denote its image; this is called the order ideal of $[f]$. It is the subset of $\mco$ consisting of elements $\alpha(f)$, as $\alpha$ ranges over $\mco$-linear maps $\fp_A/\fp_A^2\to \mco$.  Set
\[
\nu_A(f) \colonequals \length_{\mco}(\mco/O(f))\,.
\]

Let  $f$ be as in \ref{ch:notation}. The maps $A\to B\to \mco$ give rise to an exact sequence 
\[
0\lra \mco \lra \frac{\fp_A}{\fp_A^2}\lra \frac{\fp_B}{\fp_B^2} \lra 0\,,
\]
of $\mco$-modules, where the map on the left sends $1\in\mco$ to the class of $f$ in $\fp_A/\fp_A^2$; it is one-to-one because it is so when it is localized at $\fp_A$. This uses the hypothesis that $f$ is not in $\fp_A^{(2)}$. Applying $\Hom_{\mco}(-,\mco)$ gives an exact sequence
\begin{equation}
\label{eq:order}
0\lra \Hom_B(\fp_B/\fp_B^2,\mco) \lra \Hom_A(\fp_A/\fp_A^2,\mco) \xra{\ \eth^1\ } \mco
\end{equation}
where the map on the right is evaluation on the class of $f$. In particular, the image of $\eth^1$ is $O(f)$, so 
\[
\length\Coker(\eth^1) = \nu_A(f)\,.
\]
The relevance of this discussion is manifest from the result below.

\begin{lemma}
\label{le:reduction-con}
In the set-up of \ref{ch:notation}, the ring $B$ is in $\acat(c-1)$ and
\[
\length_{\mco}\con {B} -  \length_{\mco}\con A = \nu_A(f)\,.
\]
\end{lemma}

\begin{proof}
Lemma~\ref{le:symbolic} yields that $B$ is in $\acat(c-1)$. By construction, the exact sequence~\eqref{eq:order} continues as the exact sequence
\[
 \Hom_A(\fp_A/\fp_A^2,\mco) \xra{\ \eth^1\ } \mco \lra \Ext^1_{\mco}(\fp_B/\fp_B^2,\mco)\lra \Ext^1_{\mco}(\fp_A/\fp_A^2,\mco)\lra 0\,.
\]
It is easy to check that the length of $\Ext^1_{\mco}(\fp_B/\fp_B^2,\mco)$ equals that of $\con B$, and similarly for $A$. Therefore the exact sequence above yields
\[
\length_{\mco}\con B - \length_{\mco}\con A = \length\Coker(\eth^1) = \nu_A(f)\,.
\]
This is  the desired equality.
\end{proof}

We need one more observation regarding these order ideals.

\begin{lemma}
\label{le:order-ideal}
With $A$ and $f$ as in \ref{ch:notation}, let $\pi\colon A\to A'$ be a surjective map in $\acat(c)$ that is an isomorphism at $\fp_A$.
Then $\pi(f)$ is in $\fp_{A'}\setminus \fp_{A'}^{(2)}$, and $\nu_{A'}(\pi(f))=\nu_A(f)$.
\end{lemma}

\begin{proof}
Since $\pi$ is surjective and an isomorphism at $\fp_A$, one has
\[
\fp_A^{(2)}=\pi^{-1}(\fp_{A'}^{(2)}) \quad\text{and}\quad \fp_A^2A_{\fp_A}\cong \fp_{A'}^2A'_{\fp_A'}\,.
\]
Thus $f$ not in $\fp_A^{(2)}$ implies that $\pi(f)$ is not in $\fp_{A'}^{(2)}$. The equality $\nu_{A'}(\pi(f))=\nu_A(f)$ is immediate from the definition of $\nu(-)$, given that $\pi$ induces an isomorphism
\[
\tfree{(\fp_A/\fp_A^2)}\cong \tfree{(\fp_{A'}/\fp_{A'}^2)}\,.
\]
See \eqref{eq:con-diagram}.
\end{proof}

\end{chunk}

Now we present a proof of Theorem~\ref{th:reduction-delta}. See \cite[Theorem~2.28]{Iyengar/Khare/Manning/Urban:2024} for a different perspective. 

\begin{proof}[Proof of Theorem~\ref{th:reduction-delta}]
The hypothesis is that $A$ is in $\acat(c)$, the $A$-module $M$ has depth at least $c$,  and $f\in \fp_A\setminus \fp_A^{(2)}$ is not a zero-divisor on $M$. We have to verify 
\[
\delta_A(M) = \delta_B(N) \quad\text{for $B\colonequals A/fA$ and $N\colonequals M/fM$.}
\]
One has $\depth_A(\fp_A,M)=c$ and $f\in \fp_A$ is not a zero-divisor on $M$. Thus $\depth_B(\fp_B,N)=c-1$. Since $f$ is in $\fp_A$, the surjection $M\to N$ induces an isomorphism $M/\fp_AM \cong N/\fp_B N$, so there is an equality
\begin{equation}
\label{eq:reduction-rank}
\rank_{A_{\fp_A}} (M_{\fp_A}) = \rank_{B_{\fp_B}} (N_{\fp_B})\,.
\end{equation}

We reduce to the case where $f$ is not a zero-divisor also on $A$, as follows:  Let $A'$ be the image of $A$ in $\End_A(M)$ and set $B'\colonequals A'/fA'$. Then $A'$ is in $\acat(c)$ and the map $A\to A'$ is an isomorphism at $\fp_A$, by Lemma~\ref{le:faithful}. Thus
image of $f$ in $A'$ is not contained in $\fp_{A'}^{(2)}$ and  $\nu_{A'}(f)=\nu_{A}(f)$, by Lemma~\ref{le:order-ideal}.
Applying  Lemma~\ref{le:reduction-con} to $A$ and $A'$ we deduce that $B$ and $B'$ are in $\acat{(c-1)}$, and
\[
\length_{\mco}\con B - \length_{\mco} \con A = \length_{\mco} \con {B'} - \length_{\mco} \con {A'}\,.
\]
Moreover $\grade(\fp_{A'},M)=\grade(\fp_A,M) \ge c$, since the $A$ action on $M$ factors through $A'$, and similarly $\grade(\fp_{B'},N)\ge c-1$, so Theorem~\ref{th:invariance-of-domain} yields equalities
\begin{align*}
&\length_{\mco}\cmod{A'}(M) = \length_{\mco} \cmod{A}(M) \\
&\length_{\mco} \cmod{B'}(N) = \length_{\mco} \cmod{B}(N)\,.
\end{align*}
The displayed equalities above yield
\[
\delta_A(M) - \delta_{B}(N) = \delta_{A'}(M) - \delta_{B'}(N)
\]
It thus suffices to verify the desired result for the $A'$-module $M$, so replacing $A$ and $A'$, and so $B$ by $B'$, one can assume $f$ is  not a zero-divisor also on $A$, as claimed; see Lemma~\ref{le:faithful}.

The next step in the proof is a reduction to the case $M=A$.  
The assertion below concerns the map $\eta_M$ from \eqref{eq:ecoh-cd}.

\begin{claim}
There is an isomorphism $\Coker(\eta_N)\cong \Coker(\eta_M)$ of $\mco$-modules.
\end{claim}

Indeed, since $f$ is not a zero-divisor on $A$, the free resolution of the $A$-module $B$  is $0\to A\xra{f}A\to 0$. In particular, it is finite free, so for any $A$-complex $X$ there is a natural isomorphism
\[
\RHom_A(B,X) \cong \RHom_A(B,A)\lotimes_A X \cong (B\lotimes_A X)[-1]
\]
in $\dcat B$.  This gives rise to isomorphisms
\[
\RHom_A(\mco, X) \cong \RHom_B(\mco, \RHom_A(B,X)) \cong \RHom_B(\mco, (B\lotimes_A X)[-1])\,.
\]
Thus, passing to cohomology and applying $\tfree{(-)}$ yields an isomorphism 
\[
\ecoh^i_A(X)  \cong \ecoh_B^{i-1}(B\lotimes_AX)
\]
for each integer $i$, and natural in $X$. Moreover, since $f$ is not a zero-divisor on $M$, the natural map 
\[
B\lotimes_AM \to \HH 0{B\lotimes_AM} \cong M/fM = N
\]
is a quasi-isomorphism. Summing up, one gets a commutative diagram 
\[
\begin{tikzcd}[column sep=huge]
\ecoh^c_A(A)\otimes_{\mco}\tfree{(M/\fp_AM)} \arrow{d}[swap]{\eta_M} \arrow{r}{\cong}
	& \ecoh^{c-1}_B(B)\otimes_{\mco}\tfree {(N/\fp_BN)}   \arrow{d}{\eta_N}\\
\ecoh^{c}_A(M)   \arrow{r}{\cong}
	& \ecoh^{c-1}_B(N) 
\end{tikzcd}
\]
of $\mco$-modules. This justifies the claim.

Given the preceding claim, \eqref{eq:reduction-rank}, and the exact sequence Lemma~\ref{le:defect-formula}, it suffices to establish the desired equality for $M=A$, namely that $\delta_A(A)=\delta_B(B)$. In view of Lemma~\ref{le:reduction-con}, this is equivalent to
\begin{equation}
\label{eq:cmod-equality}
\length_{\mco}\cmod B - \length_{\mco}\cmod A = \nu_A(f)\,.
\end{equation}
To that end, consider the commutative square  
\begin{equation}
\label{eq:reduction}
\begin{tikzcd}
\ecoh^c_A(A) \arrow{d}[swap]{\cong}  \arrow{r} 
	& \ecoh^c_A(\mco) \arrow{d}{\cong} \\
\ecoh^{c-1}_B(B) \arrow{r} 
	& \ecoh^{c-1}_B(B\lotimes_A \mco) \arrow{r} & \ecoh^{c-1}_B(\mco)
\end{tikzcd}
\end{equation}
The maps in the lower row are induced by canonical maps $B \to B\lotimes_A \mco \to \mco$ whose composition is $\lambda_B$. The key point is this:

\begin{claim}
The map $\ecoh^{c-1}_B(B\lotimes_A\mco) \to  \ecoh^{c-1}_B(\mco)$ is one-to-one and its cokernel has length equal to $\nu_A(f)$. 
\end{claim}

Given this claim and the diagram \eqref{eq:reduction} one gets inclusions
\[
\ecoh^{c-1}_B(B) \hookrightarrow \ecoh^{c-1}_B(B\lotimes_A\mco) \hookrightarrow \ecoh^{c-1}_B(\mco)
\]
These and the isomorphisms in \eqref{eq:reduction} yield \eqref{eq:cmod-equality}, as desired.

To verify the claim we consider the canonical exact triangle
\[
 \mco[1] \lra B\lotimes_A\mco \lra \mco\lra \mco[2]
\]
in $\dcat B$. Applying $\Hom_{\dcat B}(\mco, (-)[{c-1}])$ yields an exact sequence of $\mco$-modules
\[
\Ext^{c}_B(\mco,\mco) \lra \Ext^{c-1}_B(\mco, B\lotimes_A\mco) \lra \Ext^{c-1}_B(\mco,\mco) \xrightarrow{\ \chi\ } \Ext^{c+1}_B(\mco,\mco) \lra
\]
The first part of the claim follows because the $\mco$-module $\Ext^i_B(\mco,\mco)$ is torsion for all $i\ge c$; see Lemma~\ref{le:Etors} and ~\ref{ch:tors-remark}. As an aside we note that $\chi$ is the cohomology operator of the map $A\to B$,  discovered by Eisenbud and Gulliksen; see~\cite[\S9.1]{Avramov:1998}.

It remains to verify the second part of the claim. In view of the commutative diagram~\ref{eq:reduction}, this is tantamount to the statement that the cokernel of the map $\ecoh^c_A(\mco) \to \ecoh^{c-1}_B(\mco)$ has length $\nu_A(f)$. This follows from a certain equivariance property of the map. Namely, the connecting map 
\[
 \Ext^*_A(\mco,\mco) \lra \Ext^{*-1}_B(\mco,\mco)
\]
is $\Ext^*_B(\mco,\mco)$-linear, where the source is viewed as an $\Ext^*_B(\mco,\mco)$-module via the natural map $\Ext^*_B(\mco,\mco)\to \Ext^*_A(\mco,\mco)$.  Thus the induced map
\[
\eth^*\colon \ecoh^*_A(\mco) \lra \ecoh^{*-1}_B(\mco)
\]
is $\ecoh^*_B(\mco)$-linear. The map $\eth^1$ is precisely the one in \eqref{eq:order}.  In view of Theorem~\ref{th:xi-algebra}, the exact sequence~\eqref{eq:order} implies that the image of $\ecoh^1_B(\mco)$ in $\ecoh^1_A(\mco)$ is a nonzero direct summand, and hence also that $\ecoh^{c-1}_B(\mco)\cdot \ecoh^1_A(\mco)=\ecoh^c_A(\mco)$.  Since the image of $\eth^1$ is the order-ideal of $f$, the equivariance property of $\eth^*$ implies
\[
\eth^{c}(\ecoh^c_A(\mco)) = \eth^{c}(\ecoh^{c-1}_B(\mco)\ecoh^1_A(\mco)) = \ecoh^{c-1}_B(\mco)\eth^{1}(\ecoh^1_A(\mco)) = \varpi^{\nu_A(f)} \ecoh^{c-1}_B(\mco)
\]
Thus the cokernel of $\eth^{c}$ has length $\nu_A(f)$, as claimed. 

This completes the proof of the claim and hence of the theorem.
\end{proof}

\section{Criteria for freeness}
\label{se:diamond-wiles}
In this section we establish numerical criteria, in terms of congruence modules and cotangent modules, for detecting when a module has a free summand, and also for detecting the complete intersection local rings in $\acat$.

\subsection*{Gorenstein rings}
A noetherian local ring $R$ is \emph{Gorenstein} if $R$, viewed as a module over itself, has finite injective dimension. 
Over such a ring $R$, a finitely generated $R$-module $M$ is maximal Cohen--Macaulay if and only if
\begin{equation}
\label{eq:gor-mcm}
\Ext^i_R(M,R)=0\quad \text{for all $i\ge 1$.}
\end{equation}
This follows from \eqref{eq:depth-dim}, keeping in mind that $\ddual R\cong R[\dim R]$ in $\dcat R$; see also \cite[Theorem~3.3.10]{Bruns/Herzog:1998}. The result below is Theorem~\ref{th:intro-1} from the Introduction. In the statement, $e_A(M)$ denotes the multiplicity of an $A$-module $M$; see \cite[\S4.6]{Bruns/Herzog:1998}.

\begin{theorem}
\label{th:gorenstein}
Let $A$ be a Gorenstein local ring in $\acat$ and $M$ a maximal Cohen--Macaulay $A$-module with $\mu\colonequals \rank_{\fp_A}(M_{\fp_A})\ne 0$. Then $\delta_A(M)=  \mu \cdot \delta_A(A)$ holds if and only if
\[
M\cong A^{\mu}\oplus W \qquad\text{where $W_{\fp_A}=0$.}
\]
When this holds and $e_A(M)\le \mu  \cdot e(A)$, then $M$ is free.
\end{theorem}

A key step in the proof of Theorem~\ref{th:gorenstein} is a criterion for detecting free summands in modules over general Gorenstein local rings, and not special to the category $\acat$. This is explained below. 

\begin{lemma}
\label{le:summands}
Let $R$ be a Gorenstein local ring, $M\in\rmod R$ a maximal Cohen--Macaulay module, and $x\in R$ not a zero-divisor on $R$ and on $M$. If the $R/xR$-module $M/xM$ has a free summand of rank $\mu $, then so does the $R$-module $M$.
\end{lemma}

\begin{proof}
It suffices to verify that if $R'\colonequals R/xR$ is a direct summand of $M'\colonequals M/xM$, then $R$ is a direct summand of $M$. The latter condition is equivalent to the condition that the trace ideal of $M$ is $R$; that is to say, the natural map is surjective:
\[
\tau_M(R) \colon \Hom_R(M,R)\otimes_R M\lra R\,, \quad\text{given by $f\otimes m\to f(m)$.}
\]
See, for instance, \cite[Proposition 2.8(iii)]{Lindo:2017}. By Nakayama's Lemma, it suffices to prove that $\tau_M(R)$ is surjective after applying $-\otimes_R R'$; equivalently, the horizontal map in the  diagram below is surjective:
\[
\begin{tikzcd}
(\Hom_R(M,R)\otimes_R R') \otimes_{R'} M' \arrow{d}  \arrow{r} & R' \\
\Hom_{R'}(M',R') \otimes_{R'} M'\arrow[bend right=20]{ur}[swap]{\tau_{M'}(R')}
\end{tikzcd}
\] 
The vertical map is induced by the map
\[
\Hom_R(M,R) \lra \Hom_R(M,R') \cong \Hom_{R'}(M',R')\,,
\]
where the first map is induced by $R\to R'$, and the isomorphism is standard adjunction.  It is a simple exercise to check that the diagram above is commutative. Moreover,  $\tau_{M'}(R')$ is surjective because $R'$ is a direct summand of $M'$. Thus, it suffices to verify that the vertical map is surjective. Since $M$ is maximal Cohen--Macaulay, $\Ext^1_R(M,R)=0$ by \eqref{eq:gor-mcm}, so applying $\Hom_R(M,-)$ to the exact sequence
\[
0\lra R\xra{\ x\ } R \lra R' \lra 0
\]
yields the desired result.
\end{proof}

Here is a general result on detecting free summands. The hypothesis on the Ext modules is equivalent to the vanishing of the Tate cohomology module of the pair $(\mco,M)$, in degree $c$; see Proposition~\ref{pr:eta-Gorenstein} and also~\cite[Corollary~6.3.4]{Buchweitz:2021}. 

\begin{lemma}
\label{le:summands-gor}
Let $R$ be a Gorenstein local ring and $M$ a maximal Cohen--Macaulay $R$-module. Let $\lambda\colon R\to\mco$ be surjective map, set $\fp\colonequals \Ker(\lambda)$ and $c\colonequals \height\fp$. If the induced map $\widetilde{\eta}_M\colon \Ext^c_R(\mco,R)\otimes_RM\to \Ext^c_R(\mco,M)$ is surjective and $M_\fp$ is nonzero, then $M$ has a free $R$-summand.
\end{lemma}

\begin{proof}
The argument is by reduction to Krull dimension zero. Choose an element $x$ in $R$ such that $\lambda(x)=\varpi$, the uniformizer for $\mco$, and is not zero-divisor on $R$ or on $M$. Set $S\colonequals R/xR$ and $N\colonequals M/xM$. For any $R$-module, say $U$, on which $x$ is not a zero-divisor, the surjection $U\to U/xU$ induces, for each integer $n$, a map
\[
\Ext^n_R(\mco, U) \lra \Ext^n_R(\mco,U/xU)\cong \Ext^n_S(k,U/xU)
\]
natural in $U$. This isomorphism is by standard adjunction. Specializing $U$ to $R$ and to $M$ gives a commutative diagram:
\[
\begin{tikzcd}
\Ext^c_R(\mco,R)\otimes_R M \arrow{d} \arrow[twoheadrightarrow]{r}{\widetilde{\eta}_M} & \Ext^c_R(\mco,M) \arrow{d}  \\
\Ext^c_S(k,S)\otimes_S N \arrow{r}{\widetilde{\eta}_N} & \Ext^c_S(k,N)
\end{tikzcd}
\]
The vertical map on the right is nonzero, as can be seen by considering the long exact sequence  induced by $\Hom_R(\mco, -)$ to the exact sequence 
\[
0\lra M \xra{\ x\ } M \lra N\lra 0
\]
and Nakayama's Lemma. This is where we need the hypothesis that $M_\fp\ne 0$, for it implies $\Ext^c_R(\mco,M)\ne 0$; localize at $\fp$. Thus we conclude that the map $\widetilde{\eta}_N$ in the diagram above is nonzero. It thus suffices to prove following claim, for then $M$ has a free $R$-summand, by Lemma~\ref{le:summands}.

\begin{claim}
When $\widetilde{\eta}_N$ is nonzero, the $S$-module $N$ has a free summand.
\end{claim}

To verify the claim, choose a sequence $\bs x\colonequals x_1,\dots,x_r$ in $\fm_S$ that is regular on $S$ and on $N$; this is possible as both have maximal depth. Set $S'\colonequals S/{\bs x}S$ and $N'\colonequals N/{\bs x}N$.  By \cite[Lemma~3.1.16]{Bruns/Herzog:1998} one gets the isomorphisms in the diagram:
\[
\begin{tikzcd}
&\Ext^c_S(k,S)\otimes_S N \arrow{d}[swap]{\cong} \arrow{r}{\widetilde{\eta}_N} & \Ext^c_S(k,N) \arrow{d}{\cong} \\
\mathrm{Soc}(S')\otimes_{S'} N' \ar{r}{\cong} &\Ext^0_{S'}(k,S')\otimes_{S'}N' \arrow{r}{\widetilde{\eta}_{N'}} 
	& \Ext^0_{S'}(k,N')\ar{r}{\cong} &\mathrm{Soc}(N')
\end{tikzcd}
\]
The diagram is commutative because of the naturality of the maps.  Since $\widetilde{\eta}_N$ is nonzero so is $\widetilde{\eta}_{N'}$, that is to say, $\mathrm{Soc}(S')\cdot N' \ne 0$. We claim that this property implies  $N'$ has a free $S'$-summand, so then $N$ has a free $S$-summand, by Lemma~\ref{le:summands}, which is as desired. Here is one way to verify the claim. 

Choose an element $a \in N'$ such that $\mathrm{Soc}(S')a \ne 0$ and consider the $S'$-linear map $\alpha\colon S'\to N'$ defined by $1\mapsto a$. Since $S'$ is a Gorenstein ring of Krull dimension zero, $\mathrm{Soc}(S')$ is the smallest nonzero ideal in $S'$. Thus the choice of $a$ implies that $\alpha$ is injective. It remains to observe that $S'$ being a Gorenstein ring of Krull dimension zero, it is injective as a module over itself, so $\alpha$ is split-injective.
\end{proof}

We can now present a proof of Theorem~\ref{th:gorenstein}. Instead of the argument given below, one could also use Theorem~\ref{th:reduction-delta} to reduce to the case $c=0$, as in the proof of Corollary~\ref{co:tate}; then the result is implicit in the proof of \cite[Theorem~2.4]{Diamond:1997}. The latter statement is contained in Theorem~\ref{th:diamond} further below.

\begin{proof}[Proof of Theorem~\ref{th:gorenstein}]
Suppose $A$ is in $\acat(c)$.  Since $M$ is maximal Cohen--Macaulay, $\depth_AM\ge c+1$, so it is free at $\fp_A$ and the $\mco$-module $\Ext^c_A(\mco,M)$ is torsion-free, by Lemma~\ref{le:mcat}. Thus if $\delta_A(M)=\mu \cdot \delta_A(A)$, then Lemma~\ref{le:defect-formula} yields that the map
\[
\Ext^c_A(\mco,A)\otimes_A M\lra \Ext^c_A(\mco, M)
\]
is surjective. Then a repeated application of Lemma~\ref{le:summands-gor} yields that $M$ has a free summand of rank $\mu $. Thus $M\cong A^{\mu}\oplus W$ for some $A$-module $W$. It remains to observe that since $M$ is free of rank $\mu$ at $\fp_A$, one must have $W_{\fp_A}=0$.

Conversely, if $M\cong A^{\mu} \oplus W$ with $W$ not supported at $\fp_A$, then $\delta_A(W)=0$, by Lemma~\ref{le:defect-formula}, so
\[
\delta_A(M) = \mu  \cdot \delta_A(A) + \delta_A(W)=  \mu  \cdot \delta_A(A)\,.
\]

As to the last part of the statement, write $M\cong A^{\mu } \oplus W$ for some $A$-module $W$. When the stated bound on the multiplicity of $M$ holds, one gets
\[
\mu \cdot e(A) + e_A(W) = e_A(A^{\mu}\oplus W)= e_A(M) \le \mu  \cdot e(A)
\]
Thus $e_A(W)=0$. Since $M$ is maximal Cohen--Macaulay, either $W$ is $0$ or it is also maximal Cohen--Macaulay. The latter case cannot hold, 
the multiplicity of a nonzero maximal Cohen--Macaulay module is positive; see \cite[\S4.6]{Bruns/Herzog:1998}. 
\end{proof}

\subsection*{Complete intersection rings}
Next we turn to complete intersections. Any local ring $A$ in $\acat$ is of the form $P/I$, where $P\colonequals\mco\pos{\bs t}$, a ring of formal power series, and in particular, a regular local ring. Such an $A$ is complete intersection if and only if the ideal $I$ can be generated by a regular sequence; see \cite[\S2.3]{Bruns/Herzog:1998}. 

The result below generalizes a criterion for complete intersection due to Wiles~\cite[Proposition, Appendix]{Wiles:1995} and Lenstra~\cite[Theorem in introduction]{Lenstra}, as extended in \cite[Proposition~A.1]{FKR}, both of which deal with the case $c=0$. Compare this result with \cite[Theorem~2.6]{Iyengar/Khare/Manning/Urban:2024} that states that $A$ in $\acat$ is regular if and only if $\con A=0$, if and only if $\cmod A(A)=0$.

\begin{theorem}
\label{th:wiles}
Let $A$ be a local ring in $\acat(c)$ with $\depth A\ge c+1$. The ring $A$ is complete intersection if, and only if, $\delta_A(A)=0$.
\end{theorem}

\begin{proof}
One can reduce to the case $c=0$ as in the proof of Corollary~\ref{co:tate}.  At that point, one can invoke~\cite[Proposition~A.1]{FKR}.
\end{proof}

Here is an extension of a result due to Diamond~\cite[Theorem~2.4]{Diamond:1997}. 

\begin{theorem}
\label{th:diamond}
Let $A$ be a local ring in $\acat(c)$ and $M$ a finitely generated $A$-module with $\depth_AM\ge c+1$ and $\mu\colonequals \rank_{\fp_A}(M_{\fp_A})\ne 0$. Then $\delta_A(M)=0$ if and only if $A$ is complete intersection and
\[
M\cong A^{\mu}\oplus W \qquad\text{where $W_{\fp_A}=0$.}
\]
When this holds and $e_A(M)\le \mu\cdot e(A)$, then $M$ is free.
\end{theorem}

\begin{proof}
The ``if" direction of the statement is clear, given Theorem~\ref{th:wiles}. The proof of ``only if" direction uses the observation below. 

\begin{claim}
The desired statement holds for all $A$ in $\acat(c)$ if it holds for all such $A$ with $\depth A\ge 1$.
\end{claim}

Indeed, let $A'$ be the image of $A$ in $\End_A(M)$. Since the $A$ action of $M$ factors through the surjection $A\to A'$, the depth of $M$ as an $A'$-module also equals $c+1$. Moreover, Lemma~\ref{le:faithful} yields the first inequality below:
\[
0= \delta_A(M) \ge \delta_{A'}(M) \ge 0\,.
\]
The equality is by hypothesis and the second inequality is from Corollary~\ref{co:tate}.  Thus one gets the first equality below:
\[
\delta_{A'}(M)=0 \quad\text{and}\quad \length_{\mco}\con A=\length_{\mco}\con {A'}\,.
\]
The second one is again by Lemma~\ref{le:faithful}.  It thus suffices to prove the desired result for $A'$, for when $A'$ is complete intersection Lemma~\ref{le:ci-con} implies that $A\cong A'$. 

This justifies the claim.

We verify the first part of the statement by induction on $c$; the second part then follows as in the proof of Theorem~\ref{th:gorenstein}.

The base case is $c=0$. By the preceding claim, we can assume $A$  has positive depth.   Since $c=0$, one has $\delta_A(A)\ge 0$; see Corollary~\ref{co:tate}. Thus the hypothesis that $\delta_A(M)=0$ and Lemma~\ref{le:defect-formula} implies $\delta_A(A)=0$. Thus $A$ is complete intersection, by Theorem~\ref{th:wiles}. At this point, we can invoke Theorem~\ref{th:gorenstein} to conclude that $M$ has a free summand; this is where the hypothesis that the $M_{\fp_A}$ is nonzero is required, so that the rank of $M$ at $\fp_A$ is nonzero; see Lemma~\ref{le:mcat}.

This completes the proof when $c=0$.

Suppose that $c\ge 1$.  One again, thanks to the claim above, we can suppose $\depth A\ge 1$.  Pick an element $f$ in $\fp_A\setminus \fp_A^{(2)}$ that is not a zero-divisor on $M$ and on $A$. Set $B\colonequals A/fA$ and $N\colonequals M/fM$. The ring $B$ is in $\acat{(c-1)}$ and $\delta_{B}(N)=0$, by Theorem~\ref{th:reduction-delta}. Moreover the rank of $N$ at $\fp_B$ equals $\mu$, as can be easily verified. Thus, by induction, $B$ is a complete intersection and the $B$-module $N$ has a free summand of rank $\mu$. Thus $A$ is complete intersection,  by \cite[Theorem~2.3.4]{Bruns/Herzog:1998}, and that the $A$-module $M$ has a free summand of rank $\mu$, by Lemma~\ref{le:summands}.

This completes the proof.
\end{proof}

Here is one corollary, which may be seen as a counterpart to Lemma~\ref{le:ci-con}. When $c=0$  the result below is \cite{Lenstra}, and \cite[Theorem~5.8]{Darmon/Diamond/Taylor:1997}, both of which were inspired by Wiles' work in \cite[Appendix]{Wiles:1995}.

\begin{corollary}
\label{co:ci-delta}
Let $\vf\colon A\to B$ be a surjective a map in $\acat(c)$ with $B$ a complete intersection. If $\delta_A(B)=0$, then $\vf$ is bijective and $A$ is a complete intersection.
\end{corollary}

\begin{proof}
Since $B$ is a complete intersection, $\delta_B(B)=0$, by Theorem~\ref{th:wiles}. Thus the hypothesis and Theorem~\ref{th:invariance-of-domain}, applied with $M=B$, imply $\length_{\mco}\con A = \length_{\mco}\con B$. The desired result follows from Lemma~\ref{le:ci-con}.
\end{proof}

\part{Patching and duality}\label{part:patching and duality}
In this part we summarize the commutative algebra needed for the patching construction in the derived setting. Our main purpose is to show that this construction preserves duality; see Theorem \ref{th:duality}.

\section{Abstract patching}
\label{se:patching}
This section is mostly a recollection of the patching construction,  following the approach of  Calegari and Geraghty~\cite{Calegari/Geraghty:2018}, Hansen~\cite{Hansen:2012},  \cite{KT} and Allen et.\ al.~\cite{ACC+:2018}.  In particular, the ultrapatching method is due to Scholze~\cite{Scholze}.

\begin{chunk}
\label{ch:oy}
Going forth we assume the residue field $k\colonequals \mco/\varpi\mco$ of $\mco$ is finite, and of positive characteristic $\ell$. We fix a power series ring
\[
\mco\pos{\bs y}\quad\text{where $\bs y\colonequals  y_1,\ldots,y_r$.}
\]
For each integer $n\ge 1$, fix integers $e(1,n),\ldots,e(r,n)\ge n$ and quotient rings
\[
\Lambda_n \colonequals \mco\pos{\bs y}/I_n, \quad\text{where $I_n \colonequals \left((1+y_i)^{\ell^{e(i,n)}}-1\mid 1\le i\le r\right)$}.
\]
These are flat $\mco$-algebras, augmented to $\Lambda_0=\mco$, and complete intersections. For each $n$ pick a $\Lambda_n$-algebra $R_n$ that is finite as a $\Lambda_n$-module, and such that  there is an isomorphism $R_n/(\bs y)R_n = \mco \otimes_{\Lambda_n}R_n \cong R_0$; we  suppress this isomorphism in our notation and write an equality instead.   We assume no other relations between the $R_n$.  Let $\varphi_n\colon \Lambda_n\to R_n$ denote the structure maps. 
\end{chunk}

In what follows we consider derived actions of rings on complexes. We introduce some language and notation to facilitate the discussion.

\begin{chunk}
\label{ch:derived-action}
Let $\Lambda$ be a ring, $M$ a $\Lambda$-complex and $R$ a $\Lambda$-algebra.  We say that $M$ is a \emph{derived $R$-complex}, or that $R$ has a \emph{derived action} on $M$, to mean that there is a map of $\Lambda$-algebras $R\to \End_{\dcat \Lambda}(M)$. When $M$ is a module this means exactly that there is an $R$-module structure on $M$ extending its $\Lambda$-module structure. For a general $M$ one only gets an $R$-module structure on the homology modules $\hh_i(M)$ extending their $\Lambda$-structure.

Let $\rdcat {\Lambda}R$ denote the category of pairs $(M,\iota)$  consisting of an $M$ in $\dcat{\Lambda}$ and a map $\iota\colon R\to \End_{\dcat \Lambda}(M)$ of $\Lambda$-algebras. We usually write $M$, instead of the pair, for an object on $\rdcat{\Lambda}R$, but one has to keep in mind that an $M$ in $\dcat{\Lambda}$ may host different derived $R$-actions. The morphisms in $\rdcat{\Lambda}R$ are the morphisms of $\Lambda$-complexes compatible with the derived $R$-action, that is to say, a morphism $f\colon M\to N$ in $\dcat \Lambda$ where  for each $r\in R$ the diagram 
\[
\begin{tikzcd}
M \arrow{d}[swap]{r} \arrow{r}{f} & N \arrow{d}{r} \\
M \arrow{r}[swap]{f} & N
\end{tikzcd}
\]
in $\dcat \Lambda$ is commutative.   We record some stability properties of derived $R$-actions.

Let $(M,\iota)$ be a derived $R$-complex. For any $\Lambda$-complex, $C$, the $\RHom_{\Lambda}(M,C)$ has  a derived $R$-action with structure map the composition
\[
R\xra{\ \iota\ } \End_{\dcat{\Lambda}}(M)\lra \End_{\dcat{\Lambda}}(\RHom_{\Lambda}(M,C))
\]
where the one on the right is induced by $\RHom_{\Lambda}(-,C)$. An analogous statement holds for $\RHom_{\Lambda}(C,M)$.

A map $\Lambda\to \Lambda'$ of rings   induces a derived $R'\colonequals (R\otimes_\Lambda \Lambda')$-action on $\Lambda'\lotimes_\Lambda M$ in $\dcat{\Lambda}$ via the natural map
\[
\rho\colon \Lambda'\otimes_\Lambda \End_{\dcat{\Lambda}}(M)\lra \End_{\dcat{\Lambda'}}(\Lambda'\lotimes_{\Lambda}M)\,.
\]
Thus one gets a functor $\rdcat{\Lambda}R\to \rdcat{\Lambda'}{R'}$ that sends $(M,\iota)$ to $(\Lambda'\lotimes_{\Lambda}M,\rho(\Lambda'\otimes\iota))$. 
\end{chunk}

\begin{chunk}
\label{ch:pasy}
Fix integers $d$ and $\ell_0$, and consider a category $\PatchSys$ of \emph{patching systems} whose objects are collections $\perfs C\colonequals \left(C_n,\alpha_n\right)_{n\geqslant 0}$ where 
\begin{enumerate}[\quad\rm(1)]
	\item $C_n$ is a perfect $\Lambda_n$-complex;
	\item  $\hh_i(k\lotimes_{\Lambda_n}C_n) = 0$ for $i\not\in [d,d+\ell_0]$;
	\item $C_n$ is a derived $R_n$-complex;
	\item $\alpha_n\colon  \mco \lotimes_{\Lambda_n}C_n \iso C_0$ is an isomorphism  in $\rdcat{\mco}{R_n}$, where the $R_n$-action on $C_0$ is via the surjection $R_n\to  R_0$.
\end{enumerate}
Given  (2), condition (1) is equivalent to requiring that $C_n$ is in $\dbcat{\Lambda_n}$.

A morphism  $f\colon \perfs C \to \perfs D$ in $\PatchSys$ is a family $f_n\colon C_n\to D_n$  of morphisms in $\rdcat{\Lambda_n}{R_n}$ commuting with the morphisms in the systems defining $\perfs C$ and $\perfs D$. The category $\PatchSys$ is defined over the system of rings $(\Lambda_n)$; this will be emphasized only if needed.
\end{chunk}

\begin{chunk}
\label{ch:square}
Consider the power series ring
\begin{align*}
S_\infty  &\colonequals   \mco\pos{\bs y,\bs w}  \quad\text{where $\bs w\colonequals  w_1,\ldots,w_j$, and set}\\
\fn &\colonequals (\bs y,\bs w)\,.
\end{align*}
The ring $\mco\pos{\bs y}$ is a subring of $S_\infty$. In what follows we write $\base{(-)}$ for the extension of scalars functor (also known as \emph{framing}) from the category of finitely generated $\mco\pos{\bs y}$-modules to finitely generated $S_\infty$-modules; thus
\[
\base M \colonequals S_{\infty}\otimes_{\mco\pos{\bs y}}M\,.
\]
When $\Lambda$ is a module-finite $\mco\pos{\bs y}$-algebra and $M$ a finitely generated $\Lambda$-module, $\base{\Lambda}$ is a module-finite $S_\infty$-algebra and $\base{M}$ a finitely generated $\base{\Lambda}$-module. Moreover, since the extension $\mco\pos{\bs y}\to S_\infty$ is flat, $\base{(-)}$ extends to a functor on appropriate derived categories, derived complexes, and even to the category of patching systems. In particular, one gets a functor from patching systems over $(\Lambda_n)$ to those defined over $(\base{\Lambda}_n)$. These remarks will be used without further comment.
\end{chunk}

\begin{chunk}
\label{ch:Rinfinity}
Let $r,\ell_0$ and $j$ be the integers from \ref{ch:oy}, \ref{ch:pasy}, and \ref{ch:square}, respectively. Fix a complete local, flat, noetherian $\mco$-algebra $R_\infty$ of dimension  $r+j-\ell_0+1$, equipped with surjective $\mco$-algebra morphism $\pi_n\colon R_\infty\twoheadrightarrow \base R_n$ for $n\ge 1$. We do not assume  the  $\pi_n$ are compatible.  The ultrapatching construction of \cite[Section 6.4]{ACC+:2018} and the work of \cite{Hansen:2012} give the following

\begin{theorem}
\label{th:ultrapatching}
There exists a homomorphism $\varphi_\infty\colon S_\infty\to R_\infty$ making $R_\infty$ into a finite $S_\infty$-algebra and a functor $\patch\colon \PatchSys\to \mathrm{mod}\, R_\infty$ with the following properties: 
\begin{enumerate}[\quad\rm(1)]
\item 
$\patch(\perfs C)$ is a maximal Cohen--Macaulay $R_\infty$-module, nonzero when $\perfs C\ne 0$.
\item 
There is a surjection $R_\infty/\fn R_\infty \twoheadrightarrow  R_0$ and an isomorphism of $R_0$-modules
\[
\patch(\perfs C)/\fn\patch(\perfs C) \cong \hh_{d}(C_0)\,.
\]
\item 
Treating $\patch(\perfs C)$ as an $S_\infty$-module via $\varphi_\infty$, for any open ideal $\fa \subseteq S_\infty$ and for infinitely many $n$, there are isomorphisms
\[
 (S_\infty/\fa) \lotimes_{S_\infty} \patch(\perfs C)[d] \longiso (S_\infty/\fa) \lotimes_{\base{\Lambda}_n}\base{C}_n,
\]
of derived $R_\infty$-complexes, where $R_\infty$ acts on $\base{C}_n$ via $\base{\iota_n}\pi_n$.
\item 
There is an isomorphism of $R_\infty$-algebras:
\[
\End_{\dcat{S_\infty}}(\patch(\perfs C)) 
\cong \lim_\fa\End_{\dcat{S_\infty/\fa}}((S_\infty/\fa) \lotimes_{S_\infty} \patch(\perfs C))
\]
\item 
Fix a map $f\colon \perfs C \to \perfs C$ in $\PatchSys$. If there exists an element $\tau\in R_\infty$ such that
\[
\pi_n(\tau) = \base{f}_n  \quad \text{ as elements in } \End_{\dcat{\base{\Lambda}_n}}(\base{C}_n)
\]
for all but finitely many $n$, then $\patch(f)=\tau$ in $\End_{R_\infty}(\patch(\perfs C))$.
\end{enumerate}
\end{theorem}

\begin{proof}
As noted in \ref{ch:square}, each patching system $\perfs C$ over $(\Lambda_n)$ gives rise to the framed patching system $\base{\perfs C}$ over $(\base{\Lambda}_n)$. Parts (3)--(5) of the statement only involve ${\base{\perfs C}}$, so does the construction of the functor $\patch$. Thus replacing $(\Lambda_n), (R_n)$ and  $\perfs C$ by $(\base{\Lambda}_n), (\base{R}_n)$, and $(\base{\perfs C})$, we assume we are in the framed context. 

We use the ultrapatching method introduced in \cite[\S 9]{Scholze}. Fix a non-principal ultrafilter $\uf$ on $\mathbb N$. Take any finite local ring $A$ (that is to say, a ring with finite cardinality). Consider the ring $\prod_{n=1}^\infty A$. As $A$ is finite and $\uf$ is an ultrafilter, for any element $(a_1,a_2,\ldots)$ in $\prod_{n=1}^\infty A$ there is a unique element $a \in A$ for which $\{n\in \mathbb N|a_n = a\}\in \uf$. Denote this element $f(a_1,a_2,\ldots)$. It is easy to verify that the function $f\colon \prod_{n=1}^\infty A\to A$ is an $A$-algebra homomorphism, and that for any $I\in\uf$, $f$ factors through the natural quotient map $\prod_{n=1}^\infty A\onto \prod_{n\in I} A$.

Let $\fm_A$ be the maximal ideal of $A$ and let $\fp = f^{-1}(\fm_A)$; this is a maximal ideal of $\prod_{n=1}^\infty A$. Then $f$ induces an isomorphism $\left(\prod_{n=1}^\infty A\right)_{\fp}\isomto A$ and so in particular for any $I\in \uf$, the map $\left(\prod_{n=1}^\infty A\right)_{\fp}\to \left(\prod_{n\in I} A\right)_{\fp}$ is an isomorphism.

For any sequence $M_1,M_2,\ldots$ of $A$-modules (respectively $A$-algebras) the kernel of the map $\prod_{n=1}^\infty M_n\onto \left(\prod_{n=1}^\infty M_n\right)_{\fp}$ is the set of sequences $(m_1,m_2,\ldots)\in \prod_{n=1}^\infty M_n$ for which $\{n\in\mathbb{N}|m_n=0\}$ is in $\uf$. By definition this implies that $\left(\prod_{n=1}^\infty M_n\right)_{\fp}$ is the ultraproduct of the $M_n$'s.

If the $M_n$'s are all finite with uniformly bounded cardinality, standard properties of ultraproducts imply that the set of $i\in \mathbb N$ for which there is an isomorphism
\[
\left(\prod_{n=1}^\infty M_n\right)_{\fp} \cong M_i
\]
as $A$-modules (respectively, as $A$-algebras) is an element of $\uf$. In particular, as $\uf$ is a \emph{nonprincipal} ultrafilter, $\left(\prod_{n=1}^\infty M_n\right)_{\fp}$ is isomorphic to $M_i$ for infinitely many $i$.

For any open ideal $\fa\subseteq S_\infty$ and integer $s\ge 0$ set
\[
R(s,\fa,\infty) \colonequals  
\left(\prod_{n=1}^\infty S_\infty/\fa \otimes_{\Lambda_n}  R_n/\fm_{R_n}^s
\right)_{\fp}
\]
where when $ I_n \not\subseteq \fa$ we interpret 
\[
S_\infty/\fa \otimes_{\Lambda_n} R_n/\fm_{R_n}^s = 0 \,;
\]
this only happens for finitely many $n$. As in \cite[Remark 6.4.14]{ACC+:2018}, we get surjections $R_\infty\twoheadrightarrow  R(s,\fa,\infty)$ for all $s,\fa$,  inducing a surjection $R_\infty\twoheadrightarrow  \lim_{s,\fa}R(s,\fa,\infty)$. Since each $R(s,\fa,\infty)$ is an $S_\infty$-algebra,  the structure map $S_\infty\to \lim_{s,\fa}R(s,\fa,\infty)$ lifts to a homomorphism 
\[
\varphi_\infty\colon S_\infty\to R_\infty
\]
as desired. For the moment, any lift will do, but later on we prove that there exists a lift that ensures that $R_\infty$ is finite as an $S_\infty$-module. In any case, \cite[Lemma 6.4.15]{ACC+:2018} gives the surjection 
\[
R_\infty/\fn R_\infty = \mco\otimes_{S_\infty} R_\infty \twoheadrightarrow  R_0\,.
\]

Let $\perfs C\colonequals \left(C_n,\alpha_n\right)_{n\geqslant 0}$ be a patching system. Since each $C_n$ is perfect,  we can replacing it by its minimal free resolution over $\Lambda_n$, and assume  $C_n$ is a bounded complex of free modules; see~\cite[Chapter 2, Theorem~2.4]{Roberts:1980}. Then $\mco\lotimes_{\Lambda_n}C_n$ is represented by the complex $\mco\otimes_{\Lambda_n}C_n$, which is also free and minimal as an $\mco$-complex. Thus we can assume  $\alpha$ is a morphism of $\mco$-complexes:
\[
\alpha_n \colon \mco\otimes_{\Lambda_n}C_n \longiso C_0\,.
\]
Since this is a quasi-isomorphism between minimal free complexes, it is an isomorphism; see~\cite[Chapter 2, Theorem~2.4]{Roberts:1980}. For any open ideal $\fa\subseteq S_\infty$ set
\[
C(\fa,\infty) \colonequals \left(\prod_{n=1}^\infty S_\infty/\fa \otimes_{\Lambda_n} C_n \right)_{\fp}
\]
where as before when $I_n \not\subseteq \fa$ the tensor product in question is assumed to be $0$.

Since $C_n$ is a minimal free complex over $\Lambda_n$, for each integer $i$ one has
\begin{align*}
\rank_{\Lambda_n}(C_n)_i 
	&= \rank_k \hh_i(k\otimes_{\Lambda_n} C_n) \\
	&= \rank_k \hh_i(k \otimes_{\mco} (\mco\otimes_{\Lambda_n} C_n)) \\
	&=\rank_k \hh_i(k\otimes_{\mco}C_0)
\end{align*}
The last equality holds because $\mco \otimes_{\Lambda_n}C_n \cong C_0$, via $\alpha_n$.  Thus the integers
\[
c_i\colonequals \rank_{\Lambda_n}(C_n)_i
\]
are independent of $n$. It follows that the complex $S_\infty/\fa \otimes_{\Lambda_n} C_n$ over $S_\infty/\fa$ is minimal, free and  for each $i\in \mathbb Z$ and all $n$, one has
\[
\rank_{S_\infty/\fa}(S_\infty/\fa  \otimes_{\Lambda_n} C_n)_i = c_i\,.
\]
In particular, the modules $(S_\infty/\fa \otimes_{\Lambda_n}C_n)_i$ are all finite of bounded cardinality. It follows that $C(\fa,\infty)$ is a minimal complex of finite free modules over $S_\infty/\fa$  and that for infinitely many $n$ there are isomorphisms
\[
C(\fa,\infty)\cong S_\infty/\fa \otimes_{\Lambda_n}C_n\,.
\]
Thus $\rank_{S_\infty/\fa}C(\fa,\infty)_i = c_i$ for each $i\in\mathbb Z$.  As in \cite[Proposition 6.4.10]{ACC+:2018} one gets 
\[
S_\infty/\fa  \otimes_{S_\infty/\fb} C(\fb,\infty) \cong C(\fa,\infty)
\]
for any $\fb\subseteq \fa$. Set
\[
C_\infty\colonequals \lim_s C \left(\fm_{S_\infty}^s,\infty\right).
\]
By \cite[Lemma 2.13]{KT} this is a minimal complex of $S_\infty$-modules with 
\[
S_\infty/\fa \otimes_{S_\infty} C_\infty \cong C(\fa,\infty) \quad\text{for all $\fa$,}
\]
and for each $i\in\mathbb Z$ we have 
\[
 \hh_i(C_\infty) \cong \lim_s \hh_i\left(C\left(\fm_{S_\infty}^s,\infty\right)\right)\,.
 \]
 Also (the proof of) \cite[Lemma 2.13(3)]{KT} implies that the natural map 
 \[
 \End_{\dcat{S_\infty}}(C_\infty)\lra\lim_\fa\End_{\dcat{S_\infty/\fa}}(C(\fa,\infty))
 \]
 is an isomorphism. 

The structure maps $R_\infty\xrightarrow{\pi_n}R_n \to \End_{\dcat{\Lambda_n}}(C_n)$ induce maps 
\[
R_\infty \lra R(s,\fa,\infty)\lra \End_{\dcat{S_\infty/\fa}}(C(\fa,\infty))
\]
for all $s\gg 0$, compatible with the maps 
\[
\End_{\dcat{S_\infty/\fb}}(C(\fb,\infty))\lra \End_{\dcat{S_\infty/\fa}}(C(\fa,\infty))
\]
for $\fb\subseteq \fa$. Thus these maps induce a map 
\[
\iota_\infty\colon R_\infty\lra \End_{\dcat{S_\infty}}(C_\infty)\,.
\]
As the maps 
\[
S_\infty\lra R_n \lra \End_{\dcat{\Lambda_n}}(C_n)
\]
induce the natural action of $S_\infty$ on $C_n$ it follows that the map 
\[
S_\infty\xrightarrow{\varphi_\infty}R_\infty\lra \End_{\dcat{S_\infty}}(C_\infty)
\]
induces the natural action of $S_\infty$ on $C_\infty$. Thus the $S_\infty$ action on $C_\infty$ extends to a derived $R_\infty$-action.
Set
\[
\patch(\perfs C) \colonequals \hh_{d}(C_\infty) \quad \text{in $\operatorname{Mod}{R_\infty}$.}
\]
This is finitely generated as an $S_\infty$-module, for $C_\infty$ is a finite free complex over $S_\infty$, and hence also as an $R_\infty$-module. Thus $\patch(\perfs C)$ is in $\rmod{R_\infty}$. 

We  check that this assignment is functorial on $\PatchSys$. Without loss of generality, in the remainder of the proof we assume $d=0$, to ease up notation. Let $f\colon \perfs C\to\perfs D$ be a morphism in $\PatchSys$. As before we assume each $D_n$ is a finite free $\Lambda_n$-complex, and that $\mco\otimes_{\Lambda_n}D_n\iso D_0$ is an isomorphism of $\mco$-complexes. Moreover, we assume that the maps $f_n\colon C_n\to D_n$ are morphisms of complexes; these are uniquely defined up to homotopy, and we leave it to the reader to check that the patching process preserves these homotopies.

For each open ideal $\fa\subseteq S_\infty$ the $f_n$ induce maps
\[
f_n \colon C_n \lra D_n\quad\text{and} \quad 
f_{\fa,n}\colon S_\infty/\fa \otimes_{\Lambda_n} C_n \lra S_\infty/\fa \otimes_{\Lambda_n} D_n\,,
\]
and hence a map $f_{\fa,\infty} \colon C(\fa,\infty)\to D(\fa,\infty)$. By construction
\[
S_\infty/\fa \otimes_{S_\infty/\fb} f_{\fb,n} = f_{\fa,n}\quad\text{for all $\fb\subseteq \fa$ and $n$},
\]
so  $S_\infty/\fa \otimes_{S_\infty/\fb} f_{\fb,\infty}  = f_{\fa,\infty}$. By \cite[Lemma 2.13(3)]{KT}, for each $s$ the map $f_{\fm_{S_\infty}^s,\infty}$ determine a unique, up to homotopy, map $f_\infty\colon C_\infty\to D_\infty$ satisfying 
\[
 S_\infty/\fm_{S_\infty}^s \otimes_{S_\infty} f_\infty  = f_{\fm_{S_\infty}^s,\infty}
\]
and hence $S_\infty/\fa \otimes_{S_\infty} f_\infty = f_{\fa,\infty}$. We can now define the $S_\infty$-module morphism 
\[
\patch(f) \colonequals \hh_{0}(f_\infty)\colon \hh_{0}(C_\infty)\lra \hh_{0}(D_\infty),.
\]
This map depends only on the homotopy class of $f_\infty$,  so does not depend on the choices of $f_n$. It follows that $f\mapsto \patch(f)$ is functorial.

As the maps $f_n$ commute with the structure maps
\[
 R_n\to \End_{\dcat{\Lambda_n}}(C_n)\quad\text{and}\quad  R_n\to \End_{\dcat{\Lambda_n}}(D_n)
\]
the map $f_\infty\colon C_\infty\to D_\infty$ commutes with the maps 
\[
R_\infty\lra \End_{\dcat{S_\infty}}(C_\infty)\quad\text{and}\quad R_\infty\lra \End_{\dcat{S_\infty}}(D_\infty)\,,
	\]
and so $\patch(f)$ is in fact an $R_\infty$-module homomorphism. It follows that $\patch$  defines a functor $\PatchSys\to \mathrm{mod}\, {R_\infty}$.

Now we verify claims (1)--(5). Recall that we have assumed $d=0$.

(1) Each  $C_\infty$ is a finite free $S_\infty$-complex with a derived $R_\infty$-action. For each $n$ one has $(C_n)_i = 0$ for $i\not\in [0,\ell_0]$, thus $(C_\infty)_i = 0$ for the same range of values of $i$. Since $\ell_0=\dim S_\infty - \dim R_\infty$, we can apply  Lemma~\ref{le:acyclicity} below to deduce that the $R_\infty$-module $\patch(\perfs C)$ is maximal Cohen--Macaulay, as claimed.

Let $R'$ denote the image of $R_\infty$ in $\End_{S_\infty}(\patch(\perfs C))$. Then $\dim R'=\dim R_\infty$, as the $R_\infty$-module $\patch(\perfs C)$ is maximal Cohen--Macaulay. Moreover, $R'$ is a finite $S_\infty$-algebra. At this point, one can readily adapt the proof of \cite[Lemmas~3.6 and 3.7]{Bockle/Khare/Manning:2021a} to deduce that the lift $S_\infty \to R_\infty$ of the map $S_\infty\to R'$ can be chosen to ensure that $R_\infty$ is finite as an $S_\infty$-module.

(2) Take any integer $s\ge 0$ and consider the open ideal $\fa_s \colonequals \fn+\varpi^s S_\infty$. Then
\begin{align*}
\patch(\perfs C)/\fn\patch(\perfs C)\otimes_{\mco}\mco/(\varpi^s) 
&= 
\patch(\perfs C)/{\fa_s}\patch(\perfs C) \\
& = \hh_{0}(S_\infty/{\fa_s} \lotimes_{S_\infty}  \patch(\perfs C) )\\
&\cong \hh_{0}(S_\infty/{\fa_s} \otimes_{S_\infty} C_\infty) \\
& \cong \hh_{0}(C(\fa_s,\infty)).
\end{align*}
But now as $\mco\otimes_{\Lambda_n} C_n \cong C_0$, in $\dcat{\mco}$ one gets
\begin{align*}
C(\fa_s,\infty) 
&= \left(\prod_{n=1}^\infty  S_\infty/{\fa_s} \otimes_{\Lambda_n} C_n \right)_{\fp} \\
& = \left(\prod_{n=1}^\infty
	 S_\infty/{\fa_s} \otimes_{\mco}(\mco \otimes_{\Lambda_n} C_n) \right)_{\fp} \\
&= \left(\prod_{n=1}^\infty
S_\infty/{\fa_s} \otimes_{\mco} C_0 \right)_{\fp} \\
& = \left(\prod_{n=1}^\infty
\mco/(\varpi^s)\otimes_{\mco}  C_0 \right)_{\fp} \\
& \cong \mco/(\varpi^s) \otimes_{\mco} C_0\,.
\end{align*}
Combining the computations above, one gets isomorphisms
\[
\patch(\perfs C)/\fn\patch(\perfs C)\otimes_{\mco}\mco/(\varpi^s) 
\cong \hh_{0}(\mco/(\varpi^s) \otimes_{\mco} C_0 )
 \cong \mco/(\varpi^s) \otimes_{\mco} \hh_{0}(C_0) \,.
\]
Taking inverse limits gives the desired isomorphism $\patch(\perfs C)/\fn\patch(\perfs C)\cong \hh_{0}(C_0)$.

(3) The fact that $\hh_i(C_\infty) = 0$ for $i\ne 0$ also implies that 
\[
C_\infty\simeq  \hh_{0}(C_\infty) =  \patch(\perfs C) \quad \text{in $\rdcat{S_\infty}{R_\infty}$.}
\]
The discussion above thus yields
\[
(S_\infty/\fa)  \lotimes_{S_\infty} \patch(\perfs C) \simeq 
(S_\infty/\fa) \otimes_{S_\infty} C_\infty \cong C(\fa,\infty)\cong  (S_\infty/\fa) \otimes_{\Lambda_n} C_n
\]
for infinitely many $n$, proving (2).

(4) This now follows from the above observation that the natural map 
\[
\End_{\dcat{S_\infty}}(C_\infty)\lra \lim_\fa\End_{\dcat{S_\infty/\fa}}(C(\fa,\infty))
\]
is an isomorphism.

(5) If $f_n =\pi_n(\tau)$ in $\End_{\dcat{\Lambda_n}}(C_n)$ for all but finitely many $n$, then for any open ideal $\fa$ the map $f_{\fa,\infty}$ is the image of $\tau$ under the map 
\[
R_\infty\lra R(s,\fa,\infty)\lra \End_{\dcat{S_\infty/\fa}}(C(\fa,\infty)) = \End_{{\cat K}(S_\infty/\fa)}(C(\fa,\infty))\,,
\]
and so by the construction of $\patch(f)$ we get $f_\infty = \iota_\infty(\tau)$, which gives $\patch(f) = \tau$ by the definition of the $R_\infty$-module structure on $\patch(\perfs C)$.
\end{proof}
\end{chunk}

The result below was used in proof of Theorem~\ref{th:ultrapatching}(1); see  \cite[Lemma 6.2]{Calegari/Geraghty:2018}. 
The basic argument goes to the acyclicity criteria of Buchsbaum and Eisenbud~\cite[Proposition~1.4.12]{Bruns/Herzog:1998}, and Peskine and Szpiro~\cite[1.4.24]{Bruns/Herzog:1998}. 

\begin{lemma}
\label{le:acyclicity}
Let $S$ be a Cohen--Macaulay local ring with residue field $k$. If $C$ is an $S$-complex in $\dbcat S$ admitting a derived action of a noetherian local $S$-algebra $R$ with $\dim R\le \dim S$ and 
\[
\hh_i(k\lotimes_SC)=0\quad \text{for $i\notin [0,\dim S - \dim R]$}\,,
\]
then $C\iso \hh_0(C)$ and the $R$-module $\hh_0(C)$ is maximal Cohen--Macaulay.
\end{lemma}

\begin{proof}
Given that $C$ is in $\dbcat S$ and the hypothesis on $\hh_i(k\lotimes_SC)$ replacing $C$ by its minimal free resolution, we can assume $C$  is a finite free complex with $C_i=0$ for $i\notin [0,\dim S - \dim R]$. Let $R'$ denote the image of $R$ in $\End_{\dcat S}(C)$. Then $R'$ is a finite $S$-algebra and $C$ has a derived $R'$-action. Thus the $S$-complex $\Hom_S(C,S)$ inherits a derived $R'$-action so each $\hh_i(\Hom_S(C,S))$ is an $R'$-module. Since $R'$ is a finite $S$-algebra,  for each $i$ we get
\begin{align*}
\dim_S \hh_i(\Hom_S(C,S))  \le \dim_S R' =  \dim R' \le \dim R\,.
\end{align*}
For the next step, it is convenient to use the notion of dimension of the complex $\Hom_S(C,S)$, defined as follows:
\[
\dim_S \Hom_S(C,S) 
\colonequals \max\{\dim_S \hh_i(\Hom_S(C,S))-i\mid  i\in [\dim R - \dim S, 0] \}\,.
\]
See \cite[(2)]{Christensen/Iyengar:2021} and references therein.  Since the complex $\Hom_S(C,S)$ is zero outside the range $[\dim R - \dim S,0]$, one immediately gets the inequality below:
\[
\dim_S \Hom_S(C,S) \le  \dim S\,.
\]
This gives the inequality on the right in:
\[
0\le \max\{i\mid \hh_i(C)\ne 0\} = \dim_S\Hom_S(C,S) - \dim S\le 0\,.
\]
The equality is from \cite[Proposition~6]{Christensen/Iyengar:2021}. It follows that  $\hh_i(C)=0$ for $i>0$, that is to say, $C$ is a free resolution of  $\hh_0(C)$ over $S$. Then the Auslander--Buchsbaum formula \cite[Theorem~1.3.3]{Bruns/Herzog:1998} implies the inequality below:
\begin{align*}
\depth_R \hh_0(C) 
	&= \depth_{R'}\hh_0(C) \\
	& = \depth_S \hh_0(C) \\
	& \ge \dim S - (\dim S - \dim R) = \dim R\,.
\end{align*}
The equalities hold because the maps $R\to R'$ and $S\to R'$ are finite. Thus the $R$-module $\hh_0(C)$ is  maximal Cohen--Macaulay.
\end{proof}

The result below on derived action is used in the proof of Proposition \ref{prop:H0 wt 1}. 

\begin{lemma}
\label{lem:faithful-H0}
Let $\Lambda$ be a noetherian local ring with residue field $k$  and fix $C$ in  $\dbcat{\Lambda}$ such that $\hh_i(k\lotimes_\Lambda C) = 0$ for $i\ne 0,1$ and $\mathrm{projdim}_\Lambda\hh_0(C)\le 1$. Set $C^\dagger\colonequals \RHom_{\Lambda}(C,\Lambda)[1]$. 

If $\varphi\in \End_{\dcat{\Lambda}}(C)$ is such that $\hh_0(\varphi)= 0= \hh_0(\varphi^\dagger)$, then $\varphi = 0$ in $\dbcat{\Lambda}$.
\end{lemma}

\begin{proof}
Given the hypotheses on $C$, replacing it by its minimal free resolution, we can assume the $\Lambda$-module $C_i$ is finite free for each $i$, and equal to $0$ for $i\ne 0,1$. Thus $\varphi$ can be represented as a pair $(\varphi_1,\varphi_0)$ where $\varphi_i\colon C_i\to C_i$ are $\Lambda$-linear maps satisfying $\varphi_0 d = d\varphi_1$, where $d\colon C_1\to C_0$ is the differential on $C$.  Moreover $\varphi^\dagger$ can be realized as the morphism $\Hom_{\Lambda}(\varphi,\Lambda)[1]$, and hence is represented by the pair $(\varphi_0^*,\varphi_1^*)$. Moreover $\varphi \cong \varphi^{\dagger\dagger}$ so $\varphi$ is zero in $\dcat{\Lambda}$ if and only if so is $\varphi^\dagger$.

The hypothesis $\hh_0(\varphi)=0$ implies that there is an map $h\colon C_0\to C_1$ of $\Lambda$-modules with $\varphi_0= dh$.  Replacing $\vf$ by the homotopy-equivalent morphism $\vf + [h,d]$  one can assume $\varphi_0=0$. Thus the morphism $\vf$ factors as
\[
\begin{tikzcd}
0 \arrow[r] &C_1 \arrow[d, "j'"] \arrow[r,"d"] & C_0 \arrow[d]\arrow[r] & 0 \\
0 \arrow[r] &\hh_1(C) \arrow[d, hookrightarrow, "\iota"] \arrow[r] & 0 \arrow[d] \arrow[r] & 0 \\
0 \arrow[r] &C_1 \arrow[r, "d"] & C_0 \arrow[r] & 0 
\end{tikzcd}
\]
with $\iota j=\varphi_1$. Since  $\mathrm{projdim}_\Lambda\hh_0(C)\le 1$ the $\Lambda$-module $\Coker(\iota) \cong \mathrm{Image}(d)$ is free, so $\iota$ is split-injective, and in particular $\hh_1(C)$ is free as well. It follows that $\iota^*$ is surjective. Applying $(-)^*$ to the diagram above yields the diagram of complexes of $\Lambda$-modules
\[
\begin{tikzcd}
0 \arrow[r] &C_0^* \arrow[d] \arrow[r,"d^*"] & C_1^* \arrow[d, twoheadrightarrow, "\iota^*" ]\arrow[r] & 0 \\
0 \arrow[r] & 0 \arrow[d] \arrow[r] & \hh_1(C)^* \arrow[d, "j^*" ] \arrow[d] \arrow[r] & 0 \\
0 \arrow[r] &C_0^* \arrow[r, "d^*"] & C_1^*\arrow[r] & 0 
\end{tikzcd}
\]
Since $j^*i^*=\vf_1^*$ and $\iota^*$ is surjective, the hypothesis that $\hh_{0}(\varphi^\dagger)=\hh_{-1}(\varphi^*)=0$ implies that $\hh_{-1}(j^*)=0$, and since $\hh_1(C)^*$ is free, $j^*$ lifts through $d^*$, that is to say, $j^*$ is homotopic to 0. Thus $\vf^\dagger$ is homotopic to zero, as claimed. 
\end{proof}

\section{Duality}
\label{se:duality}
In this section we prove that patching commutes with duality, in a suitable sense. The setup and notation is as in Section~\ref{se:patching}. In particular, we consider local rings
\[
\mco\pos{\bs y}\,, \quad S \colonequals \mco\pos{\bs y,\bs w}\,, \quad \text{and $R_\infty$.}
\]
All rings considered are quotients of these local rings, and hence possess dualizing complexes. We normalize them as in Section~\ref{se:CM-modules}, making them unique up to isomorphism in the appropriate derived category. Recall that $\ddual A$ denotes the dualizing complex of a local ring $A$ and $\ddual A(-)$ the duality functor; see \eqref{eq:dagger}.

\begin{chunk}
\label{ch:ddual-properties}
The following observations about dualizing complexes are used implicitly.
\begin{enumerate}[\quad\rm(1)]
\item
If $A$ is Gorenstein, for $M\in\dbcat A$ and free resolution $F\iso M$, one has
\[
\ddual A(M) \cong \Hom_A(F,A)[\dim A] \qquad\text{in $\dcat A$.}
\]
In particular, when $M$ is perfect, so is $\ddual A(M)$.
\item 
If $A\to B$ is a finite map, then  $\ddual A(M) \cong  \ddual B(M)$ for $M\in\dbcat B$.
\item
If $A\to B$ is a flat map with $A,B$ Gorenstein, and $M$ is in $\dbcat A$, then
\[
\ddual A(M)\otimes_AB \cong \ddual B(B\otimes_AM)[\dim A-\dim B] \quad\text{in $\dbcat B$.}
\]
\end{enumerate}
Indeed, $A$ Gorenstein means $\ddual A\cong A[\dim A]$---see \cite[\href{https://stacks.math.columbia.edu/tag/0DW7}{Tag 0DW7}]{stacks-project}---so (1) follows from isomorphisms
\[
\RHom_A(M,\ddual A) \cong \Hom_A(F,\ddual A) \cong \Hom_A(F,A)[\dim A]\,.
\]
As to (2), since $\RHom_A(B,\ddual A)\cong\ddual B$ in $\dcat B$, by 
\cite[\href{https://stacks.math.columbia.edu/tag/0AX1}{Tag 0AX1}]{stacks-project},  one has
\[
\RHom_A(M,\ddual A) \cong \RHom_B(M \RHom_A(B,\ddual A))\cong \RHom_B(M, \ddual B)\,.
\]
Part (3) is a direct consequence of the isomorphism:
\[
\RHom_A(M,A) \otimes_A B \cong \RHom_B(B\otimes_AM, B) \,,
\]
which is valid because $M$ is in $\dbcat A$ and $B$ is flat as an $A$-module.
\end{chunk}

\begin{chunk}
\label{ch:patch-duality}
Let  $\perfs C \colonequals (C_n, \alpha_n)_{n\geqslant 0}$ be a patching system over $(\Lambda_n)$; see \ref{ch:pasy}.  The \emph{dual} of $\perfs C$ is the patching system ${\perfs C}^{\dagger} \colonequals (C^{\dagger}_n,\alpha^{\dagger}_n)_{n\geqslant 0}$ where 
\[
C^{\dagger}_n \colonequals \RHom_{\mco}(C_n, \mco) [2d+\ell_0] = \ddual{\mco}(C_n)[2d+\ell_0-1]
\]
Since  $\Lambda_n$ is Gorenstein and $\dim\Lambda_n=1$, one gets
\[
C^{\dagger}_n = \ddual{\Lambda_n}(C_n)[2d+\ell_0-1] \cong  \RHom_{\Lambda_n}(C_n, \Lambda_n) [2d+\ell_0]\,.
\]
In particular,  $C^{\dagger}_n$ is a perfect complex over $\Lambda_n$, and hence it satisfies condition (1) in \ref{ch:pasy}. The suspension by $2d+\ell_0$ ensures that condition (2)  is also satisfied. Since duals and suspension are functors on the category of derived $R_n$-complexes, $C^{\dagger}_n$ satisfies (3) as well. Finally, a standard computation yields that applying $\ddual {\mco}(-)$ to $\alpha_n$ yields an isomorphism
\[
\ddual{\mco}(\alpha_n)\colon \ddual {\mco}(C_0) \longiso \mco\lotimes_{\Lambda_n} \ddual{\Lambda_n}(C_n)
\]
compatible with derived $R_n$-actions. Setting $\alpha^{\dagger}_n$ to be the $(2d+\ell_0)$ suspension of the inverse of this morphism completes the definition of the ${\perfs C}^{\dagger}$.

It is clear that $(-)^\dagger$ defines a contravariant functor on $\PatchSys$, and that the equivalence $(-) \iso \ddual {\Lambda_n}^2(-)$ induce a natural isomorphism in $\PatchSys$:
 \[
 \perfs C \longiso ({\perfs C}^{\dagger})^{\dagger}
 \]
\end{chunk}

The statement below is the desired result on duality. The argument is an extension of that for \cite[Proposition 4.10]{Manning}. Here $(-)^{\vee}$ denotes the duality functor on Cohen--Macaulay modules defined in \eqref{eq:dual-cm}. Applying it is justified for $\patch(-)$ takes values in maximal Cohen--Macaulay $R_\infty$-modules by the Theorem~\ref{th:ultrapatching}.

\begin{theorem}
\label{th:duality}
For $\perfs C$ in $\PatchSys$, there is an natural isomorphism of $R_\infty$-modules 
\[
\patch({\perfs C}^{\dagger})\cong \patch(\perfs C)^{\vee}\,.
\]
This is functorial in $\perfs C$, in that, if $f\colon \perfs C\to \perfs D$ is a morphism in $\PatchSys$ then 
\[
\patch(f^{\dagger})= \patch(f)^{\vee}\,.
\]
\end{theorem}

\begin{proof}
To begin with, given the discussion in \ref{ch:patch-duality}, we can assume we are already in the framed situation, $j=0$, as in the proof of Theorem~\ref{th:ultrapatching}. Moreover, we can assume the $\Lambda_n$-complex $C_n$ is finite free and concentrated in degrees $[0,\ell_0]$ for each $n$. The functor $(-)^\dagger$ preserves these properties. Let 
\[
\patcht \colon \PatchSys\lra \rdcat{S_\infty}{R_\infty}
\]
be the functor $\perfs C\mapsto C_\infty$ and  $f\mapsto f_\infty$ constructed in the proof of Theorem \ref{th:ultrapatching}, so that $\patch(\perfs C) = \hh_{0}(\patcht(\perfs C))$. It suffices to verify that there is an isomorphism
\[
\patcht ({\perfs C}^{\dagger})
	 \longiso  \ddual{S_\infty}(\patcht(\perfs C))[-\dim R_\infty] 
\]
in $\dcat{S_\infty}$, functorial in $\perfs C$, and compatible with $R_\infty$-actions. Indeed, given this isomorphism one gets
\begin{align*}
\patch ({\perfs C}^{\dagger})
	&=\hh_0(\patcht ({\perfs C}^{\dagger})) \\
	&\cong \hh_{\dim R_\infty}(\ddual{S_\infty}(\patcht(\perfs C))) \\
	&\cong \hh_{\dim R_\infty}(\ddual{S_\infty}(\patch(\perfs C)))\\
	&=\patch(\perfs C)^\vee
\end{align*}
keeping in mind that $\patch(\perfs C)$ is Cohen--Macaulay of dimension $\dim R_\infty$.

As to the isomorphism above, fix $\perfs C = (C_n,\alpha_n)_{n\geqslant 0}$ in $\PatchSys$. For any open ideal $\fa\subseteq S_\infty$ containing the ideal $\Ker(S_\infty\twoheadrightarrow  \Lambda_n)$ the complex $S_\infty/\fa \lotimes_{\Lambda_n}C_n $ of $S_\infty/\fa$-modules is finite free,  so using the properties of duality listed in \ref{ch:ddual-properties}, one gets a natural isomorphism
\[
S_\infty/\fa \lotimes_{\Lambda_n} \ddual{\Lambda_n}(C_n) \longiso 
\RHom_{S_\infty/\fa}(S_\infty/\fa \lotimes_{\Lambda_n}C_n, S_\infty/\fa)[1]
\]
respecting the action of $R_\infty$ on both sides.  Here we are using the fact that $\Lambda_n$ is Gorenstein and  $C_n$ is a perfect $\Lambda_n$-complex. Similarly there is an isomorphism
\[
S_\infty/\fa \lotimes_{S_\infty} \ddual{S_\infty}(\patcht(\perfs C))\longiso
\RHom_{S_\infty/\fa}(S_\infty/\fa \lotimes_{S_\infty}\patcht(\perfs C), S_\infty/\fa)[\dim S_\infty]
	\]	
again compatible with the action of $R_\infty$ and functorial in $\perfs C$. This uses the fact that $S_\infty$ is Gorenstein and that the $S_\infty$-complex $\patcht(\perfs C)$ is perfect. Then, for any open ideal $\fa\subseteq S_\infty$, Theorem \ref{th:ultrapatching}(3) gives an isomorphism
\begin{align*}
S_\infty/\fa \lotimes_{S_\infty} \patcht(\perfs C^\dagger)
		&\cong S_\infty/\fa  \lotimes_{\Lambda_n} C_n^{\dagger} \\
		&=  S_\infty/\fa  \lotimes_{\Lambda_n} \ddual{\Lambda_n}(C_n)[\ell_0-1]\\
		&\cong \RHom_{S_\infty/\fa}(S_\infty/\fa \lotimes_{\Lambda_n}C_n, S_\infty/\fa)[\ell_0]\\
	    &\cong \RHom_{S_\infty/\fa}(S_\infty/\fa \lotimes_{S_\infty}\patcht(\perfs C), S_\infty/\fa)[\ell_0]\\
		&\cong S_\infty/\fa \lotimes_{S_\infty}\ddual{S_\infty}(\patcht(\perfs C))[-\dim R_\infty]
	\end{align*}
for infinitely many $n$, again, compatible with the action of $R_\infty$ and functorial in $\perfs C$. As $\patcht(\perfs C)$ is a bounded complex of finite free $S_\infty$ modules, it follows from \cite[\href{https://stacks.math.columbia.edu/tag/06XY}{Section 06XY}]{stacks-project} that we have natural isomorphisms
	\begin{align*}
		\patcht(\perfs C^\dagger) 
		&\cong \lim_\fa\left(S_\infty/\fa \lotimes_{S_\infty} \patcht(\perfs C^\dagger) \right)\\
		&\cong \lim_\fa\left(S_\infty/\fa \lotimes_{S_\infty}\ddual{S_\infty}(\patcht(\perfs C))[-\dim R_\infty]\right)\\
		&\cong \ddual{S_\infty}(\patcht(\perfs C))[-\dim R_\infty] 
	\end{align*}
respecting the action of $R_\infty$ and functorial in $\perfs C$. This is as desired.
\end{proof}

\begin{remark}
As an immediate application of Theorem \ref{th:duality}, one can generalize the main result of \cite{Manning}.

Specifically let $F$ be a CM field in which $\ell$ does not ramify, let $D$ be a quaternion algebra over $F$, and let $Y_K$ be the manifold associated to a compact open $K\subseteq D^\times(\mbb{A}_F^\infty)$ (analogous to the manifolds constructed in Section \ref{sec:Hecke}). Assume appropriate analogues of Conjectures \ref{rhobar conj}, \ref{R->T conj} and \ref{vanishing conj}. Let $\fm$ be a non-Eisenstein ideal of the Hecke algebra of $Y_K$, corresponding to a Galois representation $\rhobar_\fm\colon G_F\to \GL_2(k)$, satisfying certain ``Taylor--Wiles conditions.''

Similarly to the construction given in Sections \ref{sec:Hecke} and \ref{sec:main}, one can use the chain complexes for the $Y_K$'s (localized at $\fm$) to produce a patching system ${\perfs C}$ and then apply the functor $\patch$ to produce a maximal Cohen--Macaulay module $M_\infty$ over a ring $R_\infty$. By Theorem \ref{th:ultrapatching}(2), one gets 
\[
M_\infty/\fm_{R_\infty}M_\infty\cong \hh_d(Y_K,\mco)_\fm/\fm_{R_0} \hh_d(Y_K,\mco)_\fm\,,
\]
with $d$ the lowest degree for which $\hh_*(Y_K,\mco)_\fm\ne 0$. Thus determining $M_\infty$ allows us to determine the \emph{multiplicity} of $\rhobar_\fm$ in $\hh_d(Y_K,\mco)$, just as in \cite[Section 4]{Manning}.

Crucially $R_\infty$ is the same ring as the one  in \cite{Manning}, since both rings are determined by local information at finite places, which is not affected by the transition from totally real fields to CM fields (or general number fields). Moreover, by an analogous construction to the one used in Proposition \ref{prop: twisted duality}, one can show that ${\perfs C}^\dagger\cong {\perfs C}$, and so Theorem \ref{th:duality} gives $M_\infty^\vee \cong M_\infty$, just as in \cite[Proposition 4.10]{Manning}.

If $\hh_1(Y_K,E)_{\fm} \ne 0$, then standard multiplicity one theorems for automorphic forms imply that $M_\infty$ has generic rank $1$, and so \cite[Theorem 3.1]{Manning} implies that $\dim M_\infty/\fm_{R_\infty}M_\infty = 2^k$, giving $\dim \hh_1(Y_K,\mco)_{\fm}/\fm_{R_0} = 2^k$, where $k$ is an integer (defined in \cite[Theorem 1.1]{Manning}) depending only the local properties of $\rhobar_\fm$ at the places where the quaternion algebra $D$ ramifies. The method of \cite{Manning} crucially relies on the assumption that $M_\infty$ has generic rank $1$.

This observation is not original to this paper. Frank Calegari has remarked to one of us (JM) that he discovered this result (and Theorem \ref{th:duality}) before we did.

In addition to the multiplicity statement, \cite[Theorem 3.1]{Manning} also implies (again under the assumption that $\hh_1(Y_K,E)_{\fm} \ne 0$) that the natural map 
\[
\tau_{M_\infty}\colon M_\infty \otimes_{R_\infty} M_\infty\lra \omega_{R_\infty}
\]
induced by the self-duality of $M_\infty$ is surjective. In the setting of \cite{Manning}, that is to say, the ``$\ell_0=0$'' case, this implies an important statement about the associated congruence module (see \cite[Theorem 1.2]{Manning} and \cite[Theorem 3.2]{Bockle/Khare/Manning:2021b}). This argument does not generalize to the ``$\ell_0>0$'' case considered here, but it seems plausible that the surjectivity of $\tau_{M_\infty}$  still implies some statement about the homology group $\hh_*(Y_K,\mco)_\fm$.
\end{remark}

\part{Deformation rings and Hecke algebras}
\label{part:deformation rings}
 We now  use our work in Parts \ref{part:ca} and \ref{part:patching and duality}  to prove integral, non-minimal modularity lifting theorems in defect $\ell_0>0$.  We  prove two versions of our main `$R=\TT$' theorem, one in the the case of $\PGL_{2}$ over an arbitrary number field and one in the case of weight one modular forms on a Shimura curve over $\QQ$. Our main argument (given in Section \ref{sec:main}) will be essentially identical in the two cases, so the only difference in the two cases will be in the setup. When necessary, we  refer to these two cases as cases \ref{case:PGL2} and \ref{case:wt1}:

\setcasetag{(PGL2)}
\begin{case}\label{case:PGL2}
This deals with Hecke eigenclasses arising from the cohomology of symmetric manifolds associated to $\PGL_2$ over an arbitrary number field $F$. Our main result in this case is Theorem \ref{th:intro-5}. Most of our work in this case is contingent on conjectures \ref{rhobar conj}, \ref{R->T conj}, \ref{vanishing conj} and \ref{Ihara conj}.
\end{case}
 
\setcasetag{(Wt1)}
\begin{case}\label{case:wt1}
This deals with weight one modular forms on a  Shimura curve over $\mathbb Q$. Our main result in this case is Theorem \ref{th:intro-5b}. As a corollary of this we also obtain an integral Jacquet--Langlands statement, Theorem \ref{JLcorr}, in certain situations. In this case, our results  hold unconditionally. 
\end{case}

 In Section \ref{se:Galois}, we introduce the necessary Galois deformation theory required for our results. We work in a generality which covers both cases for most of this section, and only specialize when necessary. In Section \ref{sec:Hecke} we define the Hecke algebras and complexes involved in case \ref{case:PGL2}, and introduce the conjectures \ref{rhobar conj}, \ref{R->T conj}, \ref{vanishing conj} and \ref{Ihara conj}. In Section \ref{sec:Hecke wt1} we define the Hecke algebras and complexes needed in case \ref{case:wt1}, and prove the analogues of conjectures \ref{rhobar conj}, \ref{R->T conj}, \ref{vanishing conj} and \ref{Ihara conj}. We deduce Theorem \ref{th:intro-5c}  from Theorem \ref{th:intro-5b}. Our main result, Theorem~\ref{R=T theorem}, is proved in Section \ref{sec:main}. It implies Theorem \ref{th:intro-5} in Case \ref{case:PGL2} and Theorem \ref{th:intro-5b} in Case \ref{case:wt1}. 
 
 The work in this part is   illustrative of  the applicability of the commutative algebra developed earlier to number theory. As such, we have not worked in the fullest generality possible. We have indicated  a few other plausible number theoretic applications in the introduction.

\section{Galois deformation theory}
\label{se:Galois}
For the remainder of this paper we fix a prime $\ell>2$ and a finite extension $E/{\mbb Q}_\ell$ with ring of integers $\mco$, uniformizer $\varpi$ and residue field $k\colonequals \mco/\varpi$. Let $F$ denote an arbitrary number field and assume that $\ell$ does not ramify in $F$. Let 
\[
\rhobar\colon G_F\to \GL_2(k)
\]
be a continuous Galois representation with $\rhobar|_{G_{F(\zeta_\ell)}}$ is absolutely irreducible.  For convenience we assume that for each $\sigma\in G_F$ the eigenvalues of $\rhobar(\sigma)\in\GL_2(k)$ lie in $k$ (which can be done by replacing $k$ by its unique quadratic extension if necessary). Let $\psi\colon G_F\to \mco^\times$ be a character lifting $\det\rhobar:G_F\to k^\times$.

While we work in the generality stated above for most of this section, in the remainder of the paper we apply these results only in two special cases, corresponding to the cases \ref{case:PGL2} and \ref{case:wt1} mentioned above.

\begin{itemize}
	\item Case \ref{case:PGL2}:  Here $F$ is an arbitrary number field, $\psi$ is the cyclotomic character $\varepsilon_{\ell}$ and $\rhobar|_{G_{F_v}}$ is flat at all places $v|\ell$.
	\item  Case \ref{case:wt1}:  Here $F\colonequals \QQ$, the character $\psi$  is the Teichmuller lift of $\det\rhobar$ (and hence has finite image), and $\rhobar$ is unramified at $\ell$ and odd.
\end{itemize}

\subsection{Local deformation rings}
\label{ssec:local def}
We introduce the various local (framed) deformation rings needed for the rest of the paper.

Fix a prime $v$ of $F$, and consider the local Galois representation 
\[
\rhobar_v\colonequals \rhobar|_{G_{F_v}}\colon G_{F_v}\to \GL_2(k)
\]
Write $q_v = \Nm(v) = \#(\mco_F/v)$ for the norm of $v$. Let $I_{F_v}\unlhd G_{F_v}$ be the inertia subgroup and $P_{F_v}\unlhd I_{F_v}$ be the wild inertia subgroup.

Let $\varphi_v\in G_{F_v}$ be a lift of  (arithmetic) Frobenius  $\Frob_v$ and  $\sigma_v\in G_{F_v}$ a lift of a topological generator of $I_{F_v}/P_{F_v}$, so that $\varphi_v\sigma_v\varphi_v^{-1} = \sigma_v^{q_v}$ in $G_{F_v}/P_{F_v}$. 

Let $R_v^\square$ be the unrestricted framed deformation ring parameterizing lifts of $\rhobar_v$ with determinant $\psi|_{G_{F_v}}$. Let $\rho_v^\square\colon G_{F_v}\to \GL_2(R_v^\square)$ be the universal lift of $\rhobar_v$.

Suppose first that $v|\ell$. If $\rhobar_v$ is flat, as is $\psi|_{G_{F_v}}$  (which will be the case if $\psi=\varepsilon_{\ell}$), then let $R^{\fl}_v$ be the quotient of $R_v^{\square}$ parameterizing flat deformations of $\rhobar_v$, with determinant $\varepsilon_{\ell}$.  As $\ell$ does not ramify in $F$,  from \cite[Section 2.4.1]{CHT} we get
\[
R_v^{\fl}\cong \mco\pos{x_1,x_2,\ldots,x_{3+[F_v:{\mbb Q}_\ell]}}\,.
\]
If $\rhobar_v$ and $\psi|_{G_{F_v}}$ are unramified at $v$, let $R_v^{\ur}$ be the quotient of $R_v^\square$ parameterizing unramified lifts of $\rhobar_v$ with determinant $\psi|_{G_{F_v}}$. Then we clearly have
\[
R_v^{\ur}\cong \mco\pos{x_1,x_2,x_{3}}.
\]
From now on assume that $v\nmid \ell$. It follows (see for example \cite[Theorem 2.5]{Shotton2}) that $R_v^\square$ is a reduced complete intersection,  flat over $\mco$ of relative dimension $3$.

Also, as $\Ker(\GL_2(R_v^\square)\to \GL_2(k))$ is pro-$\ell$, the natural map $\rho_v^\square(P_{F_v})\to \rhobar_v^\square(P_{F_v})$ is an isomorphism, and so if $\rhobar_v$ is tamely ramified, $\rho^\square_v$ is as well.

Let $R^{\min}_v$ be the quotient of $R_v^\square$ parameterizing ``minimally ramified'' lifts of $\rhobar_v$ (see \cite[Definition 2.4.14]{CHT}). Then by \cite[Lemma 2.4.19]{CHT} we get that 
\[
R_v^{\min}\cong \mco\pos{x_1,x_2,x_3}\,.
\]

From now on we consider the case where $v$ is unramified in $\rhobar$, and $q_v\not \equiv 1\pmod \ell$ and $\rhobar_v(\Frob_v)$ has eigenvalues with ratio $q_v$. Fix a choice $\sqrt{q_v\psi(\varphi_v)^{-1}}\in \mco^\times$ of square root of $q_v\psi(\varphi_v)^{-1}$ in $\mco^\times$, and also use $\sqrt{q_v\psi(\varphi_v)^{-1}}$ to denote the image in $k^\times$. In Case \ref{case:PGL2} we have $\psi = \varepsilon_{\ell}$, so that $q_v\psi(\varphi_v)^{-1} = 1$, and so we take $\sqrt{q_v\psi(\varphi_v)^{-1}}=1$ for convenience.

As $q_v\not\equiv 1 \pmod{\ell}$ the eigenvalues of $\rhobar(\varphi_v)$ are distinct and so up to conjugation, $\rhobar(\varphi_v)$ has the form
\[
\rhobar(\varphi_v)=\left( \begin{array}{cc} \psi(\varphi_v) \epsilon_v \sqrt{q_v\psi(\varphi_v)^{-1}}& 0 \\ 0 & \epsilon_v\sqrt{q_v\psi(\varphi_v)^{-1}}^{-1}\end{array} \right)
\]
for some $\epsilon_v\in \{1,-1\}$. If $q_v\not\equiv \pm 1\pmod{\ell}$, then the choice of $\epsilon_v$ is uniquely determined by $\rhobar(\varphi_v)$. In the case when $q_v\equiv \pm -1\pmod{\ell}$, either $\epsilon_v=1$ or $\epsilon_v=-1$ is possible.

For $\epsilon_v\in\{1,-1\}$ the Steinberg quotient $R_q^{\St(\epsilon_v)}$ is the unique torsion free reduced quotient of $R_v^\square$ characterized by the fact that the $L$-valued points of its generic fiber, for any finite extension $L/E$, correspond to representations conjugate to
\[ \left( \begin{array}{cc} \psi  \chi_v& \ast \\ 0 & \chi_v^{-1}\end{array} \right)\]
where $\chi_v$ is an unramified character with $\chi_v(\varphi_v)=\epsilon_v\sqrt{q_v\psi(\varphi_v)^{-1}}$. If $q_v\not\equiv\pm1\pmod{\ell}$, then there is a unique $\epsilon_v$ for which $R_q^{\St(\epsilon_v)}\ne 0$, while if $q_v\equiv-1\pmod{\ell}$ then $R_q^{\St(1)}$ and $R_q^{\St(-1)}$ are both nonzero.

We define the unipotent quotient $R_v^{{\rm uni}(\epsilon_v)}$ of $R_v^\square$ to be
the unique reduced quotient of $R_v^\square$ such that 
\[
\operatorname{Spec} R_v^{{\rm uni}(\epsilon_v)}=\operatorname{Spec} R_v^{{\rm St}(\epsilon_v)}\cup\operatorname{Spec} R_v^{\rm min}
\]
inside $\operatorname{Spec} R_v^\square$.

The computations of \cite[Section 5]{Shotton} give the following result.

\begin{proposition}
\label{prop:unipotent-deformations}
Assume that $\rhobar_v$ is unramified and $q_v\not\equiv 1\pmod{\ell}$.

If the ratio of the eigenvalues of $\rhobar_v(\varphi_v)$ is not $q_v^{\pm 1}$, then 
\[
R_v^\square = R_v^{\min} \cong \mco\pos{x_1,x_2,x_3}\,.
\]
Now assume that the eigenvalues of  $\rhobar_v(\varphi_v)$ have ratio $q_v$. If $q_v\not\equiv \pm1\pmod{\ell}$, let $\epsilon_v$ be the unique element of $\{1,-1\}$ for which $R_q^{\St(\epsilon_v)}\ne 0$. If $q_v\equiv -1\pmod{\ell}$, let $\epsilon_v$ be either $1$ or $-1$.

We have an isomorphism $R_v^{{\rm uni}(\epsilon_v)}\cong \mco\pos{A_v,B_v,X_v,Y_v}/(A_vB_v)$ and the corresponding universal representation sends:
\begin{align*}
\sigma_v &\mapsto 
U^{-1}
\begin{pmatrix}
	1&A_v\\
	0&1
\end{pmatrix}
U \qquad \text{where $U\colonequals \begin{pmatrix}
1&X_v\\
Y_v&1
\end{pmatrix}$, and } \\
		\varphi_v &\mapsto 
U^{-1}\begin{pmatrix}
	\psi_v(\varphi_v)\epsilon_v\sqrt{q_v\psi(\varphi_v)^{-1}}(1+B_v)^{-1}&0\\
	0&\epsilon_v\sqrt{q_v\psi(\varphi_v)^{-1}}^{-1}(1+B_v)
\end{pmatrix}
U
\end{align*}
We have
and $R^{\min}_v = R_v^{{\rm uni}(\epsilon_v)}/(A_v)=\mco\pos{B_v,X_v,Y_v}$ and $R_v^{\St(\epsilon_v)} = R_v^{{\rm uni}(\epsilon_v)}/(B_v) = \mco\pos{A_v,X_v,Y_v}$. If $q_v\not\equiv\pm 1\pmod{\ell}$, then $ R_v^{{\rm uni}(\epsilon_v)} = R_v^\square$.\qed
\end{proposition}

\begin{remark}
In our main results (e.g., Theorem \ref{R=T theorem}) we restrict ourselves to  allowing (non-minimal)  ramification of lifts of $\rhobar$  only at the places  in Proposition \ref{prop:unipotent-deformations} above (these are our ``level raising’’ primes); this is done   mainly for simplicity. We should be able to treat unrestricted ramification at places  $v$ such that 
 $q_v\equiv 1\pmod{\ell}$  and $\rhobar$ is unramified at $v$ with  $\rhobar(\Frob_{v})$  having  distinct eigenvalues, or unipotent ramification at places  $v$ such that 
 $q_v\equiv 1\pmod{\ell}$  and $\rhobar$ is unramified at $v$ with  $\rhobar(\Frob_{v})$  of order divisible by $\ell$,    without  too much trouble. (It would be more work to allow unipotent ramification at  places $v$ that are trivial  for $\rhobar$.)
 
 We note that if $\rhobar$ is odd at a real place of $F$, with corresponding (conjugacy class of)  complex conjugation $c \in G_F$, then any $v$ such that $\rhobar(\Frob_v)$ is conjugate to  $\rhobar(c)$ is a level raising prime (with $q_v\equiv-1\pmod\ell$) in our sense. 
 If the projective image of $\rhobar$ is ${\rm PSL_2}(k)$ (for $|k| \geq 5$)  and $[F(\zeta_\ell):F]>2$ then there are places $v$ such that $\rhobar(\Frob_v)$ has eigenvalues with ratio $q_v$   and $q_v \not\equiv \pm 1\pmod\ell$; these are again level raising primes in our sense.
 \end{remark}

\subsection{Global deformation rings}\label{ssec:global def}

Let $\rhobar\colon G_F\to \GL_2(k)$ be as before. Let $S$ be the set of all finite places $v$ of $F$, such that either $\rhobar$ or $\psi$ is ramified at $v$, or $v|\ell$.  At places $v \in S$ and $ v \nmid\ell$, we either assume that $q_v \not\equiv -1 \pmod {\ell}$, or that $\rhobar|_{I_v}$ is irreducible, or that $\rhobar|_{D_v}$ is reducible. (We allude to this condition as {\it non-vexing} following the use of this terminology in \cite{Calegari}; it avoids considerations of types on the automorphic side to match conditions of minimal ramification on the Galois side.)  Also fix three disjoint finite collections $T$, $\disc$ 
and $Q$ of finite places of $F$ such that for each $v\in T\cup \disc$:
\begin{itemize}
	\item $\rhobar_v$ is unramified.
	\item $v\nmid \ell$ and $q_v \not\equiv 1\pmod{\ell}$.
	\item The ratio of the eigenvalues of $\rhobar_v(\Frob_v)$ is $q_v$.
\end{itemize}
and for each $v\in Q$:
\begin{itemize}
	\item $\rhobar_v$ is unramified.
	\item $q_v \equiv 1\pmod{\ell}$.
	\item $\rhobar_v(\Frob_v)$ has has distinct eigenvalues in $k$.
\end{itemize}
For each $v\in T\cup\disc$ we  fix a $\epsilon_v\in\{1,-1\}$ for which $R_v^{\St(\epsilon_v)}\ne 0$ (and recall that for $q_v\not\equiv \pm 1\pmod{\ell}$, there is a unique such choice). Let $\Sigma$ be any subset of $T$.

In our applications, $\Sigma$ represents a collection of `level raising primes' for $\rhobar$ and $Q$ represents a collection of `Taylor--Wiles primes', which are used in our patching construction. In the case \ref{case:wt1}, $\disc$ is the discriminant of a Shimura curve over $\QQ$, while $\disc=\es$  in the case \ref{case:PGL2}.

The goal of this section is to construct global Galois deformation rings $R_{\Sigma}^{\disc}$ and $R_{\Sigma,Q}^{\disc}$, and their framed variants $R_\Sigma^{\disc,\square}$ and $R_{\Sigma,Q}^{\disc,\square}$.

Let $\Art$ (respectively, $\CNLO$) denote the category of Artinian (respectively, of complete local noetherian) $\mco$-algebras with residue field $k$.

Define a functor $\mcd\colon \Art\to \Sets$ which sends a ring $A\in\Art$ to the set $\mcd(A)$ of equivalence classes of continuous representations $\rho\colon G_F\to \GL_2(A)$ for which $\det\rho = \psi$ and the composition 
\[
G_F\xrightarrow{\rho}\GL_2(A)\onto \GL_2(k)
\]
is equal to $\rhobar\colon G_F\to \GL_2(k)$, where two such representations $\rho,\rho\colon G_F\to \GL_2(A)$ are considered equivalent if $\rho' = \gamma\rho\gamma^{-1}$ for some $\gamma\in 1+M_2(\fm_A)$. 

Define subfunctors $\mcd_\Sigma^{\disc}\subseteq \mcd_{\Sigma,Q}^{\disc}\subseteq \mcd\colon \Art\to \Sets$ such that for any $\rho\in \mcd(A)$, we have $\rho\in \mcd^\disc_\Sigma(A)$ if
\begin{itemize}
\item In Case \ref{case:PGL2}, $\rho$ is flat at $v$ for all places $v|\ell$;
\item In Case \ref{case:wt1}, $\rho$ is unramified at $\ell$;
\item If $v$ is any place of $F$ for which $v\nmid \ell$ and $\rhobar$ is ramified at $v$, then $\rho$ is minimally ramified at $v$;
\item If $v$ is any place of $F$ with $v\not\in S\cup\Sigma\cup \disc$ (including the case $v\in T\sm\Sigma$) then $\rho$ is unramified at $v$.
\item If $v\in \disc$ then $\rho|_{G_{F_v}}$ arises from $R_v^{\St(\epsilon_v)}$. 
\item If $v \in \Sigma$ and $q_v\equiv- 1\pmod{\ell}$ then $\rho|_{G_{F_v}}$ arises from $R_v^{{\rm uni}(\epsilon_v)}$.
\end{itemize}
and $\rho\in \mcd^\disc_{\Sigma,Q}(A)$ if
\begin{itemize}
	\item In Case \ref{case:PGL2}, $\rho$ is flat at $v$ for all places $v|\ell$;
	\item In Case \ref{case:wt1}, $\rho$ is unramified at $\ell$;
	\item If $v$ is any place of $F$ for which $v\nmid \ell$ and $\rhobar$ is ramified at $v$, then $\rho$ is minimally ramified at $v$;
	\item If $v$ is any place of $F$ with $v\not\in S\cup\Sigma\cup\disc\cup Q$ (including the case $v\in T\sm\Sigma$) then $\rho$ is unramified at $v$.
	\item If $v\in \disc$ then $\rho|_{G_{F_v}}$ arises from $R_v^{\St(\epsilon_v)}$. 
	\item If $v \in \Sigma$ and $q_v\equiv- 1\pmod{\ell}$ then $\rho|_{G_{F_v}}$ arises from $R_v^{{\rm uni}(\epsilon_v)}$.
\end{itemize}
Note that for $\rho\in \mcd^\disc_\Sigma(A)$ (respectively, $\rho\in \mcd^\disc_{\Sigma,Q}(A)$) there is no restriction on the ramification of $\rho$ at primes in $\Sigma$ (respectively, at primes in $\Sigma\cup Q$) with $q_v\not\equiv-1\pmod{\ell}$.

Similarly define framed versions of these functors 
\[
\mcd^{\disc,\square}_\Sigma,\mcd^{\disc,\square}_{\Sigma,Q}\colon \Art\lra \Sets
\]
as follows: For $A\in \Art$, let $\mcd^{\disc,\square}_\Sigma(A)$ (respectively, $\mcd^{\disc,\square}_{\Sigma,Q}(A)$) denote the set of tuples $(\rho,(\beta_v)_{v\in S\cup T\cup\disc})$ where $\rho\in \mcd_\Sigma(A)$ (respectively, $\rho\in \mcd_{\Sigma,Q}(A)$) and each $\beta_v$ is an element of $1+M_2(\fm_A)$, where two tuples $(\rho,(\beta_v)_{v\in S\cup T\cup\disc})$ and $(\rho',(\beta'_v)_{v\in S\cup T\cup\disc})$ are equivalent if there is some $\gamma\in 1+M_2(\fm_A)$ satisfying $\rho' = \gamma\rho\gamma^{-1}$ and $\beta_v' = \gamma\beta_v$ for all $v\in S\cup T\cup\disc$.

One may  view  $\beta_v$ as a choice of basis for $A^2$ lifting the standard basis for $k^2$.

It is well known that the functors $\mcd^\disc_\Sigma$, $\mcd^{\disc,\square}_\Sigma$, $\mcd^\disc_{\Sigma,Q}$ and $\mcd_{\Sigma,Q}^{\disc,\square}$ are all pro-representable by rings $R^\disc_\Sigma, R_\Sigma^{\disc,\square}, R^\disc_{\Sigma,Q}$ and $R_{\Sigma,Q}^{\disc,\square}$, respectively, in $\CNLO$. When $\disc$ is either clear from context or empty (as it is in Case \ref{case:PGL2}) we sometimes omit it from our notation, and write these rings as $R_\Sigma, R_\Sigma^{\square}, R_{\Sigma,Q}$ and $R_{\Sigma,Q}^{\square}$ 

By definition we have $\mcd_\Sigma^\disc\subseteq \mcd_{\Sigma,Q}^\disc$ and for any $\Sigma_1\subseteq \Sigma_2\subseteq T$ we have $\mcd_{\Sigma_1}^\disc\subseteq \mcd_{\Sigma_2}^\disc$ and $\mcd_{\Sigma_1,Q}^\disc\subseteq \mcd_{\Sigma_2,Q}^\disc$. These maps on functors induce a commutative diagram
\[
\begin{tikzcd}
	R_{\Sigma_2,Q}^\disc \arrow[twoheadrightarrow]{d} \arrow[twoheadrightarrow]{r} & R_{\Sigma_2}^\disc \arrow[twoheadrightarrow]{d}  \\
	R_{\Sigma_1,Q}^\disc \arrow[twoheadrightarrow]{r} & R_{\Sigma_1}^\disc
\end{tikzcd}
\]
Similarly we have $\mcd_\Sigma^{\disc,\square}\subseteq \mcd_{\Sigma,Q}^{\disc,\square}$ (for any $\Sigma$), $\mcd_{\Sigma_1}^{\disc,\square}\subseteq \mcd_{\Sigma_2}^{\disc,\square}$ and $\mcd_{\Sigma_1,Q}^\square\subseteq \mcd_{\Sigma_2,Q}^{\disc,\square}$ giving rise to a commutative diagram
\[
\begin{tikzcd}
	R_{\Sigma_2,Q}^{\disc,\square} \arrow[twoheadrightarrow]{d} \arrow[twoheadrightarrow]{r} & R_{\Sigma_2}^{\disc,\square} \arrow[twoheadrightarrow]{d}  \\
	R_{\Sigma_1,Q}^{\disc,\square} \arrow[twoheadrightarrow]{r} & R_{\Sigma_1}^{\disc,\square}
\end{tikzcd}
\]
It is important here that the definitions of $\mcd^{\disc,\square}_\Sigma,\mcd^{\disc,\square}_{\Sigma,Q}$ included $\beta_v$'s for all $v\in S\cup T\cup\disc$, including $v\in T\sm \Sigma$ (but not for $v\in Q$). Otherwise we would not be able to compare the functors $\mcd_{\Sigma_1}^{\disc,\square}$, $\mcd_{\Sigma_2}^{\disc,\square}$, $\mcd_{\Sigma_1,Q}^{\disc,\square}$ and $\mcd_{\Sigma_2,Q}^{\disc,\square}$ in this way.

The map $(\rho,(\beta_v)_{v\in S\cup T\cup\disc})\mapsto \rho$ clearly induces maps $R_\Sigma^{\disc}\to R_\Sigma^{\disc,\square}$ and $R_{\Sigma,Q}^{\disc}\to R_{\Sigma,Q}^{\disc,\square}$. It is well-known and easy to see that these maps are formally smooth, and thus induce (non-canonical) isomorphisms 
\[
R_\Sigma^{\disc,\square}\cong R_\Sigma^{\disc}\pos{w_1,\ldots,w_{4|S\cup T|-1}}\quad\text{and}\quad 
R_{\Sigma,Q}^{\disc,\square}\cong R_{\Sigma,Q}^{\disc}\pos{w_1,\ldots,w_{4|S\cup T|-1}}\,,
\]
respecting the commutative diagrams above.

Also for any $v\in S\cup T\cup\disc$, the map $(\rho,(\beta_v)_{v\in S\cup T\cup\disc})\mapsto \beta_v^{-1}\rho|_{G_{F_v}}\beta_v$ induces a map $R_v^{\disc,\square}\to R_{\Sigma,Q}^{\disc,\square}$ and (and hence a map $R_v^{\disc,\square}\to R_{\Sigma,Q}^{\disc,\square}\to R_{\Sigma}^{\disc,\square}$). By the definition of $R_{\Sigma,Q}^{\disc,\square}$, it is easy to see that this induces a map
\begin{equation*}\label{R_loc}
R_{\loc,\Sigma} \colonequals 
\widehat{\bigotimes_{v|\ell}}R^{\fl}_v\widehat{\otimes}
\widehat{\bigotimes_{\substack{v\in S,\\v\nmid\ell}}}R^{\min}_v\widehat{\otimes}
\widehat{\bigotimes_{v\in\Sigma}}R_v^{{\rm uni}(\epsilon_v)}
\widehat{\otimes}
\widehat{\bigotimes_{v\in T\sm\Sigma}} R^{\min}_v
\to R_{\Sigma,Q}^{\square}.
\end{equation*}
in Case \ref{case:PGL2}, and in Case \ref{case:wt1} a map
\begin{equation*}
	R_{\loc,\Sigma}^{\disc} \colonequals 
	\widehat{\bigotimes_{v|\ell}}R^{\ur}_v\widehat{\otimes}
	\widehat{\bigotimes_{\substack{v\in S,\\v\nmid\ell}}}R^{\min}_v\widehat{\otimes}
	\widehat{\bigotimes_{v\in\Sigma}}R_v^{{\rm uni}(\epsilon_v)}
	\widehat{\otimes}
	\widehat{\bigotimes_{v\in T\sm\Sigma}} R^{\min}_v
	\widehat{\otimes}
	\widehat{\bigotimes_{v\in \disc}} R^{\St(\epsilon_v)}_v
	\to R_{\Sigma,Q}^{\disc,\square}.
\end{equation*}
Moreover for $\Sigma_1\subseteq \Sigma_2\subseteq T$ it is easy to see that the natural map $R_{\loc,\Sigma_2}^{\disc}\onto R_{\loc,\Sigma_1}^{\disc}$, given by the quotient map 
\[
R_v^{{\rm uni}(\epsilon_v)}\onto  R_v^{{\rm uni}(\epsilon_v)} /(A_v) =R_v^{\min}
\]
for all $v\in\Sigma_2\sm\Sigma_1, q_v \not\equiv -1 \pmod{\ell}$ gives the commutative diagram
\[
\begin{tikzcd}
	R_{\loc,\Sigma_2}^{\disc} \arrow[twoheadrightarrow]{d} \arrow[twoheadrightarrow]{r} & R_{\Sigma_2,Q}^{\disc,\square} \arrow[twoheadrightarrow]{d} \arrow[twoheadrightarrow]{r} & R_{\Sigma_2}^{\disc,\square} \arrow[twoheadrightarrow]{d}  \\
	R_{\loc,\Sigma_1}^{\disc} \arrow[twoheadrightarrow]{r} &R_{\Sigma_1,Q}^{\disc,\square} \arrow[twoheadrightarrow]{r} & R_{\Sigma_1}^{\disc,\square}
\end{tikzcd}
\]

\subsection{Taylor--Wiles deformations.}
\label{ssec:taylor_wiles_deformations} 
	Suppose $v$ is a prime of $F$ that is unramified in $\rhobar$ and  that $q_v \equiv 1\pmod{\ell}$, and $\overline{\rho}|_{G_{F_v}}$ is unramified. Let $\Delta_v = k(v)^\times(\ell)$ and $\Lambda_v = \mco[\Delta_v]$, and suppose that $\overline{\rho}(\Frob_v)$ has  distinct eigenvalues $\gamma_{v, 1}, \gamma_{v, 2} \in k$.
	Then the universal representation $\rho^{\rm univ}_{\Sigma,Q}:G_F \to \GL_2(R_{\Sigma,Q}^{\disc})$ when restricted to $G_{F_v}$  is of the form
	\[ 
	\chi_{1,v}  \oplus \chi_{2,v}, 
	\]
	for continuous characters $\chi_{1,v}, \chi_{2,v} \colon G_{F_v} \to (R_{\Sigma,Q}^{\disc})^\times$ satisfying the following conditions: we have $(\chi_i)(\Frob_v) \equiv \gamma_{v, i}\pmod{{\fm}_{R_{\Sigma,Q}}}$ and $\chi_i|_{I_{F_v}}$ induces via the Artin map a homomorphism $\Delta_v \to (R_{\Sigma,Q}^{\disc})^\times$ that we denote by the same symbol $\chi_i$.  
	
	Let $\Delta_Q = \prod_{v\in Q}\Delta_v$, and consider the group algebra $\mco[\Delta_Q]$. 
	Then $R_{\Sigma,Q}^{\disc}$ is naturally a $\mco[\Delta_Q]$-algebra via $\prod_v \chi_{1,v}\colon \mco[\Delta_Q] \to R_{\Sigma, Q}^{\disc}$. If $\fa_Q \subset \mco[\Delta_Q]$ is the augmentation ideal, then there is a canonical isomorphism $R_{\Sigma,Q}^{\disc}/(\fa_Q) \cong R_\Sigma^{\disc}$.  We record this as the following proposition.
	
	\begin{proposition}\label{prop: O[Delta]->R}
		For any $\Sigma$, there is a map $\mco[\Delta_Q]\to R_{\Sigma,Q}^{\disc}$ such that 
		\[
		R_{\Sigma,Q}^{\disc}\otimes_{\mco[\Delta_Q]}\mco = R_\Sigma^{\disc}
		\]
		via the natural map $R_{\Sigma,Q}^{\disc}\onto R_\Sigma^{\disc}$, and for any $\Sigma_1\subseteq \Sigma_2$, the map $R_{\Sigma_2,Q}^{\disc}\onto R_{\Sigma_1,Q}^{\disc}$ is an $\mco[\Delta_Q]$-algebra homomorphism. \qed
	\end{proposition}

\section{Hecke algebras for $\PGL_{2}/F$}\label{sec:Hecke}

In this section we work exclusively in the case \ref{case:PGL2}. As before, $F$ is an arbitrary number field such that the prime $\ell$ does not ramify in $F$. Let $r_1$ and $r_2$ respectively denote the number of real and complex places of $F$. For the remainder of this section we use the notation and results of Sections \ref{se:patching} and \ref{se:duality} with $d=r_1+r_2$ and $\ell_0 = r_2$. In particular, for any $C\in \dcat{\mco}$, we  set
\[
C^{\dagger} \colonequals \RHom_\mco(C,\mco)[2d+\ell_0] = \RHom_\mco(C,\mco)[2r_1+3r_2].
\]
In the case when $C$ is a derived $A$-complex for an $\mco$-algebra $A$, this induces a derived $A$-complex structure on $C^{\dagger}$, and so we may view $(-)^{\dagger}$ as a contravariant functor $\rdcat{\mco}{A}\to \rdcat{\mco}{A}$.

\subsection{Manifolds and complexes}\label{sse:manifolds}

Consider the algebraic group $\PGL_{2}$ over $F$. Let $K_\infty\subseteq \PGL_{2}(F\otimes_{\mbb Q}{\mbb R})$ be a maximal compact subgroup. For any compact open subgroup  $K = \prod_vK_v \subseteq \PGL_2(\mbb{A}_{F}^\infty)$  consider the topological space:
\[
Y_K = \PGL_2(F)\backslash \PGL_2(\mbb{A}_F)/KK_\infty
\]
which is an orbifold of dimension $2r_1+3r_2$. We say that $K$ is \emph{sufficiently small} if $gKg^{-1} \cap \PGL_2(F)$ is  torsion-free for all $g \in \PGL_2(\mbb{A}_F^\infty)$.

We note the following lemma (\cite[Lemma 6.1]{KT}).
\begin{lemma}\label{lem_good_subgroups_act_without_fixed_points}
	Let $K \subset \PGL_2(\mbb{A}_{F}^\infty)$ be a sufficiently small subgroup.
	\begin{enumerate}[\quad\rm(1)]
		\item  Then for all $g \in \PGL_2(\mbb{A}_{F}^\infty)$, the group $\overline{\Gamma}_{K, g}\colonequals gKg^{-1} \cap \PGL_2(F)$ acts freely on $\PGL_{2}(F\otimes_{\mbb Q}{\mbb R})/K_\infty$, and $Y_K$ is endowed with the structure of smooth manifold of dimension $2r_1+3r_2$.
		
		\item Let $V = \prod_v V_v$ be a normal open compact subgroup of $K$. Then the map $Y_V \to Y_K$ is a Galois covering space, with Galois group $K/V$. \qed
	\end{enumerate}
\end{lemma}

For any $K$, let $C_{K}$ be the complex of singular chains on $Y_K$ with coefficients in $\mco$, so that $C_{K}$ computes of the homology $\hh_*(Y_K,\mco)$. Then $C_{K}$ is quasi-isomorphic to a bounded complex of free $\mco$-modules, as so may be viewed as a perfect complex in $\dcat{\mco}$.

\subsection{Double coset operators}\label{sse:double coset}

Let $G = \PGL_2(\mbb{A}_{F}^\infty)$ 
Let $C_c(G)$ be the abelian group of compactly supported continuous functions $f\colon G\to \mbb Z$. For any compact open subgroups $K_1,K_2\subseteq G$, let 
\[
\hecke{G}{K_2}{K_1} = \{f\in C_c(G)|f(k_2xk_1) = f(x) \text{ for all }x\in G,k_1\in K_1, k_2\in K_2\},
\]
which is an additive subgroup of $C_c(G)$. For any $\alpha\in G$ and any compact opens $K_1,K_2\subseteq G$, let $[K_2\alpha K_1]\in \hecke{G}{K_2}{K_1}$ denote the indicator function of the double coset $K_2\alpha K_1 = \{k_2\alpha k_1|k_1\in K_1,k_2\in K_2\}\subseteq G$. Then $\hecke{G}{K_2}{K_1}$ is clearly a free abelian group generated by the distinct $[K_2\alpha K_1]$'s.

Moreover we may define a convolution operation 
\[*:\hecke{G}{K_3}{K_2}\times \hecke{G}{K_2}{K_1} \to \hecke{G}{K_3}{K_1}\]
by
\[
(g*f)(x) = \sum_{K_2y\in K_2\backslash G} g(xy^{-1})f(y).
\]
This operation is clearly bilinear and associative, and satisfies 
\[
[K_2]* f = f = f*[K_1]
\]
for any $f\in \hecke{G}{K_2}{K_1}$. In particular, this makes $\hecke{G}{K}{K}$ into a ring.

Moreover it is well known that the convolution operation is given by
\[
[K_3\beta K_2]*[K_2\alpha K_1] = \sum_{K_1\gamma K_3}c_{\gamma}[K_3\gamma K_1]
\]
where the sum runs over all double cosets $K_3 \gamma K_1\subseteq G$ and
\[
c_\gamma = \#\{(i,j)|\beta_j\alpha_i K_1 = \gamma K_1\}
\]
where $\{\alpha_i\}$ and $\{\beta_j\}$ are (finite) sets of coset representatives, $K_2\alpha K_1 = \sqcup_i \alpha_i K_1$ and $K_3\beta K_2 = \sqcup_j \beta_jK_2$.

Let $K_1,K_2\subseteq \PGL_2(\mbb{A}_F^{\infty})$ be any compact open subgroups and $\alpha\in \PGL_2(\mbb{A}_F^{\infty})$ any element. we  define a map $[K_2\alpha K_1]: C_{K_1}\to C_{K_2}$ in $\dcat{\mco}$ as follows.

Let $L = \alpha K_1\alpha^{-1}\cap K_2$, which is also a compact open subgroup of $\PGL_2(\mbb{A}_F^{\infty})$. Define  maps $p_1\colon Y_L\to Y_{K_1}$ and $p_2\colon Y_L\to Y_{K_2}$ by $p_1([x]_L) = [x\alpha]_{K_1}$ and $p_2([x]_L) = [x]_{K_2}$ for $x\in \PGL_2(\mbb{A}_F^{\infty})$ (where for any $U$ and any $y\in \PGL_2(\mbb{A}_F^{\infty})$, $[y]_U$ denotes the equivalence class of $y$ in $Y_U = \PGL_2(F)\backslash \PGL_2(\mbb{A}_F)/UK_\infty$). Note that the map $p_1$ is well-defined as $\alpha^{-1} L\alpha\subseteq K_1$.

The maps $p_i\colon Y_L\to Y_{K_i}$ for $i=1,2$ are both finite-to-one maps of topological spaces (and are covering maps in the case where the $K_i$'s are sufficiently small, Lemma \ref{lem_good_subgroups_act_without_fixed_points}). Thus they induce maps $p_{i,*}:C_{L}\to C_{K_i}$ and $p_i^*\colon C_{K_i}\to C_{L}$. 

Define the \emph{double coset operator} $[K_2\alpha K_1]\in \Hom_{\dcat{\mco}}(C_{K_1},C_{K_2})$ as the composition $C_{K_1}\xrightarrow{p_1^*} C_{L}\xrightarrow{p_{2,*}} C_{K_2}$. As shown in \cite[Lemma 2.19]{NewtonThorne}\footnote{Technically \cite{NewtonThorne} only proves this in the case when $K_1 = K_2$. However their methods give the result stated here without any significant modifications.} we have

\begin{lemma}\label{lem:double coset}
The assignment $[K_2\alpha K_1]\mapsto[K_2\alpha K_1]$ defines an additive group homomorphism $\hecke{G}{K_2}{K_1}\to \Hom_{\dcat{\mco}}(C_{K_1},C_{K_2})$ such that the map
\[
\Hom_{\dcat{\mco}}(C_{K_2},C_{K_3})\times \Hom_{\dcat{\mco}}(C_{K_1},C_{K_2})\to \Hom_{\dcat{\mco}}(C_{K_1},C_{K_3})
\]
induced by convolution $*\colon \hecke{G}{K_3}{K_2}\times \hecke{G}{K_2}{K_1}\to \hecke{G}{K_3}{K_1}$ is just function composition.

In particular, $\hecke{G}{K}{K}\to \End_{\dcat{\mco}}(C_{K})$ is a ring homomorphism. \qed
\end{lemma}

\subsection{Hecke algebras}\label{sse:Hecke}

From now on, take a compact open subgroup $K = \prod_v K_v\subseteq \PGL_2(\mbb{A}_{F}^\infty)$, and  let $S$ be a finite set of places such that for each $v\not\in S$, $K_v$ is  a maximal compact subgroup of $\PGL_2(F_v)$ ($S$ need not be the minimal set with these properties). For each $v\not\in S$, let $T_v\in\End_{\dcat{\mco}}(C_{K})$ be the double coset operator
\[
T_v = \left[K\begin{pmatrix}\varpi_v&0\\0&1\end{pmatrix}K\right]
\]
 if $v\not\in S$  and $K_1(v)\subseteq K_v\subseteq K_0(v)$ and $d\in K_0(v)/K_v\into (\mco_F/v)^\times$ let $U_v,\langle d\rangle\in \End_{\dcat{\mco}}(C_{K})$ denote the double coset operators
\begin{align*}
U_v &= \left[K\begin{pmatrix}\varpi_v&0\\0&1\end{pmatrix}K\right]&
&\text{and}&
\langle d\rangle_v &= \left[K\begin{pmatrix}\widetilde{d}&0\\0&1\end{pmatrix}K\right]
\end{align*}
for any lift $\widetilde{d}\in \mco_{F,v}^\times$ of $d$.

Define $\TT^S(K) = \mco[T_v|v\not\in S]\subseteq \End_{\dcat{\mco}}(C_{K})$, which is well known to be a finite, commutative $\mco$-algebra. We  state a  conjecture  about existence of Galois representations attached to Hecke eigenclasses (which might be torsion) arising from the cohomology of $C_{K}$. In the case when $F$ is a CM field, part (i) is known by the work of Scholze and part (ii) is known for sufficiently large $\ell$;  see \cite[Theorem 1.1, Theorem 1.3] {NewtonThorne}. 
This is similar to Conjecture A of \cite{Calegari/Geraghty:2018}.

\begin{conj}\label{rhobar conj}
Let $K = \prod_v K_v\subseteq \PGL_2(\mbb{A}_F^{\infty})$ be a compact open subgroup, and let $S$ be any finite set of places such that for each $v\not\in S$, $K_v$ is  a maximal compact subgroup of $\PGL_2(F_v)$. Then:
\begin{enumerate}[\quad\rm(i)]
\item
For every maximal ideal $\fm$ of $\TT^S(K)$, there is a  semisimple  Galois representation $\rhobar_{\fm}\colon G_F\to \GL_2(\TT^S(K)/\fm)$ such that for all $v\not\in S$, $v \nmid \ell$, $\rhobar_{\fm}|_{G_{F_v}}$ is unramified and $\rhobar_{\fm}(\Frob_v)$ has characteristic polynomial $x^2-T_vx+\Nm(v)$. 
\item
Furthermore there is a lift  of $\rhobar_{\fm}$ to a representation
\[
\rho_K\colon G_F\to \GL_2(\TT^S(K)_\fm)
\]
such that  for all $v\not\in S$,  $v \nmid \ell$, $\rho_{K}|_{G_{F_v}}$ is unramified and $\rho_{K}(\Frob_v)$ has characteristic polynomial $x^2-T_vx+\Nm(v)$.
\end{enumerate}
\end{conj}

\begin{remark}
Note that if we assumed in (ii) above the weaker statement that, for a $\TT \in \CNL_{\mco}$  with an inclusion   $\TT^S(K)_\fm \hookrightarrow \TT$,  there is a lift  of $\rhobar_{\fm}$ to a representation $\rho_K\colon G_F\to \GL_2(\TT)$ such that  for all $v\not\in S$,  $v$ not above $\ell$, $\rho_{K}|_{G_{F_v}}$ is unramified and $\rho_{K}(\Frob_v)$ has characteristic polynomial $x^2-T_vx+\Nm(v)$, then the stronger hypothesis in (ii) that the representation can be chosen to take values in $\GL_2(\TT^S(K)_\fm )$ follows provided the residual representation $\rhobar_{\fm}$ is irreducible. This follows upon  using the Chebotarev density theorem as traces of $\Frob_v$ for $v \not \in S$, $v \nmid \ell$,   with $S$ a finite set of places, are in $\TT^S(K)_\fm $.
\end{remark}

We  say that a maximal ideal $\fm\subseteq \TT^S(K)$ is \emph{non-Eisenstein} if the representation $\rhobar_{\fm}$ is absolutely irreducible. 
In what follows we often use the lemma below, which is a consequence of the Chebotarev density theorem (see  \cite[Lemma 6.20]{KT}).

\begin{lemma}
Assume Conjecture \ref{rhobar conj}.  Then for any finite set of places $S'$ containing $S$, the natural map  $\TT^{S’}(K)_\fm \to  \TT^S(K)_\fm $  is an isomorphism when $\fm$ is non-Eisenstein.
\end{lemma}

Thus if $\fm$ is non-Eisenstein the localization $\TT^S(K)_\fm$ does not depend on $S$, and so we often use $\TT(K)_\fm$ to denote this localization, eliding $S$ in the notation. As shown in \cite{KT}, for any maximal ideal $\fm\subseteq \TT(K)$ the localization $(C_{K})_\fm$ is a direct summand of $C_{K}$ and we have $\hh_*(Y_K,\mco)_\fm= \hh_*(C_{K})_\fm = \hh_*((C_{K})_\fm)$.  Given a maximal ideal $\fm \subseteq \TT^S(K)$ and any compact open $K'\subseteq K$, we  also use $\fm\subseteq \TT^S(K')$ to denote the preimage of $\fm$ under the map $\TT^S(K')\onto \TT^S(K)$.

In what follows we  assume that $\rhobar_\fm$  is such that $\rhobar_\fm|_{F(\zeta_\ell)}$ is irreducible.  The result below is easy to prove; see  \cite[Lemma 12.3]{Jar99}  for the proof of the  first part.

\begin{lemma}\label{trivial}
Let $\rhobar:G_F \to \GL_2(k)$ be such that  $\rhobar|_{F(\zeta_\ell)}$ is irreducible. 
\begin{enumerate}[\quad\rm(i)]
\item
Then there exists a place $t \notin S$ such that  
\begin{gather*}
 \tr \rhobar(\Frob_t)/\det \rhobar(\Frob_t) \neq (1+\Nm(t))^2/\Nm(t)\\  
 \Nm(t) \not\equiv 1 \pmod{\ell}\,,
 \end{gather*}
and $\Nm(t)>4^{[F:\mbb Q]}$.
\item
For such places $t$ the (unrestricted)  deformation ring $R_t^{\square}=R_t^{\rm min}$, and  is  isomorphic to a power series ring over $\mco$ of relative dimension 3.
\item
For such places $t$, the compact open subgroup  
\begin{align*}
K_1(t^2) &= \left\{
\begin{pmatrix}
	a&b\\c&d
\end{pmatrix}\in \PGL_2(\mcohat_F)\middle|c \in t^2, ad^{-1} =1 \pmod{ t^2 }\right\}\subseteq \PGL_2(\mcohat_F)
\end{align*}
is sufficiently small. \qed
\end{enumerate}
\end{lemma}

In what follows we apply the lemma for $\rhobar=\rhobar_\fm$ and  always impose level structures at $t$ so that we work with open compact subgroups $K$ of $\PGL_2(\mbb{A}_{F}^\infty)$ that are sufficiently small. Because no lift  of $\rhobar_\fm$ is ramified at $t$ we do not need to explicitly allow the Galois deformations  of $\rhobar_\fm$ we consider below to ramify at $t$. (See   6.5.1,  Lemma 6.5.2, and Lemma 6.5.2 of \cite{ACC+:2018} for similar considerations.)

\subsection{Sufficiently small level subgroups}\label{sse:K_0(N)}

We fix a prime $t$ as in Lemma \ref{trivial}.  For any nonzero ideal $\mcn\subseteq \mco_F$ define the following  compact open subgroups of $\PGL_2(\mbb{A}_{F}^\infty)$:

\begin{align*}
K_0(\mcn) &= \left\{
\begin{pmatrix}
a&b\\c&d
\end{pmatrix}\in \PGL_2(\mcohat_F)\middle|c\in \mcn t^2, ad^{-1} =1 \pmod{ t^2 }\right\}\subseteq \PGL_2(\mcohat_F)\\
K_1(\mcn) &= \left\{
\begin{pmatrix}
	a&b\\c&d
\end{pmatrix}\in \PGL_2(\mcohat_F)\middle|c \in \mcn t^2, ad^{-1} =1 \pmod{ \mcn t^2 }\right\}\subseteq \PGL_2(\mcohat_F)
\end{align*}
so that $K_1(\mcn)\unlhd K_0(\mcn)$ and $K_0(\mcn)/K_1(\mcn)\cong (\mco_F/\mcn)^\times$ via the map
\[
\begin{pmatrix}
	a&b\\c&d
\end{pmatrix}\mapsto ad^{-1}\pmod{\mcn}\,.
\]
Also define $K_\Delta(\mcn)$ to be the smallest intermediate subgroup
\[
K_1(\mcn)\le K_\Delta(\mcn)\le K_0(\mcn)
\]
for which $|K_0(\mcn)/K_\Delta(\mcn)|$ is an $\ell^{th}$ power.

In the case when $\mcn = v^e$ for some prime $v$ and some $e\ge 0$, we  also sometimes use $K_0(v^e)$, $K_1(v^e)$ and $K_\Delta(v^e)$ to denote the compact open subgroups $K_0(v^e)\cap \PGL_2(\mco_{F,v})$, $K_1(v^e)\cap \PGL_2(\mco_{F,v})$ and $K_\Delta(v^e)\cap \PGL_2(\mco_{F,v})$ of $\PGL_2(F_v)$.

For any nonzero ideal $\mcn\subseteq \mco_F$ we  write $Y_0(\mcn) = Y_{K_0(\mcn)}$. Also for an ideal $Q$ of $\mco_F$ prime to $\mcn$, write $Y_{0,\Delta}(\mcn,Q) = Y_{K_0(\mcn)\cap K_\Delta(Q)}$. Write $C_0(\mcn)$ and $C_{0,\Delta}(\mcn,Q)$ for the corresponding prefect complexes in $\dcat{\mco}$.

 (Thus  we  are imposing  level structure at $t$ for the compact open subgroups $K$ of $\PGL_2(\mbb{A}_F^{\infty})$ we consider;  we suppress this from the notation to make it less clumsy and  hope this will not be misleading for the reader.)

\subsection*{Duality}
\label{sse:Duality}
From now on fix a nonzero-ideal $\mcn_\es\subseteq \mco_F$ and a non-Eisenstein maximal ideal $\fm\subseteq \TT^S(K_0(\mcn_\es))$. 
Assume that $N(\rhobar_\fm) = \mcn_\es$ (i.e. $\mcn_\es$ is the \emph{minimal level} for $\rhobar_\fm$). By enlarging $\mco$ if necessary we may assume that $\TT^S(K)/\fm\cong k$, and so $\rhobar_\fm$ is a representation $G_F\to \GL_2(k)$.

\begin{proposition}
\label{prop: Verdier}
Assume Conjecture \ref{rhobar conj} holds. Then for any compact open $K=\prod_vK_v\subseteq K_0(\mcn_\es)$, there is a natural derived $\TT(K)_\fm$-isomorphism 
\[
(C_{K})_\fm\cong \RHom_{\mco}((C_{K})_\fm,\mco)[2r_1+3r_2] = (C_K)_\fm^{\dagger}
\]
in $\rdcat{\mco}{\TT(K)_\fm}$. Moreover, for any $K_1,K_2\subseteq K_0(\mcn_\es)$ the adjoint of the double coset operator $[K_2\alpha K_1]\colon (C_{K_1})_\fm\to (C_{K_2})_\fm$ with respect to these isomorphisms is given by the double coset operator $[K_1\alpha^{-1} K_2]\colon (C_{K_2})_\fm\to (C_{K_1})_\fm$.
\end{proposition}

\begin{proof}
By \cite[Proposition 3.7]{NewtonThorne} and \cite[Theorem 4.2]{NewtonThorne} there is a Verdier duality isomorphism $(C_{K})_\fm\cong \RHom_{\mco}((C_{K})_\fm,\mco[2r_1+3r_2])$ sending each double coset operator $[K_2\alpha K_1]$ to $[K_1\alpha^{-1}K_2]$. Here we are using the fact that condition $\spadesuit$ of \emph{loc. cit.} holds in the case $m=2$ by Conjecture \ref{rhobar conj} and in the case $m=1$ by class field theory.

 In particular this is $\TT^S(K)$-equivariant as it is easy to check that 
	\begin{align*}
	K\begin{pmatrix}
		\varpi_v^{-1}&0\\0&1
	\end{pmatrix}
	K
	&= 
	K
\begin{pmatrix}
	0&1\\1&0
\end{pmatrix}
	\begin{pmatrix}
	\varpi_v^{-1}&0\\0&1
	\end{pmatrix}
\begin{pmatrix}
0&1\\1&0
\end{pmatrix}
	K \\
	&= 
	K\begin{pmatrix}
		1&0\\0&\varpi_v^{-1}
	\end{pmatrix}
	K \\
	&=
	K\begin{pmatrix}
	\varpi_v&0\\0&1
	\end{pmatrix}
	K,
	\end{align*}
whenever $K_v = \PGL_2(\mco_{F,v})$, and so each $T_v$ for $v\not\in S$ is self-adjoint.
\end{proof}

The isomorphism from Proposition \ref{prop: Verdier} is not in general equivariant for the Hecke operators $U_v$ and $\langle d\rangle_v$. However, it is easy to modify this isomorphism to make it equivariant for all Hecke operators.

From now on we  restrict our attention to compact open subgroups $K$ in the form $K = \prod_v K_v\subseteq K_0(\mcn_\es)$ where for all $v$ either $K_v = \PGL_2(\mco_{F,v})$ or $K_1(v)\le K_v \le K_0(v)$.  For each $v$, let 
\[
w_v = \left[
K\begin{pmatrix}
0&-1\\\varpi_v&0
\end{pmatrix}
K
\right]
\]
if $K_1(v)\le K_v \le K_0(v)$ and let $w_v = 1$ if $K_v = \PGL_2(\mco_{F,v})$.
Observe that $K_v\begin{pmatrix}
	0&-1\\\varpi_v&0
\end{pmatrix} = \begin{pmatrix}
0&-1\\\varpi_v&0
\end{pmatrix} K_v$ for $K_1(v)\le K_v \le K_0(v)$ and $K_v = \PGL_2(\mco_{F,v})$. It is easy to verify from this that
\[
[K_3\beta K_2]*\left[
K_2\begin{pmatrix}
	0&-1\\\varpi_v&0
\end{pmatrix}K_1
\right] = \left[K_3\beta \begin{pmatrix}
	0&-1\\\varpi_v&0
\end{pmatrix}K_1\right]
\]
and 
\[
\left[K_3\begin{pmatrix}
	0&-1\\\varpi_v&0
\end{pmatrix} K_2\right]*\left[
K_2\alpha K_1
\right] = \left[K_3 \begin{pmatrix}
	0&-1\\\varpi_v&0
\end{pmatrix}\alpha K_1\right]
\]
It follows that $w_v^2 = 1$ for all $v$ and $w_{v}w_{v'} = w_{v'}w_v$ for any $v$ and $v'$. Moreover by definition, $w_v = 1$ whenever $K_v = \PGL_2(\mco_{F,v})$ (and hence in particular, $w_v = 1$ for $v\not\in S$). Hence the operator $w_K = \prod_{v\in S}w_v\in \End_{\dcat{\mco}}((C_{K})_\fm)$ is well-defined and independent of $S$. As $w_K^2 = 1$, $w_K$ is an automorphism in $\dcat{\mco}$. So now we may define an isomorphism $\varphi_K\colon (C_{K})_\fm\isomto (C_K)_\fm^\dagger$ by
\[
\varphi_K:(C_{K})_\fm\xrightarrow{w_K}(C_{K})_\fm\xrightarrow{\sim}(C_K)_\fm^\dagger
\]
where the second map is the isomorphism from Proposition \ref{prop: Verdier}.

Now we have the following:

\begin{proposition}
\label{prop: twisted duality}
The isomorphism $\varphi_K\colon (C_{K})_\fm\cong (C_{K})_\fm^{\dagger}$ commutes with the actions of all the Hecke operators $T_v$, $U_v$ and $\langle d\rangle_v$.

Moreover the adjoint of $[K_2\alpha K_1]\colon (C_{K_1})_\fm\to (C_{K_2})_\fm$ with respect to $\varphi_{K_1}$ and $\varphi_{K_2}$ is $w_{K_1}[K_1\alpha^{-1} K_2]w_{K_2}\colon (C_{K_2})_\fm\to (C_{K_1})_\fm$.
\end{proposition}

\begin{proof}
The last claim is immediate from Proposition \ref{prop: Verdier} and the definition of $\varphi_K$.

For the first claim, we have $w_KT_vw_K = T_v$ for $v\not\in S$ by the definition of $w_K$. The fact that $U_v$ and $\langle d\rangle_v$ are self adjoint now follows from the observation that
\begin{align*}
\begin{pmatrix}
	0&-1\\\varpi_v&0
\end{pmatrix}
\begin{pmatrix}
	a&0\\0&d
\end{pmatrix}^{-1}
\begin{pmatrix}
	0&-1\\\varpi_v&0
\end{pmatrix}
&=
\begin{pmatrix}
	0&-1\\\varpi_v&0
\end{pmatrix}
\begin{pmatrix}
	a^{-1}&0\\0&d^{-1}
\end{pmatrix}
\begin{pmatrix}
	0&-1\\\varpi_v&0
\end{pmatrix}\\
&
=
\begin{pmatrix}
	-\varpi_vd^{-1}&0\\0&-\varpi_va^{-1}
\end{pmatrix} \\
&=
\begin{pmatrix}
	a&0\\0&d
\end{pmatrix}
\end{align*}
in $\PGL_2(F_v)$, for any $a,d\in F_v^\times$.
\end{proof}

\subsection{Level lowering maps}\label{ssec: level lowering}

For this subsection fix a compact open $K$  of $\PGL_2(\A)$ as before and a place $v$ for which $K_v = \PGL_2(\mco_{F,v})$. Write $L = K\cap K_0(v)$. Define two maps $\pi_1,\pi_2\colon (C_{L})_\fm \to (C_{K})_\fm$ by 
\[
\pi_1 = \left[K\begin{pmatrix}
	1&0\\0&1
\end{pmatrix}L
\right] \quad \text{and}\quad \pi_2=\left[K\begin{pmatrix}
	\varpi_v&0\\0&1
\end{pmatrix}L
\right]
\]
and write $\pi_{K,v} = \pi_1\oplus \pi_2:(C_{L})_\fm \to (C_{K})_\fm^{\oplus 2}$. Let $\pi_{K,v}^\dagger\colon (C_{K})_\fm^{\oplus 2}\to (C_{L})_\fm$ be the adjoint of $\pi_{K,v}$ with respect to $\varphi_K$ and $\varphi_L$.   It  is easy to see that 

\[
\pi_{K,v}\circ U_v = \begin{pmatrix}
T_v&-1\\\Nm(v)&0
\end{pmatrix}
\circ \pi_{K,v}
\]
as morphisms $(C_{L})_\fm \to (C_{K})_\fm^{\oplus 2}$.

\subsection{$\TT_\Sigma$ and $\TT_{\Sigma,Q}$}\label{sse:Sigma}

Pick disjoint finite sets $T$ and $Q$ of places of $F$ satisfying the conditions from Section \ref{ssec:global def} (which in particular imply that every element of $T$ and $Q$ is prime to $\mcn_\es$).

For any subset $\Sigma\subseteq T$ let $\mcn_\Sigma = \mcn_\es\prod_{v\in\Sigma}v$, where $\mcn_\es$ is as above. From now on assume that $T\cup Q\subseteq S$, and $S$ contains all primes dividing $\mcn_\es$ and the place $t$.
For any $\Sigma\subseteq T$, define $K_\Sigma = K_0(\mcn_\Sigma)$ and $K_{\Sigma,Q} = K_0(\mcn_\Sigma)\cap K_\Delta(Q)$. Note that $K_{\Sigma,\es} = K_\Sigma$.

Consider the Hecke algebras $\TT^S(K_\Sigma)$ and $\TT^S(K_{\Sigma,Q})$. Let $\fm\subseteq \TT^S(K_0(\mcn_\es))$ be a non-Eisenstein maximal ideal. Assume further that $\rhobar_{\fm}|_{G_{F(\zeta_\ell)}}$ is absolutely irreducible.

Again, for any compact open $K\subseteq K_0(\mcn_\Sigma)$, let $\fm\subseteq \TT^S(K)$ denote the preimage of $\fm$. Write $\TT_\Sigma =\TT^S(K_\Sigma)_\fm$ and $\TT_{\Sigma,Q} =\TT^S(K_{\Sigma,Q})_\fm$.

Now define the full Hecke algebras:
\begin{align*}
\fullT^S(K_\Sigma) &= \TT^S(K_\Sigma)[\{U_v\}_{v\in\Sigma}]\subseteq \End_{\dcat{\mco}}(C_{K_\Sigma})\\
\fullT^S(K_{\Sigma,Q}) &= \TT^S(K_\Sigma)[\{U_v\}_{v\in\Sigma\cup Q},\{\langle d\rangle_v\}_{v\in Q,d\in \Delta_v}]\subseteq \End_{\dcat{\mco}}(C_{K_{\Sigma,Q}}).
\end{align*}

These are  commutative $\mco$-algebras, finite over $\TT^S(K_\Sigma)$ and $\TT^S(K_{\Sigma,Q})$. Again note that $\fullT^S(K_{\Sigma,\es}) = \fullT^S(K_{\Sigma})$.
We define ideals  $\fm_{\Sigma,Q}\subseteq \fullT^S(K_{\Sigma,Q})$ lying over $\fm\subseteq \TT^S(K_{\Sigma,Q})$ by specifying  that
$U_v-\epsilon_v$ (for $v \in \Sigma$) and $U_v-\gamma_{1,v}$ (for $v \in Q$)  are in $\fm_{\Sigma,Q}$.   Then it is easy to see that these are indeed maximal ideals and that the ideal $\fm_{\Sigma,Q}\subseteq \fullT^S(K_{\Sigma,Q})$ lies over the ideal $\fm_{\Sigma,\es}\subseteq \fullT^S(K_{\Sigma})$ and for $\Sigma_1\subseteq \Sigma_2$, the ideal $\fm_{\Sigma_2,Q}\subseteq \fullT^S(K_{\Sigma_2,Q})$ lies over the ideal $\fm_{\Sigma_1,Q}\subseteq \fullT^S(K_{\Sigma_1,Q})$.

The following conjecture  asserts that the Galois representations of Conjecture \ref{rhobar conj} satisfy local-global compatibility. Many cases of this conjecture (up to the issue of going modulo a nilpotent ideal) are known from \cite[\S 3, \S 4]{ACC+:2018}).

\begin{conj}\label{R->T conj}
For any $\Sigma$ and $Q$, the representations $\rho_{\Sigma,Q}\colon G_F\to \GL_2(\TT_{\Sigma,Q})$  stemming from Conjecture \ref{rhobar conj}    arise from  the corresponding universal representation  $G_F \to \GL_2( R_{\Sigma,Q})$ via maps $R_{\Sigma,Q}\onto \TT_{\Sigma,Q}$  and satisfy the following properties:
\begin{enumerate}[\quad\rm(1)]
	\item 
	 $\tr\rho_{\Sigma,Q}(\Frob_v) = T_v$ for all $v\not\in S$.
	\item If $R_{\Sigma,Q}$ is given the $\mco[\Delta_Q]$ structure arising from Proposition \ref{prop: O[Delta]->R}, then  the composition of  maps  $R_{\Sigma,Q}\onto \TT_{\Sigma,Q} \to \fullT^S(K_{\Sigma,Q})_{\fm_{\Sigma,Q}} $ is an $\mco[\Delta_Q]$-algebra homomorphism. Furthermore $\chi_{1,v}(\varphi_v)$ maps to $U_v$ for $v \in Q$ for a suitable lift $\varphi_v$ of $\Frob_v$.
	\item For $v\in\Sigma$, writing $B_v\in \TT_\Sigma^\square$ for the image of $B_v\in R_v^{{\rm uni}(\epsilon_v)}$ under the map 
\[
	R_v^{{\rm uni}(\epsilon_v)}\into R_{\Sigma,\loc}\to R_\Sigma^\square \onto \TT_\Sigma^\square,
\]
we have $(B_v) = (\epsilon_vU_v-1)$ as ideals of $\TT_\Sigma^\square$. 	

\item The diagrams
\begin{center}
$
	\begin{tikzcd}
		R_{\Sigma,Q} \arrow[twoheadrightarrow]{d} \arrow[twoheadrightarrow]{r} & \TT_{\Sigma,Q} \arrow[twoheadrightarrow]{d}  \\
		R_{\Sigma} \arrow[twoheadrightarrow]{r} & \TT_{\Sigma}
	\end{tikzcd}
$
	and
$
	\begin{tikzcd}
		R_{\Sigma_2,Q} \arrow[twoheadrightarrow]{d} \arrow[twoheadrightarrow]{r} & \TT_{\Sigma_2,Q} \arrow[twoheadrightarrow]{d}  \\
		R_{\Sigma_1,Q} \arrow[twoheadrightarrow]{r} & \TT_{\Sigma_1,Q}
	\end{tikzcd}
$
\end{center}
	(for $\Sigma_1\subseteq \Sigma_2\subseteq T$) commute.
\end{enumerate}
\end{conj}

Note that the assumption that $\TT_{\Sigma,Q}$ is a quotient of $R_{\Sigma,Q}$ implies that $\rho_{\Sigma,Q}$ satisfies all of the necessary local-global compatibility assumptions (such as the fact that it is flat at all primes $v|\ell$).

\begin{lemma}\label{lem: localization}
We have that  $\fullT^S(K_{\Sigma,Q})_{\fm_{\Sigma,Q}} = \TT^S(K_{\Sigma,Q})_\fm = \TT_{\Sigma,Q}$.

\end{lemma}
\begin{proof}    This follows from the above Conjecture \ref{R->T conj} as from it we see that $U_v, v \in \Sigma \cup Q$ and $ \{\langle d\rangle_v\}_{v\in Q,d\in \Delta_v}\subseteq \End_{\dcat{\mco}}(C_{K_{\Sigma,Q}})$ belong to $\TT_{\Sigma,Q}$.
\end{proof}

We will write $\fm_\Sigma = \fm_{\Sigma,\es}$. Lemma \ref{lem: localization} implies that we may now treat $U_v$ for $v\in\Sigma\cup Q$ and $\langle d\rangle_v$ for $v\in Q$ and $d\in\Delta_v$ as elements of $\TT_{\Sigma,Q}$. In particular there is now a natural $\mco$-algebra homomorphism $\mco[\Delta_Q] = \otimes_{v\in Q}\mco[\Delta_v]\to \TT_{\Sigma,Q}$, given by $d\mapsto \langle d\rangle_v$ for $d\in\Delta_v$, and so we may view $\TT_{\Sigma,Q}$ as a $\mco[\Delta_Q]$-algebra.

Define the complexes $C_\Sigma = (C_{K_\Sigma})_{\fm_\Sigma}$ and $C_{\Sigma,Q} = (C_{K_{\Sigma,Q}})_{\fm_{\Sigma,Q}}$ in $\dcat{\mco}$. Note that these are not the same as the complexes $(C_{K_\Sigma})_{\fm}$ and $(C_{K_{\Sigma,Q}})_{\fm}$, as we are localizing with respect to the action of a larger ring.
In particular we naturally have $C_\Sigma \in \rdcat{\mco}{\TT_\Sigma}$ and $C_{\Sigma,Q}\in \rdcat{\mco}{\TT_{\Sigma,Q}}$.

\begin{proposition}\label{prop: O[Delta]->T}
For each $\Sigma\subseteq T$ and each $Q$ we have:
\begin{itemize}
	\item The action $\mco[\Delta_Q]\to \TT_{\Sigma,Q}\to \End_{\dcat{\mco}}(C_{\Sigma,Q})$ gives $C_{\Sigma,Q}$ the structure of a perfect complex of $\mco[\Delta_Q]$-modules.
	\item 
	$\TT_{\Sigma,Q}\otimes_{\mco[\Delta_Q]}\mco\cong \TT_\Sigma$
	\item $C_{\Sigma,Q}\lotimes_{\mco[\Delta_Q]}\mco\cong C_{\Sigma}$ in $\rdcat{\mco}{\TT_{\Sigma}}$.
\end{itemize}
\end{proposition}
\begin{proof} This is proved in  \cite[\S 5]{KT} in a more general context. It is also proved (for $\GL_2$) in \cite[Lemma 9.5]{Calegari/Geraghty:2018}.
\end{proof}

Now fix $v\in T\sm \Sigma$ and consider the map 
 \[
 \pi_{K,v} = \pi_1\oplus \pi_2\colon (C_{K_{\Sigma \cup \{v\},Q}})_{\fm_{\Sigma,Q}} \lra (C_{K_{\Sigma,Q}})_{\fm_{\Sigma,Q}}^{\oplus 2}
 \]
 arising from Section \ref{ssec: level lowering} (for $K=K_{\Sigma,Q}$, after localizing at $\fm_{\Sigma,Q}$).  We claim that this induces a map  $\pi_v\colon C_{\Sigma\cup\{v\},Q}\to C_{\Sigma,Q}$. To see this we consider the operator 
 \[ 
 \begin{pmatrix}
T_v&-1\\\Nm(v)&0
\end{pmatrix}
\]
that we call $U_v'$, acting on  $(C_{K_{\Sigma,Q}})_{\fm_{\Sigma,Q}}^{\oplus 2}$. Consider the maximal ideal $\fm'_{\Sigma,Q}$ of
\[
\TT_{\Sigma,Q}[U_v']/(U_v'^2-T_vU'_v+\Nm(v)) 
\]
acting on     $(C_{K_{\Sigma,Q}})_{\fm_{\Sigma,Q}}^{\oplus 2}$, that extends the maximal ideal $\fm_{\Sigma,Q}$ and is given by $(\fm_{\Sigma,Q},U'_v-  \epsilon_v)$. We consider the localization of  $(C_{K_{\Sigma,Q}})_{\fm_{\Sigma,Q}}^{\oplus 2}$ at $\fm'_{\Sigma,Q}$. This is isomorphic as a $\TT_{\Sigma,Q}$ module to  $C_{\Sigma,Q}$. Thus we get  a Hecke equivariant  (for the $\TT_{\Sigma \cup \{v\},Q}$-action) map  $\pi_v\colon C_{\Sigma\cup\{v\},Q}\to C_{\Sigma,Q}$. It also induces an adjoint map $\pi_v^\dagger\colon C_{\Sigma,Q} \to C_{\Sigma\cup\{v\},Q}$. (See \cite[\S 2.2]{Wiles:1995} for similar arguments.)

Note that we may treat $U_v$ as an element of $\TT_{\Sigma,Q}$ via Lemma \ref{lem: localization} and the quotient map $\TT_{\Sigma\cup\{v\},Q}\onto \TT_{\Sigma,Q}$. From  the relation  \[
\pi_{K,v}\circ U_v = \begin{pmatrix}
T_v&-1\\\Nm(v)&0
\end{pmatrix}
\circ \pi_{K,v}
\] noted in \S \ref{ssec: level lowering} 
  we get that the action of $U'_v$ on $C_{\Sigma,Q}$ agrees with the action of $U_v\in \TT_{\Sigma,Q}$ on $C_{\Sigma,Q}$.

 Proposition \ref{prop: twisted duality} and a standard  computation  gives  the following:

\begin{proposition}
\label{prop: localized duality}
There is a family of isomorphisms 
\[
\varphi_{\Sigma,Q}\colon C_{\Sigma,Q}\isomto C_{\Sigma,Q}^\dagger
\]
in $\rdcat{\mco}{\TT_{\Sigma,Q}}$ compatible with the isomorphisms $C_{\Sigma,Q}\lotimes_{\mco[\Delta_Q]}\mco\cong C_{\Sigma}$.

Moreover for each $v\in T\sm\Sigma$, if 
\[
\pi_v^\dagger\colon C_{\Sigma,Q} \cong  C_{\Sigma,Q}^\dagger \to  C_{\Sigma\cup \{v\},Q}^\dagger \cong C_{\Sigma\cup\{v\},Q}
\]
is the adjoint of $\pi_v\colon C_{\Sigma\cup\{v\},Q}\to C_{\Sigma,Q}$, upon using the isomorphisms  $\varphi_{\Sigma,Q}$ and $\varphi_{\Sigma\cup\{v\},Q}$, then 
\begin{align*}
\pi_v\circ \pi_v^\dagger & \simeq  U_v^2-1 \in \End_{\dcat{\mco}}(C_{\Sigma,Q})&
&\text{and}& 
\pi_v^\dagger\circ \pi_v & \simeq U_v^2-1 \in \End_{\dcat{\mco}}(C_{\Sigma\cup\{v\},Q}),
\end{align*} where $\simeq$  means up to units in the corresponding Hecke algebras $\TT_{\Sigma,Q}$ and $\TT_{\Sigma\cup\{v\},Q}$ respectively. 
\end{proposition}

\begin{proof}  The first assertion follows from  localizing the results of Proposition \ref{prop: twisted duality}. To see the second assertion above we can use that  (see the proof of \cite[Proposition 3.3]{RibCong} for a similar computation)
 \[  
 \pi_{K,v}  \circ  (U_v^2-1)=  \pi \circ \pi^\dag  \circ
 \begin{pmatrix}
0&-1\\-1&T_v
\end{pmatrix}
\circ \pi_{K,v}\]  as maps  $(C_{K_{\Sigma \cup \{v\},Q}})_{\fm}\to (C_{K_{\Sigma ,Q}})_{\fm}^{\oplus 2}$. 
\end{proof}

\begin{conj}\label{vanishing conj}
If $K\subseteq \PGL_2(\mbb{A}_F^{\infty})$ is any compact open subgroup and $\fm$ is a non-Eisenstein maximal ideal, then $\hh_i(X_K,k)_{\fm} = 0$ for $i\not\in [r_1+r_2,r_1+2r_2]$.
\end{conj}

This  conjecture is known when $F$ is an imaginary quadratic field  as then $X_K$ is a 3-dimensional manifold and  $\hh_i(X_K,k)_{\fm} = 0$ for $i=0,3$ (as $\fm$ is a non-Eisenstein maximal ideal). 

\subsection{Level raising and Ihara's Lemma}\label{sec:Ihara}

\begin{conj}\label{Ihara conj}
Let $K$ and $v$ be as in Section \ref{ssec: level lowering}, and assume that $\fm\subseteq \TT^S(K)$ is a \emph{non-Eisenstein} maximal ideal. Let $\pi_{K,v}\colon (C_{K\cap K_0(v)})_\fm\to (C_K)_\fm^{\oplus 2}$ be the level lowering map.

Then the induced map 
\[\hh_{r_1+r_2}(\pi_{K,v})\colon \hh_{r_1+r_2}(Y_{K\cap K_0(v)},\mco)_\fm \to \hh_{r_1+r_2}(Y_{K},\mco)_\fm^{\oplus 2}\]
is surjective.
\end{conj}

Note that by Conjecture \ref{vanishing conj}, $\hh_{i}(Y_{K},\mco)_\fm=0$ for $i<r_1+r_2$, and so Conjecture \ref{Ihara conj} concerns the smallest degree in which the homology $\hh_{*}(Y_{K},\mco)_\fm$ does not vanish.

As the map $\pi_v\colon C_{\Sigma\cup\{v\},Q}\to C_{\Sigma,Q}$ constructed above is a localization of the map $\pi_{K_{\Sigma,Q},v}$, we get:

\begin{lemma}\label{lem: Ihara U_v}
If Conjecture \ref{Ihara conj} holds for $K = K_{\Sigma,Q}$ then the map 
\[
\hh_{r_1+r_2}(\pi_v)\colon \hh_{r_1+r_2}(C_{\Sigma\cup\{v\},Q})\to \hh_{r_1+r_2}(C_{\Sigma,Q})
\]
 is surjective.
\end{lemma}

\begin{theorem}\label{Ihara's Lemma}
When $F$ is an imaginary quadratic field and $\ell>3$, Conjecture \ref{Ihara conj} is true. Namely  the map $\hh_1(\pi_{K,v})\colon \hh_{1}(Y_{K\cap K_0(v)},\mco)_\fm \to \hh_{1}(Y_{K},\mco)_\fm^{\oplus 2}$  is surjective.
\end{theorem}

\begin{proof} This is well known; see \cite[\S 4.1]{CV} or \cite{Klosin}. It is a consequence of the congruence subgroup property of $\SL_2(\mco_F[\frac{1}{v}])$ for $F$ an imaginary quadratic field.
\end{proof}

\section{Hecke algebras for weight one modular forms}
\label{sec:Hecke wt1}

In this section we work exclusively in Case \ref{case:wt1}, so restrict to the case $F=\QQ$ and fix a finite set $\disc$ of primes with even cardinality. We consider weight one modular forms defined on a Shimura curve $X^\disc_K$ (defined below) of discriminant $\disc$.  Frequently, we use $\disc$ also to denote the product of all the primes in the set $\disc$. The context will make clear which meaning is intended.

Theorem \ref{th:intro-5b} from the introduction was proved by Calegari \cite{Calegari}  when $\disc = \es$, i.e. the modular curve case, so  (in the proofs below) we assume  $\disc\ne \es$ for convenience. We also fix a prime $\ell>3$ which is not in $\disc$ and for all $v\in\disc$, $v\not\equiv 1 \pmod{\ell}$. Let $E$, $\mco$, $\varpi$ and $k$ be as before.

\subsection{Weight one sheaves and complexes on Shimura curves}
\label{shimura-weight one}
Let $D_\disc$ be the quaternion algebra over $\QQ$, ramified precisely at the primes in $\disc$. As $\disc$ has even cardinality, $D_\disc$ is indefinite.

Let $K = \prod_q K_q \subseteq D_\disc^\times(\A_\QQ^\infty)$ be a compact open subgroup such that: 
\begin{itemize}
	\item $K_q$ is the maximal compact subgroup of $D_{\disc}^\times(\QQ_q)$ when $q\in\disc$;
	\item $K_\ell$ is a maximal compact subgroup of $\PGL_2(\QQ_\ell)$;
	\item $K$ is \emph{sufficiently small}, in the sense that $gKg^{-1}\cap D_{\disc}^\times(\QQ)$ is torsion free for all $g\in D_{\disc}^\times(\AA_F^\infty)$;
	\item and the determinant map $K$ to  $\Pi_q \ZZ_q^*$ is surjective. 
\end{itemize} 
A typical such $K$ that we use is the subgroup $K_{H,\Delta}^\disc(N_\Sigma,Q)$ defined below.

Consider the (compact) Riemann surface 
\[
X^\disc_K \colonequals D^\times(\QQ)\backslash \left(D^\times(\A_{\QQ}^\infty)\times \half^{\pm}\right)/K
\]
where $\half^{\pm} := \CC\setminus \RR$ is the complex upper and lower half planes. Give $X^\disc_K$ its canonical structure as an algebraic curve over $\QQ$. Let $\X^\disc_K$ be a minimal integral model for $X^\disc_K$ over $\mco$, and for any $\mco$-algebra $A$, let $\X^\disc_{K,A}$ be the base change of $\X^\disc_{K}$ to $A$.

For $A=\mco,E,\mco/\varpi^n\mco$, 
let $\omega_A$ be the sheaf on $X^\disc_K$ arising from (applying an idempotent $e$ as in \cite[\S 4]{DT2}   which satisfies the condition $e^*=e$ for a Rossati involution $*$  to)   $ \pi_* \Omega_{\mathcal A/\X^\disc_{K,A}}$ for the universal abelian surface $\pi\colon \mca \to \X^\disc_{K,A}$. For any $n\ge 1$ define the \emph{coherent homology group} as in \cite[\S 7.2]{Calegari/Geraghty:2018} by
\[
\hh_i(X^\disc_K,\omega_{\mco/\varpi^n\mco}) := \Hom_{\mco}(\hh^i(X^\disc_K,\omega_{\mco/\varpi^n\mco}),E/\mco)
\]
and define
\[
\hh_i(X^\disc_K,\omega_{\mco}) := \varprojlim_n\hh_i(X^\disc_K,\omega_{\mco/\varpi^n\mco}).
\]
Note that if we define the sheaf $\omega_{E/\mco}$ to be the direct limit  over $n$ of the sheaves $\omega_{\varpi^{-n}\mco/\mco}$ then 
\[
\hh_i(X^\disc_K,\omega_{\mco}) = \Hom_\mco(\hh^i(X^\disc_K,\omega_{E/\mco}), E/\mco)
\]

By the work of \cite[\S 7.2]{Calegari/Geraghty:2018}, one can construct perfect complexes $C_K^{\disc}$ in $\dcat{\mco}$ computing $\hh_i(X^\disc_K,\omega_{\mco})$, which will play a similar in our argument to the complexes $C_K$ from Section \ref{sec:Hecke}.
\begin{proposition}\label{prop:C^D_K wt 1}
Let $K\subseteq D^\times_{\disc}(\A^\infty_{\QQ})$ be a compact open subgroup satisfying the above properties. Then 
\begin{enumerate}[\quad\rm(1)]
	\item Then there is a perfect complex $C_K^{\disc}$ in $\dcat{\mco}$ such that for all $i$:
	\begin{enumerate}[\quad\rm(a)]
		\item $\hh_i(C_K^{\disc}\lotimes_{\mco}k) = 0$ for $i\not\in [0,1]$;
		\item $\hh_i(C_K^{\disc}\lotimes_{\mco}\mco/\varpi^n)\cong \hh_i(X^\disc_K,\omega_{\mco/\varpi^n\mco})$ for all $n\ge 0$;
		\item $\hh_i(C_K^{\disc})\cong \hh_i(X^\disc_K,\omega_{\mco})$.
	\end{enumerate}
	\item Let $K_\Delta\subseteq D^\times_{\disc}(\A^\infty_{\QQ})$ be another compact open subgroup satisfying the above properties we with $K_\Delta\unlhd K$ and $\Delta:= K/K_\Delta$ a finite abelian group of $\ell$-power order. Then there is a perfect complex $C^{\disc}_{K_\Delta}$ in $\dcat{\mco[\Delta]}$ such that for all $i$:
	\begin{enumerate}[\quad\rm(a)]
		\item $\hh_i(C_{K_\Delta}^{\disc}\lotimes_{\mco[\Delta]}k) = 0$ for $i\not\in [0,1]$;
		\item $\hh_i(C_{K_\Delta}^{\disc}\lotimes_{\mco[\Delta]}\mco/\varpi^n[\Delta])\cong \hh_i(X^\disc_{K_\Delta},\omega_{\mco/\varpi^n\mco})$ as $\mco[\Delta]$-modules for all $n\ge 0$;
		\item $\hh_i(C_{K_\Delta}^{\disc})\cong \hh_i(X^\disc_{K_\Delta},\omega_{\mco})$ as $\mco[\Delta]$-modules;
		\item There is a quasi-isomorphism $C_{K_\Delta}^{\disc}\lotimes_{\mco[\Delta]}\mco\cong C_K^\disc$ in $\dcat{\mco}$,
	\end{enumerate}
	where the $\mco[\Delta]$-module structures on $\hh_i(X^\disc_{K_\Delta},\omega_{\mco/\varpi^n\mco})$ and $\hh_i(X^\disc_{K_\Delta},\omega_{\mco})$ are induced by the action of $\Delta$ on $X_{K_\Delta}^{\disc}$.
\end{enumerate}
\end{proposition}

\begin{proof}
Note that (1) is simply the special case of (2) with $K_\Delta = K$, so it is enough to prove (2). The assumption that $K$ is sufficiently small implies that the quotient map $X^{\disc}_{K_\Delta}\to X^{\disc}_K$ is {\'e}tale with Galois group $\Delta$. Thus the work of \cite[\S 7.2]{Calegari/Geraghty:2018} constructs (using \v{C}ech complexes associated to an affine open cover of $X^{\disc}_K$) for each $n\ge 1$ a perfect complex $C^{\disc}_{K_{\Delta},n}$ of $(\mco/\varpi^n\mco)[\Delta]$-modules with:
\begin{align*} 
\hh_i(C^{\disc}_{K_{\Delta},n}) &\cong \hh_i(X^\disc_{K_\Delta},\omega_{\mco/\varpi^n\mco}),\\ C^{\disc}_{K_{\Delta},n+1}\lotimes_{(\mco/\varpi^{n+1}\mco)[\Delta]}(\mco/\varpi^{n}\mco)[\Delta]&\cong C^{\disc}_{K_{\Delta},n},\\
C^{\disc}_{K_{\Delta},n}\lotimes_{(\mco/\varpi^{n}\mco)[\Delta]}(\mco/\varpi^{n}\mco) &\cong C^{\disc}_{K,n}.
\end{align*}
Now \cite[Lemmas 2.13, 2.14]{KT} produce a perfect complex $C^{\disc}_{K_{\Delta}}$ in $\dcat{\mco[\Delta]}$ with $C^{\disc}_{K_{\Delta}}\lotimes_{\mco[\Delta]}(\mco/\varpi^n\mco)[\Delta]\cong C^{\disc}_{K_{\Delta},n}$ and
\[
\hh_i(C^{\disc}_{K_{\Delta}}) \cong  \varprojlim_n\hh_i(X^\disc_K,\omega_{\mco/\varpi^n\mco}) = \hh_i(X^\disc_K,\omega_{\mco}).
\]
It is now straightforward to verify that $C^{\disc}_{K_{\Delta}}$ satisfies all of the listed properties. Note that property (a) follows as $\hh_i(C_{K_\Delta}^{\disc}\lotimes_{\mco[\Delta]}k) = \hh_i(C_{K_\Delta,1}^{\disc}) = \hh_i(X^\disc_{K_\Delta},\omega_{k})$ vanishes for $i\ne 0,1$.
\end{proof}

One can now define actions of double coset operators on the complexes $C^{\disc}_K$ by the procedure described in \cite[\S 7.2]{Calegari/Geraghty:2018} (combined with the work of \cite[\S 2.4]{KT} to lift the actions of the double coset operators on $C^{\disc}_{K,n}$ to an action on $C^{\disc}_K$).

We will now use this to define Hecke algebras analogously to Section \ref{sse:Hecke}. Let $S$ be a set of primes including the primes in $\disc$, the primes $v$ for which $K_v$ is not maximal compact, and the prime $\ell$. For $v\not\in S$, let $T_v\in\End_{\dcat{\mco}}(C_K^\disc)$ denote the double coset operator
\[
T_v = \left[K\begin{pmatrix}\varpi_v&0\\0&1\end{pmatrix}K\right]
\]
also for $v \in S\setminus \disc$ with $v\ne \ell$, $K_1(v)\le K_v\le K_0(v)$ and $d\in K_0(v)/K_v\into (\mco_F/v)^\times$, let $U_v,\langle d\rangle_v\in \End_{\dcat{\mco}}(C_K^\disc)$ denote the double coset operators
\begin{align*}
	U_v &= \left[K\begin{pmatrix}\varpi_v&0\\0&1\end{pmatrix}K\right]&
	&\text{and}&
	\langle d\rangle_v &= \left[K\begin{pmatrix}\widetilde{d}&0\\0&1\end{pmatrix}K\right]
\end{align*}
for any lift $\widetilde{d}\in \ZZ_{v}^\times$ of $d$. Set
\[
 \TT^{S}(K) \colonequals \mco[T_v|v\not\in S ]\subseteq \End_{\dcat{\mco}}(C^\disc_{K})
 \]
which is a finite, commutative $\mco$-algebra. 

Before going on, we should address a subtlety with the definition of $\TT^S(K)$. We have defined $\TT^{S}(K)$ as a subring of $\End_{\dcat{\mco}}(C^\disc_{K})$, whereas most sources (including \cite{Calegari/Geraghty:2018}, which we rely on heavily in the proof of Theorem \ref{rhobar conj-weight one}) define it as a subring of $\End_{\mco}(\hh^0(X^\disc_K, \omega_{E/\mco})) = \End_{\mco}(\hh_0(X^\disc_K, \omega_{\mco})) = \End_{\mco}(\hh^0(C^\disc_{K}))$.  
A priori, this means that the classical Hecke algebra will only be a quotient of $\TT^{S}(K)$, which  makes it difficult to apply available results in our setting.

Fortunately the  result below shows that our definition of $\TT^S(K)$ agrees with the classical definition, and so we may freely use results proved using the classical one.

\begin{proposition}\label{prop:H0 wt 1}
The ring $\TT^S(K)\subseteq \End_{\dcat{\mco}}(C^\disc_{K})$ acts faithfully on $\hh^0(C^\disc_{K})$.
\end{proposition}

For the remainder of this section we  let define $(-)^\dagger\colon \rdcat{\mco}{A}\to \rdcat{\mco}{A}$ by 
\[
C^{\dagger} \colonequals \RHom_\mco(C,\mco)[1].
\]
(more specifically, will take $d=0$ and $l_0=1$ in the results of Sections \ref{se:patching} and \ref{se:duality}). Proposition \ref{prop:H0 wt 1} is a consequence of the following analogue of Proposition \ref{prop: Verdier}:

\begin{proposition}
	\label{prop: duality-weight one}
	For each $K$, there is a natural derived $\TT^S(K)$-equivariant isomorphism, 
	\[
	C^\disc_{K}\cong \RHom_{\mco}(C^\disc_{K},\mco)[1] = (C^\disc_{K})^{\dagger}
	\]
	in $\rdcat{\mco}{\TT^S(K)}$.
\end{proposition}

\begin{proof}
	The existence of this isomorphism follows from the arguments  of \cite[Section 3.2]{Boxer/Pilloni:2020}, using Serre duality applied to the sheaves $\omega_A$ and the Kodaira-Spencer isomorphism, which gives $\omega_A^{\otimes 2}=\Omega_A$ where $\Omega_A$ is the canonical sheaf on the Shimura curve $X^\disc_{\Sigma,Q}$.    After this the Proposition \ref{prop: duality-weight one} follows from  arguments analogous to the ones given in Section \ref{sec:Hecke} (see proof of Proposition \ref{prop: Verdier}).
\end{proof}

In particular, this implies that if $f\in \TT^S(K)$ acts trivially on $\hh^0(C^\disc_{K})$ then its dual $f^{\dagger}$ acts trivially on $\hh^0((C^\disc_{K})^{\dagger}) \cong \hh^0(C^\disc_{K})$ as well. Thus Proposition \ref{prop:H0 wt 1} follows immediately from Lemma \ref{lem:faithful-H0}.

We  also need the existence of a Hasse invariant in this setting.

\begin{lemma}\label{lem: Hasse}
	Let $X$ be  the Shimura curve  $X^\disc_K$. There is an element ${\rm Ha}$ in  $\hh^0(X_k,\omega^{\otimes (\ell-1)})$ (called a Hasse invariant)   which has simple zeros at all the supersingular points of $X_k$, and such that the map  
	\[
	\hh^0(X_k,\omega^{\otimes k}) \to \hh^0(X_k,\omega^{\otimes (k+\ell-1)})
	\] 
	given by multiplication by ${\rm Ha}$ is equivariant for Hecke operators $T_r$ and $U_r$ for all primes $r \neq \ell$ and all weights $k\ge 1$, and is also equivariant for $U_\ell$ if $k\ge 2$.
	
	Further this section lifts to a section $\widetilde{\rm Ha} \in  \hh^0(X_\mco,\omega^{\otimes (\ell-1)})$ such that for all $m\ge 1$ and all weights $k\ge 1$, the map
	\[
	\hh^0(X_{\mco/\varpi^m},\omega_{\mco/\varpi^m}^{\otimes k}) \to \hh^0\left(X_{\mco/\varpi^m},\omega_{\mco/\varpi^m}^{\otimes \left(k+(\ell-1)\ell^{m-1}\right)}\right)
	\] 
	given by multiplication by $\left(\widetilde{\rm Ha}\right)^{\ell^{m-1}}$ is equivariant for Hecke operators $T_r$ for all $r\not\in S$.
\end{lemma}

\begin{proof}
	The section ${\rm Ha}$ exists by \cite[Section 5]{Buzzard:1997}. The integral lift is constructed in \cite[Section 7]{Kassaei:2004}; this relies on the assumption that $\ell>3$.
	
	The final statement about Hecke equivariance follows by the argument given in \cite[Proposition 7.2.1]{Boxer:2015}. The key point is that while the lift $\widetilde{\rm Ha} \in  \hh^0(X_\mco,\omega^{\otimes (\ell-1)})$ is not canonical, the image of $\left(\widetilde{\rm Ha}\right)^{\ell^{m-1}}$ in $\hh^0\left(X_{\mco/\varpi^m},\omega_{\mco/\varpi^m}^{\otimes (\ell-1)\ell^{m-1}}\right)$ will be independent of the choice of $\widetilde{\rm Ha}$, which allows one to mimic the proof that multiplication by ${\rm Ha}$ itself if Hecke equivariant.
\end{proof}

\subsection{The conjectures \ref{rhobar conj} and \ref{R->T conj}  in weight one}
Here is the analog of Conjecture \ref{rhobar conj} in this setting; we thank George Boxer for explaining the argument to us.

\begin{theorem}
\label{rhobar conj-weight one}
The following statements hold.
	\begin{enumerate}[\quad\rm(i)]
		\item
		For every maximal ideal $\fm$ of $\TT^{S}(K)$, there is a  semisimple  Galois representation $\rhobar_{\fm}\colon G_{\QQ}\to \GL_2(\TT^{S}(K)/\fm)$ such that for all $v\not\in S$, $v \ne \ell$, $\rhobar_{\fm}|_{G_{F_v}}$ is unramified with $\tr \rhobar_{\fm}(\Frob_v) = T_v\pmod{\fm}$, and $\rhobar_{\fm}|_{G_{\ell}}$ is unramified.
		\item
		Furthermore there is a lift of $\rhobar_{\fm}$ to a representation
		\[
		\rho_K\colon G_{\QQ}\to \GL_2(\TT^{S}(K)_\fm)
		\]
		such that  for all $v\not\in S$,  $v \ne \ell$, $\rho_{K}|_{G_{F_v}}$ is unramified and $\rho_{K}(\Frob_v)$ has characteristic polynomial $x^2-T_vx+\psi(\Frob_v)$, where $\psi:G_{\QQ}\to \mco^\times$ is the Teichmuller lift of $\det \rhobar_{\fm}$, and $\rho_K|_{G_\ell}$ is unramified.
	\end{enumerate}
\end{theorem}

\begin{proof}
	The proof follows the ``doubling'' strategy used in the proof of \cite[Theorem 3.11]{Calegari/Geraghty:2018} in the case of modular curves. All but one step of proof in loc. cit. works for Shimura curves with no modifications.
	
Abbreviate $X^\disc_K$ by $X$. Take any integer $m\ge 1$. By Lemma \ref{lem: Hasse} there exits some power, $A$, of $\widetilde{\rm Ha}$ such that the map 
\[
\hh^0(X_{\mco/\varpi^m},\omega_{\mco/\varpi^m}) \to \hh^0\left(X_{\mco/\varpi^m},\omega_{\mco/\varpi^m}^{\otimes  n}\right)
\]
given by multiplication by $A$ is equivariant for the Hecke operators $T_r$ for all $r\not\in S$, where $n-1$ is the weight of $A$.

This implies that the ``doubling map''
\[\hh^0(X,\omega^{}_{E/\mco})^2[\varpi^m]  \to \hh^0\left(X,\omega^{\otimes n}_{E/\mco}\right)\]
given by $(f,g) \mapsto Af + AT_\ell(g)-U_\ell(Ag)$ is also equivariant for all of the Hecke operators for all $r\not\in S$. Here, $T_\ell$ is the Hecke operator acting on $\hh^0(X,\omega^{}_{E/\mco})$. This can be defined exactly as in \cite[Section 3.1]{Boxer/Pilloni:2020}.\footnote{Note that while that argument is written for the modular curve, it does not rely on the $q$-expansion principle, and in fact holds without modification for Shimura curves.}

The proof in \cite{Calegari/Geraghty:2018} reduces (by an argument that still works for Shimura curves) to showing that this doubling map is injective. Moreover, as in loc. cit. it is sufficient to prove this claim in the case $m=1$. Indeed, as $\hh^0(X,\omega^{}_{E/\mco})^2[\varpi^m]$ is a finite length $\mco$-module, it suffices to prove that the map
\[\hh^0(X,\omega^{}_k)^2  = \hh^0(X,\omega^{}_{E/\mco})^2[\varpi]  \to \hh^0\left(X,\omega^{\otimes n}_{E/\mco}\right)[\varpi] = \hh^0(X,\omega^{\otimes n}_k)\]
is injective. Moreover, as the multiplication by ${\rm Ha}$ map 
\[\hh^0\left(X,\omega^{\otimes k}_k\right)\to \hh^0\left(X,\omega^{\otimes (k+\ell-1)}_k\right)\]
commutes with $U_\ell$ for $k\ge 2$, it will actually suffice to show that the map 
\[ \hh^0(X,\omega^{}_k)^2  \to \hh^0(X,\omega^{\otimes \ell}_k)\]  
given by $(f,g) \mapsto {\rm Ha}\cdot f + {\rm Ha}\cdot T_\ell(g)-U_\ell({\rm Ha}\cdot g)$ is injective.

In loc. cit.  this is proven via $q$-expansions, which are not available for Shimura curves. However the arguments in \cite[Section 5.1]{BCGP} give an alternative proof of this injectivity which does not rely on $q$-expansions, which completes the proof.
	
In the specific case of Shimura curves however, one can substantially simplify the proof given in \cite[Section 5.1]{BCGP}. If $(f,g)$ is in the kernel of the doubling map, then letting $f_0 = f+T_\ell(g)$, one gets $U_\ell({\rm Ha}\cdot g) = {\rm Ha}\cdot f_0$ for $f_0,g\in \hh^0(X,\omega_k)$. As in \cite[(5.1.1)]{BCGP}\footnote{Again, while \cite[\S 5.1]{BCGP} only considers weight one forms on a modular curve, the construction of this diagram does not rely on the $q$-expansion principle, and works without modification in the case of Shimura curves}, one can now consider the following commutative diagram:
	\[
	\begin{tikzcd}
		\hh^0(X,\omega_k) \arrow{d} \arrow{r}{f\mapsto U_\ell({\rm Ha}\cdot f)} &[3em] \hh^0(X,\omega^{\otimes \ell}_k) \arrow{d}  \\
		\hh^0(SS,\omega_k)  \arrow{r}{\sim} & \hh^0(SS,\omega^{\otimes \ell})
	\end{tikzcd}
	\]
	where $SS\subseteq X$ is the supersingular locus. Here, the vertical maps are restriction maps and the lower horizontal map is an isomorphism. As ${\rm Ha}$ vanishes on $SS$, so does $U_\ell({\rm Ha}\cdot g) = {\rm Ha}\cdot f_0$. Thus $U_\ell({\rm Ha}\cdot g)$ maps to $0$ in $\hh^0(SS, \omega^{\otimes \ell})$, and so the commutative diagram implies that the restriction of $g$ to $SS$ must vanish. Thus $ g = {\rm Ha} \cdot h$ for some $g \in  \hh^0(X,\omega^{2-\ell})$. But as $\ell>2$, this cohomology group vanishes, so $g = 0$, and hence $(f,g) = (0,0)$.
	\end{proof} 

From now on we  fix a Galois representation $\rhobar\colon G_\QQ\to \GL_2(k)$ for which $\rhobar|_{\QQ(\zeta_\ell)}$ is absolutely irreducible and $\rhobar|_{G_\ell}$ is unramified. Let $N_\es = N(\rhobar)$ denote the conductor of $\rhobar$. We take $\psi\colon G_\QQ\to \mco^\times$ to be the Teichmuller lift of $\det \rhobar$.

We now make more particular choices of the compact open subgroups $K \subseteq D_\disc(\A_\QQ^\infty)$. Let $T$ and $Q$ be finite sets of primes, disjoint from each other and from $\disc$, satisfying the conditions from Section \ref{ssec:global def}. Let $t$ be  an auxiliary  prime  that that is not in $T\cup Q\cup\disc$ and satisfies the conditions of Lemma \ref{trivial} above for our choice of $\rhobar$.

As in Section \ref{sec:Ihara}, for any integer $N\ge 1$ with $(N,\disc) = 1$ we define (identifying $D_\disc(\ZZ_v) = \GL_2(\ZZ_v)$ for $v\not\in \disc$):
\begin{align*}
	K_0^\disc(N) &= \left\{
	\begin{pmatrix}
		a&b\\c&d
	\end{pmatrix}\in D_\disc^\times(\Zhat)\middle|c\equiv 0\pmod{N t^2}, d \equiv 1 \pmod{ t^2 }\right\}\subseteq D_\disc^\times(\Zhat)\\
	K_1^\disc(N) &= \left\{
	\begin{pmatrix}
		a&b\\c&d
	\end{pmatrix}\in D_\disc^\times(\Zhat)\middle|c\equiv 0\pmod{N t^2}, d \equiv 1 \pmod{N t^2 }\right\}\subseteq D_\disc^\times(\Zhat)
\end{align*}
so that again $K_1^\disc(N)\unlhd K_0^\disc(N)$ with $K_0^\disc(N)/K_1^\disc(N)\cong (\ZZ/N\ZZ)^\times$, and $K_0^\disc(N)$ and $K_1^\disc(N)$ are sufficiently small. Also define $K^\disc_\Delta(N)$ and $K^\disc_H(N)$ to be the smallest intermediate subgroups
\[
K^\disc_1(N)\le K^\disc_\Delta(N), K^\disc_H(N)\le K^\disc_0(N)
\]
for which $|K^\disc_0(N)/K^\disc_\Delta(N)|$ is an $\ell^{th}$ power and $|K^\disc_0(N)/K^\disc_H(N)|$ is prime to $\ell$. For any $M$ and $N$ define $K_{H,\Delta}^{\disc}(N,M) = K_H^\disc(N)\cap K^\disc_\Delta(M)$. 

 \label{pg:intermediate-subgroups} 

For any subset $\Sigma\subseteq T$, set $N_\Sigma= N_\phi\Sigma$, where by  $\Sigma$ we mean the product of the primes in $\Sigma$. 

For convenience, write $K_{\Sigma,Q}^\disc=K_{H,\Delta}^{\disc}(N_\Sigma,Q)$, and $K_{\Sigma}^\disc = K_{\Sigma,\es}^\disc = K_H^\disc(N_\Sigma)$. Let $X^\disc_{\Sigma,Q} = X^\disc_{K^\disc_{\Sigma,Q}}$, $\rho^\disc_{\Sigma,Q} = \rho_{K^\disc_{\Sigma,Q}}$, $X^\disc_\Sigma = X^\disc_{\Sigma,\es}$ and $\rho^\disc_{\Sigma} = \rho^\disc_{\Sigma,\es}$. 

By Proposition \ref{prop:C^D_K wt 1}, $C^{\disc}_{K_{\Sigma,Q}^\disc}$ can be given the structure of a perfect complex of $\mco[\Delta_Q]$-modules. By construction, the action of $\mco[\Delta_Q]$ defined by this structure coincides with the action of $\mco[\Delta_Q]$ on $C^{\disc}_{K_{\Sigma,Q}^\disc}$ defined using the diamond operators $\langle d\rangle_v$ for $v\in Q$.

Let $S$ be a finite set of primes containing $T\cup Q\cup \disc$, all primes dividing $N_\es$ and the prime $t$. From now on we  assume that $\rhobar$ arises (in the sense of Theorem \ref{rhobar conj-weight one}) from a maximal ideal $\fm$ of the Hecke algebra $\TT^{S}(K^\disc_\es)$ acting on $\hh^1(X^\disc_\es,\omega_\mco)$ (which is isomorphic to  $\hh^0(X^\disc_\es, \omega_{E/\mco})^*$).  

Let  $\TT^\disc_{\Sigma,Q}$ be the localization $\TT^{S}(X^\disc_{\Sigma,Q})_\fm$. We  observe now that the analog of Conjecture \ref{R->T conj}  can be proved  in this setting by the variant of doubling method used in the proof of Theorem \ref{rhobar conj-weight one}.

\begin{theorem}
\label{R->T conj-weight one}
	For any $\Sigma$, $Q$ and $\disc$, the representations $\rho^\disc_{\Sigma}\colon G_F\to \GL_2(\TT^\disc_{\Sigma})$  stemming from Theorem  \ref{rhobar conj-weight one}    arise from  the corresponding universal representation  $G_F \to \GL_2( R^\disc_{\Sigma,Q})$ via maps $R^\disc_{\Sigma,Q}\onto \TT^\disc_{\Sigma,Q}$  and satisfy the following properties:
	\begin{enumerate}[\quad\rm(1)]
		\item 
		$\tr\rho^\disc_{\Sigma,Q}(\Frob_v) = T_v$ for all $v\not\in S$.
		\item There is a natural $\mco[\Delta_Q]$-algebra structure on $\TT^\disc_{\Sigma,Q}$ defined analogously to as in Section \ref{sec:Hecke} (via the full Hecke algebra) which makes the map $R^\disc_{\Sigma,Q}\onto \TT^\disc_{\Sigma,Q}$ into a $\mco[\Delta_Q]$-algebra homomorphism. 
		\item For $v\in\Sigma$, writing $B_v\in \TT_\Sigma^{\disc,\square}$ for the image of $B_v\in R_v^\square$   under the map 
		\[
		R_v^\square\into R_{\Sigma,\loc}\to R_\Sigma^\square \onto \TT_\Sigma^{\disc,\square}
		\] 
		we have $(B_v) = (\epsilon_v\sqrt{q_v\psi(\varphi_v)^{-1}}U_v-1)$ as ideals of $\TT_\Sigma^{\disc,\square}$. 
			\item For $\Sigma_1\subseteq \Sigma_2\subseteq T$, the diagram below commutes:
			\begin{center}

		$
		\begin{tikzcd}
			R_{\Sigma,Q}^\disc \arrow[twoheadrightarrow]{d} \arrow[twoheadrightarrow]{r} & \TT_{\Sigma,Q}^\disc \arrow[twoheadrightarrow]{d}  \\
			R_{\Sigma}^\disc \arrow[twoheadrightarrow]{r} & \TT_{\Sigma}^\disc 
		\end{tikzcd}
		$
		and
		$
		\begin{tikzcd}
			R_{\Sigma_2,Q}^\disc \arrow[twoheadrightarrow]{d} \arrow[twoheadrightarrow]{r} & \TT_{\Sigma_2,Q}^\disc \arrow[twoheadrightarrow]{d}  \\
			R_{\Sigma_1,Q}^\disc \arrow[twoheadrightarrow]{r} & \TT_{\Sigma_1,Q}^\disc 
		\end{tikzcd}
		$ 
\end{center}
	\end{enumerate} 
\end{theorem}

For any $\Sigma\subseteq T$,  $Q$ and $v\in T\sm \Sigma$, we can now define the complex $C_{\Sigma,Q}^\disc\in \rdcat{\mco}{\TT_{\Sigma,Q}}$ and the level lowering map $\pi_v\colon C_{\Sigma\cup\{v\},Q}^\disc\to C_{\Sigma,Q}^\disc$ completely analogously to the definition given in Section \ref{sec:Hecke} (that is, by localizing at an appropriate maximal ideal of an extension of $\TT_{\Sigma,Q}^\disc$).

Note that $C_{\Sigma,Q}^\disc$ again has the structure of a perfect complex of $\mco[\Delta_Q]$-modules, and we have $C_{\Sigma,Q}^\disc \lotimes_{\mco[\Delta_Q]}\mco\cong C_\Sigma^\disc$ in $\rdcat{\mco}{\TT_{\Sigma}}$.

\begin{proposition}
	\label{prop: localized duality-weight one}
There is a family of isomorphisms 
\[
\varphi_{\Sigma,Q}^\disc \colon C_{\Sigma,Q}^\disc \isomto { (C_{\Sigma,Q}^{\disc} )}^\dagger
\]
in $\rdcat{\mco}{\TT_{\Sigma,Q}}$ compatible with the isomorphisms $C_{\Sigma,Q}^\disc \lotimes_{\mco[\Delta_Q]}\mco\cong C_{\Sigma}^\disc$.

Moreover for each $v\in T\sm\Sigma$, if 
\[
\pi_v^\dagger\colon C_{\Sigma,Q}^\disc \cong  { (C_{\Sigma,Q}^{\disc} ) }^\dagger \to  { (C_{\Sigma\cup \{v\},Q}^\disc)}^\dagger \cong C_{\Sigma\cup\{v\},Q}^\disc
\]
is the adjoint of $\pi_v\colon C_{\Sigma\cup\{v\},Q}^\disc \to C_{\Sigma,Q}^\disc$ upon using the isomorphisms  $\varphi_{\Sigma,Q}$ and $\varphi_{\Sigma\cup\{v\},Q}$, then one has
\begin{align*}
	&\pi_v\circ \pi_v^\dagger \simeq  q_vU_v^2-\psi(\varphi_v)\in \End_{\dcat{\mco}}(C_{\Sigma,Q}^\disc)  \\
	&\pi_v^\dagger\circ \pi_v  \simeq  q_vU_v^2-\psi(\varphi_v)\in \End_{\dcat{\mco}}(C_{\Sigma\cup\{v\},Q}^\disc),
\end{align*} where $\simeq$  means up to units in the corresponding Hecke algebras. 
\end{proposition}

\begin{proof}
Using the arguments  of \cite[Section 3.2]{Boxer/Pilloni:2020}, we deduce the existence of  isomorphisms  
\[
\varphi^\disc_{\Sigma,Q}\colon C^\disc_{\Sigma,Q}\isomto  {   (C^\disc_{\Sigma,Q})   }^\dagger
\]
 using Serre duality applied to the sheaves $\omega_A$ and using the Kodaira-Spencer isomorphism that gives $\omega_A^{\otimes 2}=\Omega_A$ where $\Omega_A$ is the canonical sheaf on the Shimura curve $X^\disc_{\Sigma,Q}$.    After this the Proposition \ref{prop: localized duality-weight one} follows from  arguments analogous to the ones given in Section \ref{sec:Hecke} (see proof of Proposition \ref{prop: localized duality}).
\end{proof}

\subsection{The conjectures \ref{vanishing conj} and \ref{Ihara conj}  in weight one}

For any curve $X$, the cohomology groups $\hh^i(X,\mathcal F)$ with coefficients in a coherent sheaf $\mathcal F$  are non-zero only for $i=0,1$. It follows from the definition of $C_{K}^\disc$ that we deduce that $\hh_i(C_K^\disc\lotimes_\mco k) = 0$ for $i\not\in[0,1]$. Thus the analogue of Conjecture \ref{vanishing conj} holds in this setting.

The analogue of conjecture \ref{Ihara conj} follows from  the following proposition:

\begin{proposition}\label{prop: Ihara wt 1} 
	The map 
\[
\hh^0(\pi_v)\colon \hh^0(X^\disc_{\Sigma},\omega_{E/\mco})^*_{\fm_\Sigma} \to   \hh^0(X^\disc_{\Sigma\sm \{v\}},\omega_{E/\mco})^*_{\fm_{\Sigma \sm\{v\}}}
\] 
for $v \in \Sigma$  is surjective.
\end{proposition}

\begin{proof}
	By duality it suffices to prove that the kernel of the  map  $\hh^0(X^\disc_{\Sigma\sm \{v\}},\omega_k)^2 \to  \hh^0(X^\disc_{\Sigma},\omega_k)$ is Eisenstein (and in particular vanishes on localizing at $\fm_\Sigma$).  This follows from \cite[Proposition 5]{DT2}.
\end{proof}

\subsection{$R=\TT$ theorems and the torsion Jacquet--Langlands correspondence in the weight one case}

In Section \ref{sec:main} we  prove the following Theorem (as a consequence of Theorem \ref{R=T theorem} in Case \ref{case:wt1}):

\begin{theorem}\label{weightone:thm}
	The surjective maps   $R^\disc_\Sigma \to \TT^\disc_\Sigma$ are isomorphisms.
\end{theorem}

As a corollary of this, we deduce  a  torsion Jacquet--Langlands correspondence for weight one modular forms.

\begin{theorem}\label{JLcorr}
	Consider  a residual representation $\rhobar: G_\QQ 
	\to \GL_2(k)$ that  arises from a maximal ideal $\fm_\disc$ of the Hecke algebra acting on $\hh^1(X^\disc_\Sigma,\omega_\mco)$.  We assume that $\rhobar|_{\QQ(\zeta_\ell)}$ is irreducible. Then:
\begin{enumerate}[\quad\rm(1)]	
\item
$\rhobar$ also arises from a maximal ideal $\fm_\es$ of the Hecke algebra acting on $\hh^1(X^\es_{\Sigma\cup\disc},\omega_\mco)$.
\item	
For any set $\Sigma$ of level raising primes satisfying the conditions from Section \ref{ssec:global def}, if $\TT^\es_{\Sigma\cup\disc}$ and $\TT^\disc_\Sigma$ are the Hecke algebras acting on $\hh^1(X^\es_{\Sigma\cup \disc},\omega_\mco)$ and $\hh^1(X^\disc_{\Sigma},\omega_\mco)$, localized at $\fm_\es$ and $\fm_\disc$ respectively, then there is a natural surjective map $\TT_{\Sigma\cup\disc}^{\es}  \onto \TT_\Sigma^{\disc}$ with kernel generated by $q_vU_v^2 -\psi(\Frob_v)$ for $v \in \disc$.
\end{enumerate}
\end{theorem}

\begin{proof}
	  We first observe that by  using the duality   $\hh^0(X^\disc_\Sigma, \omega_{E/\mco})^*=\hh^1(X^\disc_\Sigma,\omega_\mco)$, and  that  $\hh^0(X^\disc_{\Sigma},\omega_k)=\hh^0(X^\disc_\Sigma, \omega_{E/\mco})[\varpi]$, we deduce that
 $\rhobar$ arises from $\hh^0(X^\disc_{\Sigma},\omega_k)$. By multiplying by the Hasse invariant we  then get  that  $\rhobar$ also arises from  $\hh^0(X^\disc_{\Sigma},\omega_k^{\otimes \ell})$. 
 
 Let $X= X^\disc_{\Sigma}$. Consider the long exact sequence of cohomology arising from the exact sequence of sheaves on $X$:
	\[ 0 \rightarrow \omega_\mco^{\otimes{\ell}} \rightarrow \omega_\mco^{\otimes{\ell}} \rightarrow \omega_k^{\otimes{\ell}} \rightarrow 0\] (where the second arrow is multiplication by  a uniformizer $\varpi$ of $\mco$). We claim that   $\hh^1(X_\mco, \omega_\mco^{\otimes{\ell}})[\varpi]=0$. Indeed  $\hh^1(X_\mco, \omega_\mco^{\otimes{\ell}})$ vanishes as  using  Serre duality,  and the Kodaira-Spencer  isomorphism $\omega^{\otimes 2}\simeq \Omega_X$,  this follows from the vanishing of   $\hh^0(X_k, \omega_\mco^{\otimes{(2-\ell)}})$ (for $\ell>2$).  
	
	From the claim,  and  the fact that   $\rhobar$  arises from  $\hh^0(X^\disc_{\es},\omega_k^{\otimes \ell})$, we deduce that $\rhobar$ arises from a classical weight $\ell$  form $f$  that is a section of  $\hh^0(X^\disc_{\es},\omega_\mco^{\otimes \ell})$.
	  
	  To prove the existence of $\fm_\es$,  we can now apply the classical Jacquet--Langlands correspondence to $f$ and get a  classical weight $\ell$ form $f_\es$ on $X^\es_\disc$ with residual representation $\rhobar$. One can then apply level and weight optimization on the modular curve  to get the desired maximal ideal $\fm_\es$ producing $\rhobar$. Here one uses crucially the companion forms results of  \cite{Gross:1990} and \cite{Coleman/Voloch:1992}   to lower the weight from $\ell$ to 1 using  that  the mod $\ell$ representation $\rhobar$ that $f_\es$ gives rise to is unramified at $\ell$. This proves (1). 
	
	  To prove (2), using Theorem \ref{weightone:thm}  (for $\disc\ne\es$) and the results in \cite{Calegari}  (for $\disc=\es$) and part (1), we deduce   that for any set $\Sigma$ of level raising primes for $\rhobar$ that  satisfy  our earlier conditions,  we have   isomorphisms $R^\es_{\Sigma\cup\disc}\cong \TT_{\Sigma\cup\disc}^{\es}$ and $R^\disc_\Sigma \cong \TT_\Sigma^{\disc}$, so the claim  follows from the analogous statement for the map  $R^\es_{\Sigma\cup\disc}  \to R^\disc_\Sigma$, which is clear from the definitions.
\end{proof}

\begin{remark}\label{JLremark}
	
	If one had available results about  companion forms  for mod $\ell$ representations on  Shimura curves one could  show that if  $\rhobar$ arises from   $\hh^1(X^\disc_{\Sigma},\omega_\mco)$ then it also arises from  $\hh^1(X^{\disc’}_{{\Sigma}\cup {\disc}\backslash {\disc’}},\omega_\mco)$ for  all $\disc’ \subseteq \disc$, and not just $\disc’=\es$.    This will need results for companion forms on Shimura curves (to lower the weight from $\ell$ to 1 when the mod $\ell$ representation is unramified at $\ell$) that would be the analogs of the results of Gross and Coleman--Voloch  \cite{Gross:1990} and \cite{Coleman/Voloch:1992}  on modular curves.  This also is the obstruction to going from a maximal ideal  $\fm_\Sigma$ of the Hecke algebra acting on  $\hh^1(X^\es_{\Sigma \cup \disc},\omega_\mco)$ to a  maximal ideal 
  $\fm_\disc$ of the Hecke algebra acting on  $\hh^1(X^\disc_{\Sigma},\omega_\mco)$ (under the necessary assumption that the primes in $\disc$ are level raising primes for $\rhobar$ arising from $\fm_\es$).  Toby Gee has informed us that the possibility of proving a result like Theorem \ref{JLcorr} had been considered earlier  by him in work with Boxer and  Calegari. \end{remark}

\section{Applications to integral modularity lifting}
\label{sec:main}

In this section we apply the results from Part~\ref{part:ca} in order to prove $R=\TT$ theorems in non-minimal cases. The essential strategy is the same as the arguments given in \cite[Chapter 2]{Wiles:1995} and \cite[Theorem 3.4]{Diamond:1997}. Namely, one first proves the result in the minimal level case via patching --- in our case the relevant patching result is already known, see \cite[Theorem 1.3]{Calegari/Geraghty:2018}. One then proceeds by induction on the level, using Ihara's Lemma to control the growth of congruence modules, and then using the numerical criterion to prove the isomorphism $R_\Sigma=\TT_\Sigma$. 

The key difference in our approach is that we apply the inductive argument and the numerical criterion directly to the patched modules $M_{\Sigma,\infty}$, rather than applying them to the global deformation rings and Hecke algebras. This allows us to sidestep many of the complications that arise by working over non totally real fields or with weight $1$ modular forms (the ``$\ell_0>0$'' case) as the patched situation in these cases behaves quite similarly to that for weight $2$ modular forms over $\mbb Q$. This modification to the argument is what necessitates working with congruence modules in higher codimension.

In both minimal and non-minimal cases the patched modules $M_{\Sigma,\infty}$ are known to be maximal Cohen--Macaulay over the corresponding patched deformation rings $R_{\Sigma,\infty}$. In the minimal case (i.e. $\Sigma=\es$) $R_{\es,\infty}$ is a power series ring, which is well-known to imply $M_{\es,\infty}$ is free, giving our base case. However in the non-minimal case our patched rings $R_{\Sigma,\infty}$ are not power series rings, and have multiple components (although they are quite nice and easy to describe, and in particular are complete intersections in our case). This causes difficulty in analyzing the structure of patched modules in the non-minimal case purely via patching.

We show by induction on $\Sigma$, using our version of the numerical criterion, Theorem \ref{th:diamond} (although for our purposes Theorem \ref{th:gorenstein} is sufficient, as we already know that relevant local deformation rings are complete intersections), that each $M_{\Sigma,\infty}$ has a free direct summand. The desired statement about global deformation rings and Hecke algebras (Theorem \ref{R=T theorem}) follows immediately from this.  

Proposition \ref{pr:change-modules}, together with Theorem \ref{th:invariance-of-domain}, provides the necessary generalization of the growth of congruence modules argument of \emph{loc. cit.}, provided one assumes the version of Ihara's Lemma given in Conjecture \ref{Ihara conj} (which is known in the case of Bianchi manifolds by Theorem \ref{Ihara's Lemma}) or Proposition \ref{prop: Ihara wt 1}.

One should note that the augmentation $\lambda\colon R_{\Sigma,\infty}\onto \mco$ which we define below in order to apply our numerical criterion is picked essentially arbitrarily, and need have no relation to the global deformation rings and Hecke algebras. This means in particular that our method does not need to distinguish between the cases where the ring $\TT$ has characteristic $0$ points and the case when it does not.

One slight complication that arises in our argument, compared to that of \emph{loc. cit.}, is that we have no direct control over the generic rank of the patched modules $M_{\Sigma,\infty}$, unless the Hecke algebra $\TT_\Sigma$ has at least one characteristic $0$ point lying on each component of $\operatorname{Spec} R_{\Sigma,\infty}$. Fortunately Theorem \ref{th:gorenstein} needs no assumption about the generic rank of the module (in contrast to \cite[Theorem 2.4]{Diamond:1997}), so this does not present an issue for our main application to proving $R=\TT$. This does however prevent us from proving an analogue of the `multiplicity one' result of \cite[Theorem 3.4]{Diamond:1997}---instead we can only prove the weaker statement in Theorem \ref{R=T theorem}. Note that this issue is already present in \cite[Theorem 1.3]{Calegari/Geraghty:2018}, although it is somewhat more severe in our case as $\operatorname{Spec} R_{\Sigma,\infty}$ has multiple components and one cannot conclude that $M_{\Sigma,\infty}$ has the same generic rank on each component.

We  treat the two cases \ref{case:PGL2} and \ref{case:wt1} in parallel here. To fix notation:
\begin{itemize}
	\item In case \ref{case:PGL2} we freely use all notation introduced in Section \ref{sec:Hecke}, and assume conjectures \ref{rhobar conj}, \ref{R->T conj}, \ref{vanishing conj} and \ref{Ihara conj}. In particular we consider a non-Eisenstein maximal ideal $\fm\subseteq \TT^S(K_0(\mcn_\es))$ with residue field $k$, and the corresponding Galois representation $\rhobar_{\fm}\colon G_F\to \GL_2(k)$. We  assume that $\fm$ was chosen so that $N(\rhobar_\fm) = \mcn_\es$. 
	
	For each $\Sigma\subseteq T$ we then have rings $R_{\Sigma,Q}$ and $\TT_{\Sigma,Q}$, a surjective map $R_{\Sigma,Q}\onto \TT_{\Sigma,Q}$ from Conjecture \ref{R->T conj} and a complex $C_{\Sigma,Q}\in\rdcat{\mco}{\TT_{\Sigma,Q}}$. These will again be denoted $R_\Sigma$, $\TT_\Sigma$ and $C_\Sigma$ when $Q = \es$.
	
	We  again let $r_1$ and $r_2$ be the number of real and complex places of $F$ respectively, and let $d = r_1+r_2$ and $l_0 = r_2$. Note that by Conjecture \ref{vanishing conj}, $\hh_i(k\lotimes_{\mco}C_{\Sigma}) = 0$ for $i\not\in [d,d+\ell_0]$ and $\hh_d(C_{\Sigma}) = \hh_{r_1+r_2}(Y_0(\mcn_\Sigma),\mco)_{\fm_\Sigma}$.
	
	\item In case \ref{case:wt1} we freely use all notation introduced in Section \ref{sec:Hecke wt1}. Again, we consider a non-Eisenstein maximal ideal $\fm\subseteq\TT^{S}(K^\disc_\es)$ with residue field $k$, and the corresponding Galois representation $\rhobar_{\fm}\colon G_\QQ\to \GL_2(k)$. We  assume that $\fm$ was chosen so that $N(\rhobar_\fm) = N_\es$. 
	
	The discriminant $\disc$ is fixed throughout this section, and so omit it from our notation. 
	 Thus we consider rings $R_{\Sigma,Q}$ and $\TT_{\Sigma,Q}$, a surjective map $R_{\Sigma,Q}\onto \TT_{\Sigma,Q}$ and a complex $C_{\Sigma,Q}\in\rdcat{\mco}{\TT_{\Sigma,Q}}$. These will again be denoted $R_\Sigma$, $\TT_\Sigma$ and $C_\Sigma$ when $Q = \es$.
	
	Here we  let $d=0$ and $l_0 = 1$. Then we again have $\hh_i(k\lotimes_{\mco}C_{\Sigma}) = 0$ for $i\not\in [d,d+\ell_0]$ and $\hh_d(C_{\Sigma}) = \hh^0(X^{\disc}_\Sigma,\omega_{E/\mco})^*_{\fm_\Sigma}$.
\end{itemize}

Our main result is the following, which is Theorem \ref{th:intro-5} in Case \ref{case:PGL2} and Theorem \ref{th:intro-5b} in Case \ref{case:wt1}:

\begin{theorem}
\label{R=T theorem}
Assume that $\rhobar_\fm|_{G_{F(\zeta_\ell)}}$ is absolutely irreducible. Then there exists an integer $\mu\ge 1$ such that for each $\Sigma\subseteq T$ there exists an $R_\Sigma$-module $W_\Sigma$ and an isomorphism of $R_\Sigma$-modules
\[
\hh_d(C_{\Sigma})\cong R_\Sigma^\mu \oplus W_\Sigma\,.
\]
In particular, $R_\Sigma$ acts faithfully on $\hh_d(C_{\Sigma})$ and so the map $R_\Sigma\onto \TT_\Sigma$ is an isomorphism for all $\Sigma$.
\end{theorem}

This result is deduced from Theorem~\ref{thm:free summand} below. 
As in Section \ref{ssec:global def}, let $S$ be the set of all finite places of $F$ where either $\rhobar_\fp|_{G_{F_v}}$ ramifies or $v|\ell$. Let $j = 4|S\cup T\cup \disc|-1$, so that for all $\Sigma$ and $Q$ one has
\[
R_\Sigma^\square = R_\Sigma\pos{w_1,\ldots,w_j} \quad\text{and}\quad R_{\Sigma,Q}^\square = R_{\Sigma,Q}\pos{w_1,\ldots,w_j}\,.
\]
The  result below that is by now standard  (see for instance \cite[Proposition 5.2]{Calegari/Geraghty:2018})  allows us to construct sets of Taylor--Wiles primes $Q_n$:

\begin{proposition}
\label{prop: TW primes}
Assume that $\rhobar_\fm|_{F(\zeta_\ell)}$ is absolutely irreducible. Then there exist integers $r,g\ge 1$ such that for all integers $n\ge 1$ there is a set of finite places $Q_n$ of $F$ for which:
\begin{itemize}
	\item $|Q_n| = r$;
	\item $Q_n\cap (S\cup T\cup\disc) = \es$;
	\item $\Nm(v)\equiv 1\pmod{\ell^n}$ for each $v\in Q_n$;
	\item $\rhobar_\fm$ is unramified at $v$ and $\rhobar_\fm(\Frob_v)$ has distinct eigenvalues in $k$ for each $v\in Q_n$;
	\item The map $R_{\loc,T}\to R_{T,Q_n}^\square$ extends to a surjection
	\[
	R_{\loc,T}\pos{x_1,\ldots,x_g}\onto R_{T,Q_n}^\square
	\]
	 thus inducing compatible surjections $R_{\loc,\Sigma}\pos{x_1,\ldots,x_g}\onto R_{\Sigma,Q_n}^\square$ for  $\Sigma\subseteq T$.
\end{itemize}
Moreover we have 
\[
\dim R_{T,\loc}\pos{x_1,\ldots,x_g} = \dim R_{T,\loc} + g = (1+j+r)-l_0\,. 
\]
\end{proposition}

Now for each $\Sigma\subseteq T$ let $R_{\Sigma,\infty} = R_{\Sigma,\loc}\pos{x_1,\ldots,x_g}$. Also let
\[
S_\infty = \mco\pos{{\bs y},{\bs w}}=\mco\pos{y_1,\ldots,y_r,w_1,\ldots,w_j}
\]
so that $\dim R_{\Sigma,\infty} = \dim S_\infty - l_0$ for all $\Sigma$. For $v\in \Sigma$, let $A_v,B_v\in R_v^{{\rm uni}(\epsilon_v)}\subseteq R_{\Sigma,\infty}$ be as in Proposition \ref{prop:unipotent-deformations}, so that $R_{\Sigma\cup\{v\},\infty}/(A_v) = R_{\Sigma,\infty}$.

Write $c = j+r-l_0$, so that $\dim R_{\Sigma,\infty} = c+1$ for all $\Sigma$.

For any $n\ge 1$, let $\Lambda_n=\mco\pos{\Delta_{Q_n}}$. Since $|Q_n| = r$, $\Delta_{Q_n}$ is isomorphic to $\prod_{i=1}^r (\mbb Z/\ell^{e(i,n)}\mbb Z)$ for some integers $e(i,n)\ge n$. Thus we may write $\Lambda_n = \mco\pos{{\bs y}}/I_n$ for some ideal $I_n\subseteq \mco\pos{{\bs y}}$ as in Section \ref{ch:oy}.

Now let $R_{\Sigma,n} = R_{\Sigma,Q_n}$ and $C_{\Sigma,n} = C_{\Sigma,Q_n}$. Also let $R_{\Sigma,0} = R_\Sigma$ and $C_{\Sigma,0} = C_{\Sigma}$. Note that we may regard $C_{\Sigma,n}$ as an element of $\rdcat{\Lambda_n}{R_{\Sigma,n}}$ via the composition $R_{\Sigma,n}\onto \TT_{\Sigma,Q_n}\into \End_{\dcat{\Lambda_n}}(C_{\Sigma,n})$. Let $\alpha_{\Sigma,n}$ be the isomorphism 
\[
C_{\Sigma,n}\lotimes_{\Lambda_n}\mco\isomto C_{\Sigma,0}
\]
from Proposition \ref{prop: O[Delta]->T}.
Recalling that $\hh_*(k\lotimes_{\Lambda_n}C_{\Sigma,n})$ vanishes outside of $[d,d+\ell_0]$, the work of Section \ref{sec:Hecke} (Propositions \ref{prop: Verdier}, \ref{prop: twisted duality},  \ref{prop: O[Delta]->T}) in Case \ref{case:PGL2} and Section \ref{sec:Hecke wt1} (Proposition \ref{prop: localized duality-weight one}) in Case \ref{case:wt1} gives the following:

\begin{lemma}\label{lem:Patching System}
For each $\Sigma\subseteq T$ the tuple ${\perfs C}_\Sigma= \left(C_{\Sigma,n},\alpha_{\Sigma,n}\right)_{n\geqslant 0}$ is a patching system. Moreover there is an isomorphism ${\perfs C}_\Sigma^\vee\cong {\perfs C}_\Sigma$ of patching systems. \qed
\end{lemma}

Now for any $\Sigma\subseteq T$ let $M_{\Sigma,\infty} = \patch({\perfs C}_\Sigma)$. By Theorems \ref{th:ultrapatching} and \ref{th:duality} each $M_{\Sigma,\infty}$ is a self-dual maximal Cohen--Macaulay $R_{T,\infty}$-module. Also by construction the action of $R_{T,\infty}$ on $M_{\Sigma,\infty}$ factors through $R_{\Sigma,\infty}$.

Moreover (after picking an appoperiate homomorphism $S_\infty\to R_{T,\infty}$ as in Theorem \ref{th:ultrapatching}) we have $R_{\Sigma,\infty}\otimes_{S_\infty}\mco\cong R_\Sigma$ and $M_{\Sigma,\infty}\otimes_{S_\infty}\mco\cong \hh_d(C_{\Sigma})$.

Theorem \ref{R=T theorem} will thus follow from the following:

\begin{theorem}\label{thm:free summand}
There is an integer $\mu\ge 1$ such that for each $\Sigma\subseteq T$ there is an isomorphism $M_{\Sigma,\infty}\cong R_{\Sigma,\infty}^\mu\oplus W_{\Sigma,\infty}$ for some $R_{\Sigma,\infty}$-module $W_{\Sigma,\infty}$.
\end{theorem}

We need the following consequence of  Theorem \ref{th:duality},  Proposition \ref{prop: localized duality} and  Conjecture \ref{Ihara conj} (in particular, its consequence Lemma \ref{lem: Ihara U_v}) in Case \ref{case:PGL2}, and Propositions \ref{prop: localized duality-weight one} and \ref{prop: Ihara wt 1} in Case \ref{case:wt1}.

\begin{proposition}\label{prop:maps}
For each $\Sigma\subseteq T$ and each $v\in T\sm\Sigma$ the family of maps $\pi_v\colon C_{\Sigma\cup\{v\},n}\to C_{\Sigma,n}$ induces a morphism ${\perfs C}_{\Sigma\cup\{v\}}\to {\perfs C}_\Sigma$ and hence a $R_{T,\infty}$-module homomorphism $\pi_v\colon M_{\Sigma\cup\{v\},\infty}\to M_{\Sigma,\infty}$.

Moreover the map $\pi_v\colon M_{\Sigma\cup\{v\},\infty}\to M_{\Sigma,\infty}$ is surjective, and if $\pi_v^\vee\colon M_{\Sigma,\infty}\to M_{\Sigma\cup\{v\},\infty}$ is the dual map, induced by the self-duality of $M_{\Sigma,\infty}$ and  $M_{\Sigma\cup\{v\},\infty}$, then  for some units $u_v,u'_v\in R_{T,\infty}^\times$ we have 
\[
\pi_v\circ \pi_v^\vee = u_vB_v\in R_{T,\infty}\quad\text{and}\quad \pi_v^\vee\circ\pi_v = u'_vB_v\in R_{T,\infty}\,.
\]
\end{proposition}

\begin{proof}
 The surjectivity of $\pi_v$ follows from Lemma \ref{lem: Ihara U_v} and Nakayama's Lemma. Next we note that under the map 
\[
R_v^{{\rm uni}(\epsilon_v)}\into R_{T,\infty}\onto R_\Sigma^\square \onto \TT_\Sigma^\square\,,
\]
the element $B_v$ maps to $w_v(\epsilon_v \sqrt{q_v\psi(\varphi_v)^{-1}}U_v-1)$ for some unit $w_v\in (\TT_\Sigma^\square)^\times$ by Conjecture \ref{R->T conj}(3) and Theorem \ref{R->T conj-weight one}(3) which means that $w_v^{-1}B_v(2+w_v^{-1}B_v)$ maps to
\begin{align*}
(\epsilon_v \sqrt{q_v\psi(\varphi_v)^{-1}}U_v-1)(\epsilon_v \sqrt{q_v\psi(\varphi_v)^{-1}}U_v+1) 
	&= (\epsilon_v \sqrt{q_v\psi(\varphi_v)^{-1}}U_v)^2-1 \\
	&= q_v\psi(\varphi_v)^{-1}U_v^2-1 
\end{align*}
(as $\epsilon_v=\pm 1$). Recall that in Case \ref{case:PGL2} this equals $U_v^2-1$ since $\psi(\varphi_v)=q_v$. 

Note that  $\psi(\varphi_v)\in \mco^\times$ and $2+w_v^{-1}B_v\in R_{T,\infty}^\times$, as $B_v\in \fm_{R_{T,\infty}}$ and $2\in \mco^\times$.   
 Using  the formulas  for maps  $\pi_v\circ \pi_v^\dagger$ and $\pi_v^\dagger\circ\pi_v$  of the complexes  ${\perfs C}_\Sigma, {\perfs C}_{\Sigma \cup \{v\}}$ in Proposition \ref{prop: localized duality} in Case \ref{case:PGL2} and Proposition \ref{prop: localized duality-weight one}  in  Case \ref{case:wt1}, and   Theorem \ref{th:duality},   which gives that  $
\patch(\pi_v^{\dagger})= \patch(\pi_v)^{\vee}$,  we deduce  the  formulas  for  $\pi_v\circ \pi_v^\vee$ and $\pi_v^\vee\circ\pi_v$.
\end{proof}

Now pick an augmentation $\lambda\colon R_{T,\infty}\onto R_{\es,\infty}\onto \mco$ such that $\lambda(A_v) = 0$ and $\lambda(B_v)\ne 0$ for all $v\in T$. It follows that each $R_{\Sigma,\infty}$ is regular at $\fp\colonequals \Ker\lambda$, and so $(R_{\Sigma,\infty},\lambda) \in \acat(c)$. Let $\mu = \rank_\fp M_{\es,\infty}$, and note that $M_{\es,\infty}\cong R_{\es,\infty}^{\mu}$, by the Auslander--Buchsbaum formula. (Note that $M_{\es,\infty}\ne 0$ as $C_{\es}\ne 0$ by the choice of $\fm$, and so $\mu\ge 1$.)

\begin{lemma}\label{lem:constant rank}
For each $\Sigma\subseteq T$ and each $v\in T\sm\Sigma$ the map 
\[
\pi_v\colon M_{\Sigma\cup\{v\},\infty}\lra M_{\Sigma,\infty}
\]
is an isomorphism at $\fp$. In particular, $\rank_\fp M_{\Sigma,\infty} = \mu$ for all $\Sigma$.
\end{lemma}
\begin{proof}
By the choice of $\lambda$, $B_v\not\in \fp$, and so $B_v$ is a unit in $(R_{T,\infty})_\fp$. Thus by Proposition \ref{prop:maps}, the maps $\pi_v\circ \pi_v^\vee = u_vB_v$ and $\pi_v^\vee\circ\pi_v = u'_vB_v$ are both multiplication by units in $(R_{T,\infty})_\fp$, after localizing at $\fp$. The claim follows.
\end{proof}

\begin{remark}
 To prove the  lemma it suffices to  have an element $\alpha_v \in R_{\Sigma \cup \{v\},\infty}$ which annihilates $\Ker \pi_v$ and maps to an element of $R_{\es,\infty}$ that  is not in $\Ker \lambda$; the computation $\pi_v^\vee\circ\pi_v = u'_vB_v$ shows that we may take $\alpha_v=B_v$. On the other hand for  the proof of  Theorem~\ref{thm:free summand} below it is  crucial to have the precise computation  that $(\pi_v\circ \pi_v^\vee) = (B_v)$ as ideals of $R_{\Sigma,\infty}$  to  compute the  change of congruence modules $\length_\mco \Psi_{R_{\Sigma\cup\{v\},\infty}}(M_{\Sigma\cup\{v\},\infty}) 
= \length_\mco \Psi_{R_{\Sigma\cup\{v\},\infty}}(M_{\Sigma,\infty}) +\mu\cdot\nu_\mco(\lambda(B_v))$.
\end{remark}

\begin{proof}[Proof of Theorem~\ref{thm:free summand}]
For any $\Sigma\subseteq T$ and $v\in T\sm\Sigma$, as $\pi_v\circ\pi_v^\vee = u_vB_v$ we get  using Proposition \ref{prop:maps} that the composition
\begin{align*}
F_{R_{T,\infty}}^{c}(M_{\Sigma,\infty}) 
&\cong F_{R_{T,\infty}}^{c}(M_{\Sigma,\infty}^\vee)
\xrightarrow{\pi_v^\vee}F_{R_{T,\infty}}^{c}(M_{\Sigma\cup\{v\},\infty}^\vee)
\cong F_{R_{T,\infty}}^{c}(M_{\Sigma\cup\{v\},\infty})\\
&\xrightarrow{\pi_v}F_{R_{T,\infty}}^{c}(M_{\Sigma,\infty})
\end{align*}
is  multiplication by $\lambda(u_vB_v)$ and so as $F_{R_{T,\infty}}^{c}(M_{\Sigma,\infty})\cong \mco^\mu$, its determinant is $\lambda(u_v)^\mu\lambda(B_v)^\mu$.  As $u_v$ is a unit, Proposition \ref{pr:change-modules} yields 
\begin{align*}
\length_\mco \Psi_{R_{\Sigma\cup\{v\},\infty}}(M_{\Sigma\cup\{v\},\infty}) 
&= \length_\mco \Psi_{R_{\Sigma\cup\{v\},\infty}}(M_{\Sigma,\infty})+\nu_\mco(\lambda(u_v)^\mu\lambda(B_v)^\mu)\\
&= \length_\mco \Psi_{R_{\Sigma,\infty}}(M_{\Sigma,\infty})+\mu\cdot\nu_\mco(\lambda(B_v)).
\end{align*}
We note that the hypotheses of Proposition \ref{pr:change-modules}  are satisfied as Lemma  \ref{lem:constant rank} yields that  $M_{\Sigma\cup\{v\},\infty}$ and  $M_{\Sigma,\infty}$ have the same rank $\mu$ at $\lambda$ and,  as noted after Lemma \ref{lem:Patching System}, we know that both these modules are maximal Cohen-Macaulay over  $R_{\Sigma\cup\{v\},\infty}$ and self-dual.

Since $R_{\es,\infty}$ is formally smooth and hence regular, we get that $\Phi_{R_{\es,\infty}} = 0$ and
\[
\Psi_{R_{\es,\infty}}(M_{\es,\infty}) =  \Psi_{R_{\es,\infty}}(R_{\es,\infty}^\mu)= \Psi_{R_{\es,\infty}}(R_{\es,\infty})^\mu=\Phi_{R_{\es,\infty}}^\mu = 0\,.
\]
by Theorem \ref{th:wiles}. It now follows by induction that
\[
\length_\mco \Psi_{R_{\Sigma,\infty}}(M_{\Sigma,\infty}) = \mu\sum_{v\in\Sigma}\nu_\mco(\lambda(B_v))
\]
for any $\Sigma\subseteq T$.

On the other hand, by \S \ref{ssec:global def} and Proposition \ref{prop: TW primes}, one gets that $R_{\Sigma,\infty}$ is a power series ring over 
$\displaystyle\widehat{\bigotimes_{{v\in\Sigma}}}R_v^{{\rm uni}(\epsilon_v)}$. By Proposition \ref{prop:unipotent-deformations} it is thus a power series ring over 
\[
\widehat{\bigotimes_{v\in\Sigma}}\mco\pos{A_v,B_v}/(A_vB_v)
\]
An easy computation now gives that
\[
\Phi_{R_{\Sigma,\infty}}\cong \prod_{v\in \Sigma}\left(\mco/\lambda(B_v)\mco\right)
\]
and so
\[
\length_\mco \Phi_{R_{\Sigma,\infty}} = \sum_{v\in\Sigma}\nu_\mco(\lambda(B_v)).
\]
Hence $\delta_{R_{\Sigma,\infty}}(M_{\Sigma,\infty}) = 0$ for all $\Sigma\subseteq T$.

As each $R_{\Sigma,\infty}$ is a complete intersection (and hence Gorenstein), Theorem \ref{thm:free summand} now follows from Theorems \ref{th:gorenstein} and \ref{th:wiles}. 
\end{proof}

\section*{Acknowledgements}
This work is partly supported by National Science Foundation grant DMS-200985 (SBI), and  by  National Science Foundation grant  DMS-2200390  and  Simons Fellowship (CBK). It was written while the third author (JM) was at the Max Planck Institute for Mathematics in Bonn, and he thanks them for their support. We are grateful to Gebhard B\"ockle, and an anonymous referee, for detailed comments on an earlier version of this manuscript. We also thank Patrick Allen, George Boxer, Fred Diamond, Toby Gee and Jack Thorne for helpful discussions and correspondence.

\newpage

\section*{Glossary}

\begin{longtable}{lp{0.75\textwidth}}
$\Art$ & category of Artinian $\mco$-algebras, \pageref{ssec:global def} \\
$c$& the height of $\fp_A$, also known as codimension, \pageref{ch:congruence-module}\\
$\ecoh^i_A(M)$ & the torsion-free quotient of $\Ext^i_A(\mco,M)$, \pageref{pg:ecoh}\\
$\cmod A(M)$& the congruence module attached to an $A$-module $M$, \pageref{pg:cmod}\\
$\CNLO$ & category of complete local noetherian $\mco$-algebras, \pageref{ssec:global def} \\
$\acat$ & category of pairs $(A,\lambda)$ with $\lambda\colon A\to\mco$ a surjective map in $\CNLO$ such that $A$ is regular at $\Ker(\lambda)$, \pageref{pg:acat}\\
$\acat(c)$ & pairs $(A,\lambda)$ in $\acat$ with $\height \Ker(\lambda)=c$, \pageref{pg:acat}\\
$\con A$ & the torsion part of the cotangent module of $\lambda\colon A\to \mco$, \pageref{eq:conormal}\\
$\delta_A(M)$ & the Wiles defect of an $A$-module $M$, \pageref{eq:defect} \\
$\depth_AM$ & the depth of an $A$-module $M$, \pageref{se:depth}\\
$\dcat A$ &the derived category of $\rmod A$, \pageref{pg:ecoh}\\
$\dbcat A$ &the bounded derived category of $\rmod A$, \pageref{pg:ecoh}\\
$\rdcat {\Lambda}R$ & a category of $\Lambda$-complexes with derived $R$-action, \pageref{ch:derived-action} \\
$\mcd$ & deformation conditions, \pageref{ssec:global def} \\
$\varepsilon_\ell$ &the cyclotomic character, \pageref{se:Galois}\\
$\grade(I,M)$& the length of the longest $M$-regular sequence in $I$, \pageref{se:depth}\\
$\lambda_A(M)$ &the map $M\to M\otimes_A\mco$, \pageref{pg:cmod}\\
$\rmod A$ &the category of finitely generated $A$-modules, \pageref{pg:ecoh} \\
$I_{F_v}$ &the inertia subgroup of prime $v$, \pageref{ssec:local def} \\
$P_{F_v}$ &the wild inertia subgroup of a prime $v$, \pageref{ssec:local def}\\
$\Nm(v)$ & norm of a prime $v$ in a number field, \pageref{ssec:local def}\\
$N(\rhobar)$ & Artin conductor of $\rhobar$, \pageref{R=T theorem} \\
$\vf(v)$ & lift of the Frobenius, \pageref{ssec:local def} \\
$\sigma_v$ & a lift of the topological generator of $I_{F_v}/P_{F_v}$, \pageref{ssec:local def}\\
$\mco$ & a discrete valuation ring, typically complete, \pageref{se:congruence module} \\
$\PatchSys$ &a category of patching systems, \pageref{ch:pasy}\\
$\eta_M$ & torsion-free quotient of a derived trace map, \pageref{sse:defect}\\
$\fp^{(2)}$ &the second symbolic power of a prime ideal $\fp$, \pageref{pg:symbolic}\\
$\tors(U)$ & the torsion submodule of an $\mco$-module $U$, \pageref{ch:tors-remark}\\
$\tfree U$ & the torsion-free quotient of an $\mco$-module $U$, \pageref{ch:tors-remark}\\
$\omega_A$ &the dualizing complex of $A$, suitably normalized, \pageref{sse:dualizing}\\
$\omega_A(M)$ &the dual with respect of $\omega_A$ of a complex $M$, \pageref{sse:dualizing}\\
${\perfs C}^{\dagger}$ &the dual of a patching system $\perfs C$, \pageref{ch:patch-duality}\\
$M^\vee$ & the highest degree homology module of $\omega_A(M)$, \pageref{eq:dual-cm}\\
$M[I]$ &the submodule of $M$ annihilated by the ideal $I$, \pageref{se:intro}\\
$[-]$ &the suspension functor in any triangulated category, \pageref{pg:ecoh} \\
$\Sigma$ & level raising primes, \pageref{ssec:global def} \\
$Q$ & sets of  Taylor--Wiles primes, \pageref{ssec:global def} \\
$R_\Sigma$ & global deformation ring with ramification at  places in $S$ and $\Sigma$, \pageref{ssec:global def}\\
$R_v^\square$ & framed deformation ring at $v$, \pageref{ssec:local def}\\
$R^{\fl}_v$ &  finite flat deformation ring at $v|\ell$, \pageref{ssec:local def}\\
$R_v^{\ur}$ & quotient of $R_v^\square$ parameterizing unramified lifts of $\rhobar_v$, \pageref{ssec:local def}\\
$R_{\Sigma}^{\disc}$ &global Galois deformation ring, \pageref{ssec:global def}\\
$R_{\loc,\Sigma}$ & completed tensor product of local deformation rings, \pageref{R_loc}\\
$R_{\infty,\Sigma}$ & power series ring over $R_{\loc,\Sigma}$, \pageref{prop: TW primes}\\
$\tau_M(R) $ &the trace map on an $R$-module $M$, \pageref{le:summands}\\
$Y_K$ & congruence manifold of level $K$, \pageref{sse:manifolds}\\
$C_K$ & complex of singular chains on $Y_K$, \pageref{sse:manifolds}\\
$[K_2\alpha K_1]$ & double coset operator, \pageref{sse:double coset}\\
$T_v,U_v,\langle d\rangle_v$ & Hecke operators, \pageref{sse:Hecke}\\
$K_0(\mcn), K_1(\mcn), K_\Delta(\mcn)$ & congruence subgroups of level $\mcn t^2$, \pageref{sse:K_0(N)}\\
$K^\disc_\Delta(N), K^\disc_H(N)$ &  subgroups intermediate  between $K^\disc_1(N)$ and $K^\disc_0(N)$, \pageref{pg:intermediate-subgroups},\\
$Y_0(\mcn),C_0(\mcn)$ & congruence manifolds and chain complexes at level $K_0(\mcn)$, \pageref{sse:K_0(N)}\\
$Y_{0,\Delta}(\mcn,Q),C_{0,\Delta}(\mcn,Q)$ & congruence manifolds and chain complexes at level $K_0(\mcn)\cap K_\Delta(Q)$, \pageref{sse:K_0(N)}\\
$\TT^S(K)$ & Hecke algebra at level $K$, \pageref{sse:Hecke}\\
$\rhobar_\fm$ & residual Galois representation corresponding to $\fm\subseteq \TT^S(K)$, \pageref{rhobar conj}\\
$\pi_{K,v}$ & level lowering map, \pageref{ssec: level lowering}\\
$\pi_v$ & localized level lowering map, \pageref{sse:Sigma}\\
$\TT_\Sigma,\TT_{\Sigma,Q}$ & localizations of Hecke algebras, \pageref{sse:Sigma}\\
$K_\Sigma,K_{\Sigma,Q}$ & shorthands for $K_0(\mcn_\Sigma)$ and $K_0(\mcn_\Sigma)\cap K_\Delta(Q)$, \pageref{sse:Sigma}\\
$\fullT^S(K_\Sigma),\fullT^S(K_{\Sigma,Q})$ & full Hecke algebras, \pageref{sse:Sigma}\\
$C_{\Sigma},C_{\Sigma,Q}$ & localizations of $C_{K_\Sigma}$ and $C_{K_{\Sigma,Q}}$, \pageref{sse:Sigma}\\
$X^\disc_K$ &Shimura curve, \pageref{sec:Hecke wt1} \\
$\hh^i(X^\disc_K, \mathcal{F})$ &coherent cohomology of $X^{\disc}_K$ with values in a sheaf $\mathcal F$, \pageref{shimura-weight one}\\
${\rm Ha}$ &Hasse invariant attached to a Shimura curve, \pageref{shimura-weight one}\\
\end{longtable}

\newpage

\bibliographystyle{amsplain}

\end{document}